\documentclass[12pt]{amsart}

\usepackage{geometry}
\usepackage[ansinew]{inputenc}
\usepackage{amssymb}
\usepackage{amsmath}
\usepackage{amsthm}
\usepackage{bm}
\usepackage{bbm}
\usepackage{hyperref}           
\usepackage{mathrsfs}
\usepackage{enumerate}
\usepackage{titletoc}

\usepackage{comment}

\usepackage[numbers, sort&compress]{natbib}
\input{math.sty}

\begin{document}

\title[BvM for Compound Poisson processes]{Bernstein - von Mises theorems for statistical inverse problems II:  Compound Poisson processes}\thanks{We would like to thank Kolyan Ray for helpful comments leading to improvements of the paper. We are grateful to the associate editor and an anonymous referee for valuable remarks on the manuscript. RN acknowledges support by the European Research Council (ERC) under grant agreement No.647812.}

\author{Richard Nickl}

\address{Statistical Laboratory, Department of Pure Mathematics and Mathematical Statistics, University of Cambridge, CB3 0WB, Cambridge, UK}
\email{r.nickl@statslab.cam.ac.uk}

\author{Jakob S\"ohl}
\address{Delft Institute of Applied Mathematics, Faculty of Electrical Engineering, Mathematics and Computer Science, TU Delft, Van Mourik Broekmanweg 6, 2628 XE, Delft, The Netherlands}
\email{j.soehl@tudelft.nl}

\date{\today}


\begin{abstract}
We study nonparametric Bayesian statistical inference for the parameters governing a pure jump process of the form $$Y_t = \sum_{k=1}^{N(t)} Z_k,~~~ t \ge 0,$$ where $N(t)$ is a standard Poisson process of intensity $\lambda$, and $Z_k$ are drawn i.i.d.~from jump measure $\mu$.  A high-dimensional wavelet series prior for the L\'evy measure $\nu = \lambda \mu$ is devised and the posterior distribution arises from observing discrete samples $Y_\Delta, Y_{2\Delta}, \dots, Y_{n\Delta}$ at fixed observation distance $\Delta$, giving rise to a nonlinear inverse inference problem. We derive contraction rates in uniform norm for the posterior distribution around the true L\'evy density that are optimal up to logarithmic factors  over H\"older classes, as sample size $n$ increases. We prove a functional Bernstein--von Mises theorem for the distribution functions of both $\mu$ and $\nu$, as well as for the intensity $\lambda$, establishing the fact that the posterior distribution is approximated by an infinite-dimensional Gaussian measure whose covariance structure is shown to attain the information lower bound for this inverse problem. As a consequence posterior based inferences, such as nonparametric credible sets, are asymptotically valid and optimal from a frequentist point of view.

\medskip


\smallskip

\end{abstract}

\maketitle

 \tableofcontents

\section{Introduction}

While the Bayesian approach to inverse problems is widely used in scientific and statistical practice, very little theory is available that explains why Bayesian algorithms should be trusted to provide objective solutions of inverse problems in the presence of statistical noise, particularly in infinite-dimensional, non-linear cases which naturally arise in applications, see \cite{S10, DS16}. In the recent contributions \cite{NS17, N17, MNP17} proof techniques were developed that can be used to derive theoretical guarantees for posterior-based inference, based on suitably chosen priors, in various settings, including inverse problems arising with diffusion processes, $X$-ray tomography or elliptic partial differential equations. A main idea of \cite{N17, MNP17} is that a careful analysis of the `Fisher information operator' inducing the statistical observation scheme combined with tools from Bayesian nonparametrics \cite{CN13, CN14} can be used to derive sharp results about the frequentist behaviour of posterior distributions in general inverse problems.

The analysis of the `information operator' depends highly on the particular problem at hand, and in the present article we continue this line of investigation in a statistical inverse problem very different from the ones considered in \cite{NS17, N17, MNP17}, namely in the problem of recovering parameters of a stochastic jump process from discrete observations. Statistically speaking, the inverse problem is a `missing observations' problem that arises from the fact that we do not observe all the jumps and need to `decompound' the effect of possibly seeing an accumulation of jumps without knowing how many have occurred. This has been studied from a non-Bayesian perspective for certain classes of L\'evy processes by several authors, we mention here the seminal papers \cite{BG03, BR06, EGS07, NR09} -- see also~\cite{Ret15} for various further references -- and \cite{NR12, Trabs2015, NRST16, Co17} relevant for the results obtained in the present paper. A typical estimation method used in several of these articles is based on spectral regularisation techniques built around the fact that the L\'evy measure identifying all parameters of the jump process can be expressed in the Fourier domain by the L\'evy-Khintchine formula (see (\ref{LK}) below). 

Given the sophistication of the non-linear estimators proposed so far in the `decompounding problem' just described, one may wonder if a `principled' Bayesian approach that just places a standard high-dimensional random series prior on the unknown L\'evy measure can at all return valid posterior inferences, for example in the sense of frequentist's coverage of credible sets, in such a measurement scheme. In the present article we provide some answers to this question in the prototypical setting where one observes discrete increments of a compound Poisson processes at fixed observation distance $\Delta>0$. To lift some of the technicalities occurring in the proofs we restrict ourselves to periodic and hence compactly supported processes, and -- to avoid identifiability problems arising in the periodic case -- to small enough $\Delta$. We show that the posterior distribution optimally recovers all parameters of the jump process, both in terms of convergence rates for the L\'evy density $\nu$ and in terms of efficient inference for the intensity of the Poisson process and the distribution function of the jump measure $\mu$. For the latter we obtain functional Bernstein--von Mises theorems which are the Bayesian analogues of the `Donsker-type' central limit theorems obtained in \cite{NR12}, \cite{Co17} for frequentist regularisation estimators.  Just as in \cite{N17}, our proofs are inspired by techniques put forward in \cite{CN13, CN14, C14, CR15, C17} in `direct' problems. However, due to the different structure of the jump process model, our proofs need to depart from those in \cite{N17} in various ways, perhaps most notably since we have to consider a prior with a larger support ellipsoid, and hence need to prove initial contraction rates for our posterior distribution by quite different methods than is commonly done, see Section \ref{prelimc}. The inversion of the information operator in the jump process setting also poses some surprising subtleties that nicely reveal finer properties of the inference problem at hand -- our explicit construction of the inverse information operator in Section \ref{opsec} also gives new, more direct proofs of the semi-parametric lower bounds obtained in \cite{Trabs2015} (whose lower bounds admittedly hold in a more general setting than ours). Finally we should mention that substantial work -- using tools from empirical process theory -- is required in our setting when linearising the likelihood function to obtain quantitative LAN-expansions since, in contrast to \cite{N17}, our observation scheme is far from Gaussian. In this sense the techniques we develop here are relevant also beyond compound Poisson processes, although, as argued above, the theory for non-linear inverse problems is largely constrained by any specific case one is studying.

The paper is structured as follows: In Section~\ref{sec_model} we give basic definitions and describe the model and  prior. In Section~\ref{sec_results} we state the contraction rates in supremum norm, the Cram\'er--Rao lower bound as well as the Bernstein--von Mises theorems in multi-scale spaces and for functionals of the L\'evy measure. Section~\ref{sec_proofs} contains the proof of the contraction rates and of the multi-scale Bernstein--von Mises theorem. Sections~\ref{prelimc}-\ref{sec:tedious} contain the remaining proofs.

\section{Model and prior}\label{sec_model}

\subsection{Basic definitions}

Let $(N(t): t \ge 0)$ be a standard Poisson process of intensity $\lambda>0$. Let $\mu$ be a probability measure on $(-1/2,1/2]$ such that $\mu(\{0\})=0$, and let $Z_1, Z_2, \dots$ be an i.i.d.~sequence of random variables drawn from $\mu$. In what follows we view $I=(-1/2,1/2]$ as a compact group under addition modulo $1$. Then the (periodic) compound Poisson process taking values in $(-1/2,1/2]$ is defined as
\begin{equation}
Y_t = \sum_{k=1}^{N(t)} Z_k,~~~ t \ge 0,
\end{equation}
where $Y_0=0$ almost surely, by convention. The process $(Y_t: t \ge 0)$ is a pure jump L\'evy process on $I=(-1/2,1/2]$ with L\'evy measure $d\nu = \lambda d\mu$. We observe this process at fixed observation distance $\Delta$, namely $Y_\Delta, Y_{2\Delta}, \dots, Y_{n \Delta}$, and define the increments of the process 
\begin{equation}\label{increm}
X_1 = Y_\Delta, X_2 = Y_{2\Delta} - Y_\Delta, \dots, X_n = Y_{n\Delta} - Y_{(n-1)\Delta}.
\end{equation}
The $X_k$'s are i.i.d.~random variables drawn from the infinitely divisible distribution $\mathbb P_\nu = \mathbb P_{\nu, \Delta}$ which has characteristic function (Fourier transform)
\begin{equation}\label{LK}
\phi_{\nu} (k) = \mathcal F \mathbb P_\nu (k) = \exp \left(\Delta \int_I ( e^{2\pi i kx} -1) d\nu \right), ~ k \in \mathbb Z,
\end{equation}
e.g., by the L\'evy--Khintchine formula for L\'evy processes in compact groups (Chapter IV.4 in \cite{P67}). Obviously $(\phi_\nu(k) : k \in \mathbb Z)$ identifies $\mathbb P_\nu$ but under the hypotheses we will employ below it will also identify $\nu$ and thus the law of the jump process $(Y_t:t\ge 0)$. The inverse problem is to recover $\nu$ from i.i.d.~samples drawn from the probability measure $\mathbb P_\nu$.

We denote by $C(I)$ the space of bounded continuous functions on~$I$ equipped with the uniform norm $\|\cdot\|_\infty$, and let $M(I)=C(I)^*$ denote the (dual) space of finite signed (Borel) measures on~$I$. For $\kappa_1, \kappa_2 \in M(I)$ their convolution is defined by
$$\kappa_1 \ast \kappa_2 (g) = \int_I \int_I g(x+y) d\kappa_1(x) d\kappa_2(y),~~ g \in C(I),$$ and the last identity holds in fact for arbitrary $g\in L^\infty(I)$ by approximation, see Proposition~8.48 in~\cite{Fo99}.
This coincides with the usual definition of convolution of functions when the measures involved have densities with respect to the Lebesgue measure. We shall freely use standard properties of convolution integrals, see, e.g., Section 8.2 in \cite{Fo99}. 

An equivalent representation of $\mathbb P_\nu$ is by the infinite convolution series
\begin{equation} \label{convsum}
\mathbb P_\nu = e^{-\Delta \nu(I)} \sum_{k=0}^\infty \frac{\Delta^k \nu^{\ast k}}{k!}
\end{equation}
where $\nu^0=\delta_0, \nu^{\ast 1} =\nu, \nu^{\ast 2} = \nu \ast \nu$ and $\nu^{*k}$ is the $k-1$-fold convolution of $\nu$ with itself. [To see this just check the obvious fact that the Fourier transform of the last representation coincides with $\phi_{\nu}$ in (\ref{LK}), and use injectivity of the Fourier transform.]

We will denote by $\mathbb P_\nu^\mathbb N$ the infinite product measures describing the laws of infinite sequences of i.i.d.~samples (\ref{increm}) arising from a compound Poisson process with L\'evy measure $\nu$, and~$\mathbb E_\nu$ will denote the corresponding expectation operator. We denote by $L^p=L^p(I), 1 \le p<\infty,$ the standard spaces of functions $f$ for which $|f|^p$ is Lebesgue-integrable on~$I$, whereas, in slight abuse of notation, for a finite measure $\kappa$ we will denote by $L^p(\kappa), 1\le p \le \infty,$ the corresponding spaces of $\kappa$-integrable functions on~$I$, predominantly for the choices $\kappa = \nu, \kappa = \mathbb P_\nu$. The spaces $L^2(I), L^2(\kappa)$ are Hilbert spaces equipped with natural inner products  $\langle \cdot, \cdot \rangle, \langle \cdot, \cdot\rangle_{L^2(\kappa)}$, respectively. The symbol $L^\infty(I)$ denotes the usual space of bounded measurable functions on~$I$ normed by $\|\cdot\|_\infty$. We also write $\lesssim, \approx$ for (in-)equalities that hold up to fixed multiplicative constants, and employ the usual $o_P,O_P$-notation to indicate stochastic orders of magnitude of sequences of random variables.

\subsection{Likelihood, prior and posterior}

We study here the problem of conducting nonparametric Bayesian inference on the parameters $\nu, \mu, \lambda$, assuming a regularity constraint $\nu \in C^s(I), s>0$, where $C^s$ is the usual H\"older space over $I$ normed by $\|\cdot\|_{C^s}$ (when $s \in \mathbb N$ these are the ordinary spaces of $s$-times continuously differentiable functions, e.g., Section~2.2.2 in \cite{T83}). To define the likelihood function we need a common dominating measure for the statistical model $(\mathbb P_\nu: \nu \in \mathcal V)$ where $\mathcal V$ is some family of L\'evy measures possessing densities with respect to Lebesgue measure $\Lambda$ with density $\Lambda=1_{(-1/2,1/2]}$. Since $\Lambda$ is idempotent -- $\Lambda \ast \Lambda = \int_I \Lambda(\cdot-y) \Lambda(y)dy = \Lambda$~--  we can consider the resulting compound Poisson measure $\mathbb P_\Lambda = e^{-\Delta} \delta_0 + (1-e^{-\Delta}) \Lambda$ as a fixed reference measure on~$I$. Then for any absolutely continuous $\nu$ on~$I$ the densities $p_\nu$ of $\mathbb P_\nu$ with respect to $\mathbb P_\Lambda$ exist. The likelihood function of the observations $X_1, \dots, X_n$ is defined as 
\begin{equation}\label{likelihood}
L_n(\nu) = \prod_{i=1}^n p_\nu(X_i),~~~ \nu \in \mathcal V.
\end{equation}
We also write $\ell_n(\nu) = \log L_n(\nu)$ for the log-likelihood function. Next, if $\Pi$ is a prior distribution on a $\sigma$-field $\mathcal S_\mathcal V$ of $\mathcal V$ such that the map $(\nu,x) \mapsto p_\nu(x)$ is jointly measurable, then standard arguments 
imply that the resulting posterior distribution given observations $X_1, \dots, X_n$ is
\begin{equation} \label{post}
\Pi(B|X_1, \dots, X_n)  = \frac{\int_B L_n(\nu) d\Pi(\nu)}{\int_\mathcal V L_n(\nu) d\Pi(\nu)}.
\end{equation}
We shall model an $s$-regular function by a high-dimensional product prior expressed through a wavelet basis: Let 
\begin{equation}\label{wavbasis}
\left\{\psi_{lk}: k =0, \dots, (2^l \vee 1)-1,\: l=-1, \dots, J-1 \right\}, J \in \mathbb N,
\end{equation}
form a periodised Daubechies' type wavelet basis of $L^2=L^2(I)$, orthogonal for the usual $L^2$-inner product $\langle \cdot, \cdot \rangle$ (described in Section 4.3.4 in \cite{GineNickl2016}; where the constant `scaling function' is written as the first element $\psi_{-1,0}\equiv 1$, in slight abuse of notation). Basic localisation and approximation properties of this basis are, for any $g \in C^s(I)$ and $j \in \mathbb N$,
\begin{align}
&\sup_{x \in I}\sum_{k}|\psi_{jk}(x)| \lesssim 2^{j/2},~~|\langle g, \psi_{jk}\rangle| \lesssim \|g\|_{C^s} 2^{-j(s+1/2)},\nonumber\\
&\|P_{V_j}(g) -g\|_{L^2(I)} \lesssim \|g\|_{C^s} 2^{-js},\label{wavprop}
\end{align}
where $P_{V_j}$ is the usual $L^2$-projector onto the linear span $V_j$ of the $\psi_{lk}$'s with $l \le j-1$.

Now consider the random function
\begin{equation} \label{prior0}
v = \sum_{l \le J-1} \sum_k a_l u_{lk} \psi_{lk}(\cdot),~a_l=2^{-l} (l^2+1)^{-1},~~ J \in \mathbb N,
\end{equation}
where $u_{lk}$ are i.i.d.~uniform $U(-B,B)$ random variables, and $B$ is a fixed constant. The support of this prior is isomorphic to the hyper-ellipsoid $$V_{B,J} :=\prod_{l=-1}^{J-1}(-Ba_l, Ba_l)^{2^{l}\vee1} \subset \mathbb R^{2^J}$$ of wavelet coefficients. To model an $s$-regular L\'evy measure $\nu$ we define the random function 
\begin{equation} \label{prior}
\nu=e^v, ~~\Pi = \Pi_J = \text{ the law } \mathcal L(\nu) \text{ of } \nu \text{ in } V_{B,J}
\end{equation}
and shall choose $J=J_n$ such that $2^J$ grows as a function of $n$ approximately as
\begin{equation}\label{jprior}
2^J \approx n^{\frac{1}{2s+1}}.
\end{equation}
We note that the weights $a_l=2^{-l}(l^2+1)^{-1}$ ensure that the random function $v$ has some minimal regularity, in particular is contained in a bounded subset of~$C(I)$.

Throughout we shall work under the following assumption on the L\'evy measure and on the prior identifying the law of the compound Poisson process generating the data.

\begin{assumption}\label{overall}
Assume the true L\'evy measure $\nu_0$ has a Lebesgue density, still denoted by~$\nu_0$, which is contained in $C^s(I)$ for some $s>5/2$, that $\nu_0$ is bounded away from zero on~$I$, and that for $v_0 = \log \nu_0$ and some $\gamma>0$,
\begin{equation}\label{intpt}
|\langle v_0, \psi_{lk} \rangle| \le (B-\gamma) a_l~~\forall l,k,
\end{equation}
where $a_l$ was defined in (\ref{prior0}). Assume moreover that $B, \Delta$ are such that $\lambda = \int_I \nu < \pi/\Delta$ for all $\nu$ in the support of the prior.
\end{assumption}

The assumption $s>5/2$ (in place of, say, $s>1/2$) may be an artefact of our proof methods (which localise the likelihood function by an initially suboptimal contraction rate) but, in absence of a general  `Hellinger-distance' testing theory (cf. Appendix~D in~\cite{GvdV17} or Section~7.1 in~\cite{GineNickl2016}) for the inverse problem considered here, appears unavoidable. 

The assumption (\ref{intpt}) with $\gamma>0$ guarantees that the true L\'evy density is an `interior' point of the parameter space $V_{B,J}$ for all $J$ -- a standard requirement if one wishes to obtain Gaussian asymptotics for posterior distributions. 

Finally, the bound on $\lambda$ ensures identifiability of $\nu$, and thus of the law of the compound Poisson process, from the measure $\mathbb P_\nu$ generating the observations. That such an upper bound is necessary is a consequence of the fact that we are considering the periodic setting, see the discussion after Assumption \ref{id} below. For the present parameter space $V_{B,J}$, Assumption~\ref{overall} enforces a fixed upper bound on $\Delta$ -- alternatively for a given value of $\Delta$ we could also renormalise $\nu$ by a large enough constant to make the intensities $\lambda$ small enough, but we avoid this for conciseness of exposition.

\section{Main results}\label{sec_results}

\subsection{Supremum norm contraction rates}

Even though the standard `Hellinger-distance' testing theory to obtain contraction rates is not directly viable in our setting, following ideas in \cite{C14} we can use the Bernstein--von Mises techniques underlying the main theorems of this paper to obtain (near-) optimal contraction rates for the L\'evy density $\nu_0$ in supremum norm loss. The idea is basically to represent the norm by a maximum over suitable collections of linear functionals, and to then treat each functional individually by semi-parametric methods. It can be shown that the minimax rate of estimation for L\'evy densities in $C^s(I)$ with respect to the supremum loss is $(\log n/n)^{s/(2s+1)}$, see \cite{Coca18} for a discussion. The following theorem achieves this rate up to the power of the log-factor. 

\begin{theorem}\label{suprat}
Suppose that $X_1, \dots, X_n$ are generated from (\ref{increm}) and grant Assumption~\ref{overall}. Let $\Pi(\cdot|X_1,\dots, X_n)$ be the posterior distribution arising from prior $\Pi=\Pi_J$ in (\ref{prior}) with~$J$ as in (\ref{jprior}). Then for every $\kappa>3$ we have as $n \to \infty$ that $$\Pi\left(\nu: \|\nu- \nu_0\|_\infty > n^{-s/(2s+1)} \log^\kappa n |X_1, \dots, X_n\right) \to^{\mathbb P_{\nu_0}^\mathbb N} 0.$$
\end{theorem}

Unlike in the standard i.i.d.~setting in \cite{C14}, we cannot rely on an initial optimal contraction rate in Hellinger distance for $\nu$, which introduces new difficulties when dealing with `semi-parametric bias terms'. Our proofs (via Lemma \ref{approxi} below) overcome these problems at the expense of an additional $\log^\kappa n$-factor.

The only comparable posterior contraction rate result of this kind we are aware of in the literature can be found in \cite{GMS15}, who obtain contraction rates for the Hellinger distance $h(\mathbb P_{\nu}, \mathbb P_{\nu_0})$ between the infinitely divisible distributions $\mathbb P_\nu, \mathbb P_{\nu_0}$ induced by the L\'evy measures $\nu, \nu_0$. Without any sharp `stability estimates' that would allow to derive optimal bounds on the distance $\|\nu-\nu_0\|_\infty$, or even just on $\|\nu-\nu_0\|_{L^2}$, in terms of $h(\mathbb P_{\nu}, \mathbb P_{\nu_0})$, the results in~\cite{GMS15} do a fortiori not imply any guarantees for Bayesian inference on the statistically relevant parameters $\nu, \mu, \lambda$.

The above contraction rate result shows that the Bayesian method works in principle and that estimators that converge with the minimax optimal rate up to log-factors can be derived from the posterior distribution, see~\cite{GhosalGhoshvanderVaart2000}.

\subsection{Information geometry of the jump process model}\label{opsec}

\subsubsection{LAN-expansion of the log-likelihood ratio process}

In order to formulate, and prove, Bernstein--von Mises type theorems, and to derive a notion of semi-parametric optimality of the limit distributions that will occur, we now obtain, for $L_n$ the likelihood function defined in (\ref{likelihood}), the LAN-expansion of the log-likelihood ratio process $$\ell_n(\nu_{h,n})-\ell_n(\nu)= \log \frac{L_n(\nu_{h,n})}{L_n(\nu)},~~ n \in \mathbb N,$$ of the observation scheme considered here, in perturbation directions $\nu_{h,n}$ that are additive on the log-scale. This will induce the score operator for the model and allow us to derive the inverse Fisher information (Cram\'er--Rao lower bound) for a large class of semi-parametric subproblems. Some ideas of what follows are implicit in the work by Trabs (2015), although we need a finer analysis for our results, including inversion of the score operator itself.

\begin{proposition}[LAN expansion] \label{lanprop}
Let $\nu=e^v$ be a L\'evy density that is bounded and bounded away from zero, and for $h \in L^\infty(I)$ consider a perturbation $\nu_{h,n} = e^{v+h/\sqrt n}$. Then if $X_i \sim^{i.i.d.}\mathbb P_\nu$ we have
\begin{equation}\label{lanee}
\ell_n(\nu_{h,n})-\ell_n(\nu) = \frac{1}{\sqrt n} \sum_{i=1}^n A_{\nu}(h)(X_i) - \frac{1}{2}\|A_\nu(h)\|_{L^2(\mathbb P_\nu)}^2 +o_{\mathbb P^{\mathbb N}_{\nu}}(1),
\end{equation}
where the score operator is given by the Radon--Nikodym density
\begin{equation}\label{oneform}
A_\nu(h) \equiv \Delta \frac{ d(h\nu - \int_I hd\nu \cdot \delta_0 ) \ast \mathbb P_\nu}{ d\mathbb P_\nu}.
\end{equation}
The operator $A_\nu$ defines a continuous linear map from $L^2(\nu)$ into \(L^2_0(\mathbb P_\nu) := \left\{g \in L^2(\mathbb P_\nu):\right.\)  $\left.\int_I g d\mathbb P_\nu = 0 \right\}.$
\end{proposition}

The proposition is proved in Section~\ref{lansec}. 

\smallskip

In the remainder of this section we study properties of $A_\nu$ and of its adjoint $A_\nu^*$, in particular we construct certain inverse mappings. Due to the presence of the Dirac measure in (\ref{oneform}) some care has to be exercised when identifying the natural domain of the inverse of the `information' operator $A_\nu^*A_\nu$.  In particular we can invert $A_\nu^* A_\nu$ only along directions~$\psi$  for which $\psi(0)=0$. An intuitive explanation is that the axiomatic property $\nu(\{0\})=0$ is required for $\nu$ to identify the law of the compound Poisson process (otherwise `no jumps' and `jumps of size zero' are indistinguishable), and as a consequence when making inference on the functional $\int_I \psi d\nu$ one should a priori restrict to $\int_I \psi 1_{\{0\}^c} d\nu$, a fact that features in the Cram\'er--Rao information lower bound (\ref{crlb}) to be established below.

\subsubsection{Derivation of the (right-)inverse of the score operator}

To proceed we will set $\Delta=1$ without loss of generality for the moment. If $\kappa \in M(I)$ is a finite signed measure on~$I$ and $g: I \to \mathbb R$ a function such that $\int_I |g|d|\kappa|<\infty$, we use the notation $g\kappa$ for the element of $M(I)$ given by $(g \kappa)(A)=\int_A g d\kappa$, $A$ a Borel subset of~$I$. Then, for a fixed L\'evy density $\nu \in L^\infty(I)$, consider the operator
\begin{equation} \label{anu}
h \mapsto A_\nu(h) := \frac{d [(\nu h) \ast \mathbb P_\nu]}{d \mathbb P_\nu} (x) - \left[\int_I d (\nu h)  \right], ~~x \in I,
\end{equation}
defined on the subset of $M(I)$ given by $$\mathcal D \equiv \{\kappa= \kappa_a+c \delta_0, ~\kappa_a \in M(I) \text{ has Lebesgue-density }h_a \in L^2(\nu); c \in \mathbb R\}.$$ This operator serves as an extension of $A_\nu$ from (\ref{oneform}) to the larger domain $\mathcal D$. It still takes values in $L^2_0(\mathbb P_\nu)$; in fact $\delta_0$ is in the kernel of $A_\nu$ since 
\begin{equation}\label{akernel}
A_\nu(\delta_0) = \frac{\nu(0) d\mathbb P_\nu}{d\mathbb P_\nu} - \int_I \nu(x)d\delta_0(x) = \nu(0)-\nu(0)=0,
\end{equation} 
but extending $A_\nu$ formally to $\mathcal D$ is convenient since the inverse of $A_\nu$ to be constructed next will take values in $\mathcal D$. Define 
\begin{equation} \label{deconv}
\pi_\nu = e^{\nu(I)} \sum_{m=0}^\infty \frac{(-1)^m \nu^{\ast m}}{m!},
\end{equation}
a finite signed measure for which $\mathbb P_\nu \ast \pi_\nu = \delta_0$ (by checking Fourier transforms). Formally, up to a constant, $\pi_\nu$ equals the inverse Fourier transform $\mathcal F^{-1} (1/\phi_{\nu})$ of $1/\phi_\nu$, and convolution with $\pi_\nu$ can be thought of as a `deconvolution operation'.
\begin{lemma}\label{scoinv}
Assume the L\'evy density $\nu \in L^\infty(I)$ is bounded away from zero on~$I$. The operator $A_\nu: \mathcal D \to L_0^2(\mathbb P_\nu)$ from (\ref{anu}) has inverse
\begin{equation}
\tilde A_\nu: L^2_0(\mathbb P_\nu) \to  \mathcal D, ~~\tilde A_\nu(g) := \frac{1}{\nu(\cdot)} \pi_\nu \ast (g \mathbb P_\nu) (\cdot),
\end{equation} in the sense that $A_\nu \tilde A_\nu = Id$ on $L^2_0(\mathbb P_\nu)$.
\end{lemma}
\begin{proof}
For any $g \in L^2_0(\mathbb P_\nu)$, by the Cauchy--Schwarz inequality, $g \mathbb P_\nu$ defines a finite signed measure, so that $\tilde A_\nu$ is well-defined and takes values in $M(I)$. Since $\mathbb P_\nu \ast \pi_\nu = \delta_0$ the Radon--Nikodym theorem (Theorem 5.5.4 in \cite{D02}) implies $$ \frac{d\left[\mathbb P_\nu \ast \pi_\nu \ast (g \mathbb P_\nu) \right]}{d \mathbb P_\nu} =  \frac{d(g \mathbb P_\nu)}{d \mathbb P_\nu} = g,~\quad \quad \quad~~~~~\mathbb P_\nu \text{ a.s.}.$$ We then have
\begin{equation}
A_\nu(\tilde A_\nu(g)) = \frac{d\left[\mathbb P_\nu \ast \pi_\nu \ast (g \mathbb P_\nu) \right]}{d \mathbb P_\nu} - \int_I d[\pi_\nu \ast (g \mathbb P_\nu)] = g,
\end{equation}
where the second term vanishes since for such $g$, by the definition of convolution,
$$\int_I d[\pi_\nu \ast (g \mathbb P_\nu)] = \int_I g d\mathbb P_\nu \int_I d \pi_\nu =0.$$ That $\tilde A_\nu$ takes values in $\mathcal D$ is immediate from the definition of $\pi_\nu$ and (\ref{convsum}). 
\end{proof}

\subsubsection{The adjoint score operator}

We now calculate the adjoint operator of $A_\nu$. 
\begin{lemma}\label{adji}
Assume the L\'evy density $\nu \in L^\infty(I)$ is bounded away from zero on~$I$. If we regard $A_\nu$ from (\ref{oneform}) as an operator mapping the Hilbert spaces $L^2(\nu)$ into $L_0^2(\mathbb P_\nu)$ then its adjoint $A_\nu^*: L^2_0(\mathbb P_\nu) \to L^2(\nu)$ is given by $A^*_\nu (w) = \Delta \mathbb P_\nu (-\cdot) \ast w$.
\end{lemma}
\begin{proof}We set without loss of generality~$\Delta=1$. Let $h \in L^2(\nu)$ and $w \in C(I) \subset L^2(\mathbb P_\nu)$ such that $\int w d\mathbb P_\nu =0$. Then by Fubini's theorem
\begin{align*}
 &\langle A_\nu(h), w \rangle_{L^2(\mathbb P_{\nu})}  = \int_I A_{\nu}(h) w d\mathbb P_{\nu}   = \int_I w d(\mathbb P_{\nu} \ast (h\nu)) - \int h \nu \int w d\mathbb P_{\nu} \\
 &\quad = \int_I \int_I w(x+y) h(x) \nu(x)dx d\mathbb P_{\nu}(y) = \int_I h (\mathbb P_\nu(-\cdot) \ast w) d\nu = \langle h, A_\nu^*(w) \rangle_{L^2(\nu)}
\end{align*}
so that the formula for the adjoint holds on the dense subspace $C(I)$ of $L^2_0(\mathbb P_\nu)$. The Cauchy--Schwarz inequality implies that $\mathbb P_\nu (-\cdot) \ast w \in L^2(\nu)$ so that the case of general $w \in L_0^2(\mathbb P_\nu)$ follows from standard approximation arguments.
\end{proof}

Inspecting the formula for $A_\nu^*$ we can formally define the `inverse' map 
$$(A^*_\nu)^{-1}(g) = \pi_\nu(-\cdot) \ast g \text{ with } (\pi_\nu(-\cdot)\ast g)(x)=\int_I g(x+y)d\pi_\nu(y),~g\in L^2(\PP_\Lambda),$$
for $\nu\in L^\infty(I)$ and scaled by $1/\Delta$ if $\Delta \neq 1$. If $g \in L^\infty(I)$ satisfies $g(0)=0$ then using $\mathbb P_\nu\ast\pi_\nu=\delta_0$ (cf. after~(\ref{deconv})) we have that $(A^*_\nu)^{-1} (g) \in L^2_0(\mathbb P_\nu)$ since
\begin{equation} \label{zent}
\int_I  (A^*_\nu)^{-1} (g) d\mathbb P_\nu = \int_I \pi_\nu(-\cdot) \ast g \;d\mathbb P_\nu 
= \int_I  g \;d(\mathbb P_\nu\ast\pi_\nu)
=g(0)=0.
\end{equation}

\subsubsection{Inverse information operator and least favourable directions}

Now let $\psi \in L^\infty(I)$ be arbitrary but such that $\psi(0)=0$, for instance we can take $\psi 1_{\{0\}^c}$ for any $\psi \in C(I)$. If $\nu \in L^\infty(I)$ is bounded away from zero then $\psi/\nu\in L^2(\PP_\Lambda)$ and by what precedes $(A^*_\nu)^{-1} (\psi/\nu) \in L^2_0(\mathbb P_\nu)$ and hence in view of Lemma \ref{scoinv} we can define, for any such $\psi$, the new function
\begin{equation} \label{tpsi}
\tilde \psi_d = -\tilde A_\nu \left[\left(A^*_\nu\right)^{-1} \left(\frac{\psi}{\nu}\right)\right]
\end{equation}
as an element of $\mathcal D$. Concretely, in view of (\ref{convsum}), (\ref{deconv}), (when $\Delta=1$, otherwise divide the right hand side in the following expression by $\Delta^2$)
\begin{equation}\label{psid}
\tilde \psi_d = -\tilde A_{\nu}\bigg[\pi_\nu(-\cdot) \ast \frac{\psi}{\nu}\bigg]= -\frac{1}{\nu}\, \pi_\nu \ast \bigg(\Big(\pi_\nu(-\cdot) \ast \frac{\psi}{\nu}\Big) \mathbb P_\nu\bigg) (\cdot).
\end{equation} 
We can then write
$\tilde \psi_d = \tilde \psi + c \delta_0$ where 
\begin{equation}\label{psioh}
\tilde \psi = \tilde \psi _d -c \delta_0
\end{equation}
is the part of $\tilde \psi_d$ that is absolutely continuous with respect to Lebesgue measure $\Lambda$, and $c \delta_0$ is the discrete part (for some constant $c$).

The content of the next lemma is that $\tilde \psi$ allows to represent  the LAN inner product
\begin{equation}\label{lanprod}
\langle f, g \rangle_{LAN}\equiv \langle A_\nu (f), A_\nu (g) \rangle_{L^2(\mathbb P_\nu)},~~f,g \in L^2(\nu),
\end{equation} 
in the standard $L^2$-inner product $\langle \cdot, \cdot \rangle$ of $L^2(I)$.
\begin{lemma} \label{zero} Assume the L\'evy density $\nu \in L^\infty(I)$ is bounded away from zero on~$I$. If $\psi \in L^\infty(I)$ satisfies $\psi(0)=0$ then for all $h \in L^2(\nu)$ and $\tilde \psi_d, \tilde \psi$ given as in (\ref{psid}), (\ref{psioh}), $$\int_ I A_{\nu} (h) A_{\nu}(\tilde \psi) d\mathbb P_{\nu} = \int_I A_{\nu} (h) A_{\nu}(\tilde \psi_d) d\mathbb P_{\nu} = -\langle h,\psi \rangle.$$ \end{lemma}
\begin{proof}
From (\ref{akernel}) and (\ref{psioh}) we have $A_{\nu}(\tilde \psi_d - \tilde \psi)=0$, so the first identity is immediate. By Lemma \ref{scoinv} and the definition of $\tilde \psi_d$ we see $A_\nu(\tilde \psi_d) = - \pi_\nu(-\cdot) \ast (\psi/\nu)$ in $L^2_0(\mathbb P_\nu)$ and from Lemma \ref{adji} we hence deduce
$$\int_I A_\nu(h) A_{\nu} (\tilde \psi_d) d\mathbb P_{\nu} = -\int_I h [\mathbb P_{\nu}(-\cdot) \ast \pi_\nu(-\cdot) \ast (\psi/\nu)] \nu  =  - \int_I h \psi,$$ using also that $\mathbb P_{\nu}(-\cdot) \ast \pi_\nu(-\cdot) =\delta_0$ (cf. after (\ref{deconv})).
\end{proof}

\subsubsection{Cram\'er-Rao information lower bound}

Using the LAN expansion and the previous lemma  we derive the Cram\'er--Rao lower bound for $1/\sqrt n$-consistently estimable functional parameters of the L\'evy measure of a compound Poisson process, following the theory laid out in Chapter 25 in \cite{vdV98}. We recall some standard facts from efficient estimation in Banach spaces: assume for all $h$ in some linear subspace $H$ of a Hilbert space with Hilbert norm $\|\cdot\|_{LAN}$ that the LAN expansion
$$\log \frac{d \mathbb P^n_{v +h/\sqrt n}}{d \mathbb P^n_v} = \Delta_n(h) - \frac{1}{2} \|h\|_{LAN}^2,~~ v \in  H,$$
holds, where $\mathbb P^n_v$ are laws on some measurable space $\mathcal X_n$ and where $\Delta_n(h) \to^d \Delta(h)$ as $n \to \infty$ with $\Delta(h) \sim N(0, \|h\|_{LAN}^2), h \in H$. Consider a map $$K: (H, \|\cdot\|_{LAN}) \to \mathbb R$$ that is suitably differentiable with continuous linear derivative map $\kappa : H \to \mathbb R$. By Theorem 3.11.5 in \cite{vdVW96} the Cram\'er--Rao information lower bound for estimating the parameter $K (\nu)$ is given by $\|\kappa^*\|_{LAN}^2$ where $\kappa^*$ is the Riesz-representer of the map $\kappa: (H, \|\cdot\|_{LAN}) \to \mathbb R$.

We now apply this in the setting of the LAN expansion obtained from Proposition~\ref{lanprop}, with laws $\mathbb P^n_v$ parametrised by $v= \log \nu$, tangent space $H = L^\infty$ and LAN-norm $\|h\|_{LAN} = \|A_{\nu_0}h\|_{L^2(\mathbb P_{\nu_0})}$, where $A_{\nu_0}: (H, \|\cdot\|_{L^2(\nu_0)}) \to L^2_0(\mathbb P_{\nu_0})$ is the score operator studied above corresponding to the true absolutely continuous L\'evy density $\nu_0$ generating the data (note that the central limit theorem ensures $\Delta_n(h) \to^d \Delta (h)$ for these choices). For $\psi \in L^\infty(I)$ we consider the map $$K: v \mapsto \int_I \psi \nu = \int_I \psi e^v,$$ which can be linearised at $\nu_0$ with derivative $$\kappa: h \mapsto \int_I \psi h \nu_0 = \langle \psi_{(0)}, h \rangle_{L^2(\nu_0)} = \int_I \psi 1_{\{0\}^c} \nu_0 h,$$ where by definition $\psi_{(0)} = \psi 1_{\{0\}^c}$. Using Lemma \ref{zero} we have
\begin{align*}
\kappa(h)= \langle \psi_{(0)} \nu_0, h \rangle =  -\langle \widetilde {(\psi_{(0)} \nu_0)}_d, h \rangle_{LAN} \equiv \langle \kappa^*, h \rangle_{LAN}.
\end{align*}
We conclude that the Cram\'er--Rao information lower bound for estimating $\int_I \psi \nu_0$ from discretely observed increments of the compound Poisson process equals 
\begin{align} \label{crlb}
\|\kappa^*\|_{LAN}^2  &= \|A_{\nu_0}(\widetilde {(\psi_{(0)} \nu_0)}_d)\|_{L^2(\mathbb P_{\nu_0})}^2 = \|(A_{\nu_0}^*)^{-1} [\psi_{(0)}]\|_{L^2(\mathbb P_{\nu_0})}^2\nonumber\\
& = \|\pi_\nu(-\cdot) \ast (\psi 1_{\{0\}^c})\|_{L^2(\mathbb P_{\nu_0})}^2,
\end{align}
where we used Lemma \ref{scoinv}  in the second equality. Note that the last identity holds under the notational assumption $\Delta=1$ employed in the preceding arguments and the far right hand side needs to be scaled by $1/\Delta^2$ when $\Delta \neq 1$.

\subsection{A multi-scale Bernstein--von Mises theorem}

We now formulate a Bernstein--von Mises theorem that entails a Gaussian approximation of the posterior distribution arising from prior (\ref{prior}) in an infinite-dimensional multi-scale space. We will show in the next subsection how one can deduce from it various Bernstein--von Mises theorems for statistically relevant aspects of $\nu, \mu, \lambda$. Following \cite{CN14} (see also p.596f.~in \cite{GineNickl2016}) the idea is to study the asymptotics of the measure induced in sequence space by the action $(\langle \nu, \psi_{lk}\rangle)$ of  draws $\nu \sim \Pi(\cdot|X_1, \dots, X_n)$ of the conditional posterior distribution on the wavelet basis $\{\psi_{lk}\}$ from (\ref{wavbasis}). In sequence space we introduce weighted supremum norms 
\begin{equation}
\|x\|_{\mathcal M(w)}= \sup_{l} \frac{\max_k|x_{lk}|}{w_l},~~\mathcal M(w) =\{(x_{lk}): \|x\|_{\mathcal M(w)}<\infty\},
\end{equation}
 with monotone increasing weighting sequence $(w_l)$ to be chosen. Define further the closed separable subspace $\mathcal M_0(w)$ of  $\mathcal M(w)$ consisting of sequences for which $w_l^{-1}\max_k |x_{lk}|$ converges to zero as $l\to\infty$, equipped with the same norm.

The Bernstein--von Mises theorem will be derived for the case where the posterior distribution is centred at the  random element $\hat \nu (J)= (\hat \nu(J)_{l,k})$ of $\mathcal M_0(w)$ defined as follows
\begin{equation} \label{cent}
\hat \nu(J)_{l,k} \equiv  \int_I \psi_{l k} \nu_0 + \frac{1}{n} \sum_{i=1}^n (A_{\nu_0}^*)^{-1} [\psi_{l k} 1_{\{0\}^c}](X_i),~\quad l \le J-1, k,
\end{equation}
with the convention that $\hat \nu(J)_{l,k}=0$ whenever $l \ge J$ (the operator $(A_{\nu_0}^*)^{-1}$ was defined just after Lemma \ref{adji} above). A standard application of the central limit theorem and of~(\ref{zent}) implies as $n \to \infty$ and under $\mathbb P_{\nu_0}^\mathbb N$ that, for every fixed $k,l$, $$\sqrt n \big(\hat \nu(J)_{l,k} - \int_I \psi_{lk} \nu_0 \big) \to^d N(0, \|(A_{\nu_0}^*)^{-1} [\psi_{lk} 1_{\{0\}^c}]\|_{L^2(\mathbb P_{\nu_0})}^2),$$  and hence in view of (\ref{crlb}) the random variable $\hat \nu(J)$ is a natural centring for a Bernstein--von Mises theorem. Since $\nu \in L^\infty(I)$ the law of $\sqrt n(\nu-\hat \nu(J))$ defines a probability measure in the space $\mathcal M_0(\omega)$ for $\omega$ as in the next theorem. Next, denote by $\mathcal N_{\nu_0}$ the law $\mathcal L(\mathbb X)$ of the centred Gaussian random variable $\mathbb X$ on $\mathcal M(w)$ whose coordinate process has covariances $$E\mathbb X_{l,k} \mathbb X_{l',k'} = \langle (A_{\nu_0}^*)^{-1}(\psi_{lk} 1_{\{0\}^c}), (A_{\nu_0}^*)^{-1}(\psi_{l'k'} 1_{\{0\}^c})\rangle_{L^2(\mathbb P_{\nu_0})}.$$ The proof of the following theorem implies in particular that $\mathcal N_{\nu_0}$ is a tight Gaussian probability measure concentrated on the space $\mathcal M_0(w)$ where weak convergence occurs. Recall (Theorem 11.3.3 in \cite{D02}) that weak convergence of a sequence of probability measures on a separable metric space $(S,d)$ can be metrised by the bounded Lipschitz (BL) metric
\begin{align*}
\beta_S(\kappa,\kappa')&= \sup_{F:S \to \mathbb R, \|F\|_{Lip} \le 1} \left|\int_S F(s) d(\kappa-\kappa')(s) \right|,\\
\|F\|_{Lip} &= \sup_{s \in S}|F(s)|+ \sup_{s\neq t, s,t \in S} \frac{|F(s)-F(t)|}{d(s,t)}.
\end{align*}
\begin{theorem}\label{bvm1}
Suppose that $X_1, \dots, X_n$ are generated from (\ref{increm}) and grant Assumption~\ref{overall}.  Let $\Pi(\cdot|X_1,\dots, X_n)$ be the posterior distribution arising from prior $\Pi=\Pi_J$ in (\ref{prior}) with~$J$ as in (\ref{jprior}). Let $\beta_{\mathcal M_0(\omega)}$ be the BL metric for weak convergence of laws in  $\mathcal M_0(\omega)$, with $\omega=(\omega_l)$ satisfying $\omega_l / l^4 \uparrow \infty$ as $l \to \infty$. Let $\hat \nu_J$ be the random variable in $\mathcal M_0(\omega)$ given by~(\ref{cent}). Then for $\nu \sim \Pi(\cdot|X_1, \dots, X_n)$ and $\mathcal N_{\nu_0}$ as above we have in $\mathbb P_{\nu_0}^\mathbb N$-probability, as $n \to \infty$, $$\beta_{\mathcal M_0(\omega)}\left(\mathcal L\big(\sqrt n (\nu-\hat \nu(J))|X_1, \dots, X_n\big), \mathcal N_{\nu_0} \right) \to 0.$$ 
\end{theorem}

Theorem \ref{bvm1} is proved in Section~\ref{sec:bvm1} and has various implications for posterior-based inference on the parameter~$\nu$. Arguing as in \cite{CN14}, Section 4.2, we could construct credible bands for the unknown L\'evy density $\nu$ with $L^\infty$-diameter shrinking at the rate as in Theorem~\ref{suprat} from Bayesian multi-scale credible bands. We will leave this application to the reader and instead focus on inference on functionals of the L\'evy measure $\nu$ that are continuous, or differentiable, for  $\|\cdot\|_{\mathcal M(\omega)}$ (see Section 4.1 in \cite{CN14}, \cite{C17}).

Theorem~\ref{bvm1} assumes a certain growth at infinity of the weight sequence $\omega_l$. The requirement $\omega_l/\sqrt l \uparrow \infty$ is necessary for the limit process to be a tight Gaussian Borel probability measure in the space $\mathcal M_0(\omega)$, see \cite{CN14}. Similar to the presence of an additional log-factor in Theorem \ref{suprat}, here we need to impose the slightly more restrictive condition $\omega_l / l^4 \uparrow \infty$ for the control of semi-parametric bias terms in our proofs.

\subsection{Bernstein--von Mises theorem for functionals of the L\'evy measure}

We now deduce from Theorem \ref{bvm1}  Bernstein--von Mises theorems for the functionals $$V(t)=\int_{-1/2}^t \nu(x)dx, ~t \in I,$$ which for $t=1/2$ also includes the intensity $\lambda = \int_I d\nu = V(1/2)$ of the underlying Poisson process. From the usual `Delta method' we can then also deduce a Bernstein--von Mises theorem for the distribution function $M(t)= \int_I 1_{(-1/2,t]} d\mu$ of the jump measure $\mu =\nu/\lambda= \nu/\int_I \nu$.  The key to this is the following lemma, proved in (the proof of) Theorem 4 of \cite{CN14}. 
\begin{lemma}\label{parscn}
Suppose the weights $(\omega_l)$ satisfy $\sum_l 2^{-l/2}\omega_l<\infty$. Then the mapping $$L: (\nu_{lk}) \mapsto V = \int_0^\cdot \sum_{l,k}\nu_{lk} \psi_{lk}$$ is linear and continuous from $\mathcal M_0(\omega)$ to $L^\infty(I)$ for the respective norm topologies. 
\end{lemma}
For the next theorem we require some more definitions: We denote $V_0(t)=\int_{-1/2}^t\nu_0(x)dx$.
Let $\mathcal N_{V_0}$ be the law of the tight Gaussian random variable in $L^\infty(I)$ given by $L(Z), Z \sim \mathcal N_{\nu_0}$.
We define $l_{\nu_0}$ to be the linear mapping $L^\infty(I) \to L^\infty(I)$ with  $l_{\nu_0}[h]=(h V_0(\frac{1}{2})-V_0 h(\frac{1}{2}))/V_0^2(\frac{1}{2})$.
Finally we denote by $\mathcal N'_{M_0}$ the law of the tight Gaussian random variable in~$L^\infty(I)$ given by $l_{\nu_0}[L(Z)]$.

The measures $\mathcal N_{V_0}, \mathcal N'_{M_0}$ have separable range in the image in $L^\infty(I)$ of $\mathcal M_0(\omega)$ under a continuous map. The metrisation of weak convergence of laws towards $\mathcal N_{V_0}, \mathcal N'_{M_0}$ in the non-separable space $L^\infty$ by $\beta_{L^\infty(I)}$ thus remains valid (Theorem 3.28 in \cite{D14}).
\begin{theorem}\label{bvm2}
Suppose that $X_1, \dots, X_n$ are generated from (\ref{increm}) and grant Assumption~\ref{overall}.  Let $\nu \sim \Pi(\cdot|X_1,\dots, X_n)$ be a draw from the posterior distribution arising from prior $\Pi=\Pi_J$ in (\ref{prior}) with $J$ as in (\ref{jprior}) and let $L$ be the linear mapping from Lemma \ref{parscn}. Conditional on $X_1, \dots, X_n$ define $V=L(\nu)$ and $\hat V =L(\hat \nu_J)$ where $\hat \nu_J$ is given in (\ref{cent}).

Then we have as $n \to \infty$ and in $\mathbb P_{\nu_0}^\mathbb N$-probability that $$\beta_{L^\infty(I)}\left(\mathcal L(\sqrt n (V-\hat V)|X_1, \dots, X_n), \mathcal N_{V_0} \right) \to 0.$$  In particular if $N_{\lambda_0}$ is the law on $\mathbb R$ of $L(Z)(\frac{1}{2})$ then as $n \to \infty$, $$\beta_{\mathbb R}\left(\mathcal L(\sqrt n (V(\tfrac{1}{2})-\hat V(\tfrac{1}{2}))|X_1, \dots, X_n), N_{\lambda_0}\right) \to^{\mathbb P_{\nu_0}^\mathbb N} 0.$$
Moreover, if $M = V/V(\frac{1}{2})$ and $\hat M = \hat V/\hat V(\frac{1}{2})$, then as $n \to \infty$, $$\beta_{L^\infty(I)}\left(\mathcal L(\sqrt n (M-\hat M)|X_1, \dots, X_n), \mathcal N'_{M_0} \right) \to^{\mathbb P_{\nu_0}^\mathbb N} 0.$$  
\end{theorem}
\begin{proof}
The first two limits are immediate consequences of Theorem~\ref{bvm1}, Lemma~\ref{parscn} and the continuous mapping theorem. For the last limit we apply the Delta method for weak convergence (\cite{vdV98}, Theorem 20.8) to the map $V \mapsto  V/V(\frac{1}{2})$, which is Fr\'echet differentiable from $L^\infty(I) \to L^\infty(I)$ at any $\nu \in L^\infty(I)$ that is bounded away from zero, with derivative $l_{\nu}$.
\end{proof}

Arguing just as before (\ref{crlb}) one shows that the above Gaussian limit distributions all attain the semi-parametric Cram\'er--Rao lower bounds for the problems of estimating $V, M, \lambda=V(\frac{1}{2})$, respectively. In particular they imply that `Bayesian credible sets' are optimal asymptotic frequentist confidence sets for these parameters -- the arguments are the same as in \cite{CN14}, Section 4.1, and hence omitted. These results are the `Bayesian' versions of the Donsker type limit theorems obtained for frequentist estimators in \cite{NR12, Co17}, where the same limit distributions were obtained. 

\subsection{Concluding Remarks}

\noindent

\textbf{Adaptive prior choices:} Our series prior is defined via asymptotic growth of $J$ (see \eqref{jprior}) that depends on $n$ and on knowledge of the smoothness~$s$. A possible extension of our work would be to make the results adaptive to the choice of $J$, e.g., by placing a hyperprior on $J \in \mathbb N$ whose probability mass function is proportional to $\exp(-c2^JL(J))$ with $L(J)=J$ or $=1$. While it seems possible to prove an upper bound for $2^J$ of order $(n/\log n)^{1/(2s+1)}$ with such a hyperprior, it is unclear whether a corresponding lower bound holds as well. Small values of $J$ can entail a large bias and the control of the semi-parametric bias poses considerable difficulties in our proofs. As in \cite{R17}, a self-similarity condition on $\nu$ may help to overcome such problems, but this is beyond the scope of the present paper.

\smallskip

\textbf{Scaling of $\Delta$:} For identifiability reasons, Assumption~\ref{overall} imposes an upper bound on the (fixed) distance between observations $\Delta$. Otherwise the observation distance $\Delta$ enters the contraction rate result in Theorem~\ref{suprat} only via multiplicative constants. In the Bernstein--von Mises results (Theorems~\ref{bvm1} and~\ref{bvm2}), the limiting processes scale with $1/\Delta$, as can be seen from the scaling of $(A_{\nu}^*)^{-1}$ before equation~\eqref{zent}. This suggests that `high-frequency' analogues of our Bernstein-von Mises results, comparable to those in \cite{NRST16}, should hold true as well, with convergence rate $1/\sqrt{n \Delta}$ instead of $1/\sqrt n$.

\smallskip

\textbf{Bernstein--von Mises theorems for general inverse problems:} This paper builds on key ideas for nonparametric Bernstein--von Mises theorems in direct models \cite{CN13,CN14, C14, CR15, C17}. For inverse problems previous work on Bernstein--von Mises theorems treated regression-type problems where the likelihood has a more explicit Gaussian structure, see \cite{N17,MNP17} and also the more recent contributions \cite{GK18, NR18}. In our jump process setting, the log-likelihood function does not have the form of a Gaussian process, but we show how empirical process methods \cite{GineNickl2016} can be used to obtain exact Gaussian posterior asymptotics in such situations as well. Our proof techniques are thus potentially relevant for other models with independent and identically distributed observations.

\section{Proofs of the main theorems}\label{sec_proofs}

\subsection{Asymptotics for the localised posterior distribution} \label{asysec}

The first step will be to localise the posterior distribution near the `true' $\nu_0 \in C^s$ by obtaining a preliminary (in itself sub-optimal) contraction rate for the prior $\Pi$ from (\ref{prior}). Recall the notation $v=\log \nu$ and define
\begin{equation} \label{dn}
D_{n,M} := \left\{\nu: v \in V_{B,J}, \|v - v_0\|_{L^2} \le M \eps_{n}^{L^2}, \|v - v_0\|_{\infty} \le M \eps_{n}^{L^\infty}  \right\}
\end{equation}
with $M$ a constant and $$\eps_{n}^{L^2} = n^{-\frac{s-1/2}{2s+1}}(\log n)^{1/2+\delta},\qquad \eps_{n}^{L^\infty} = n^{-\frac{s-1}{2s+1}}(\log n)^{1/2+\delta}$$ for any $\delta>1/2$. We have the following
\begin{proposition}\label{initio}
For $D_{n,M}$ as in (\ref{dn}), prior $\Pi$ arising from (\ref{prior}) with $J$ chosen as in (\ref{jprior}) and under Assumption~\ref{overall}, we have for any $s>5/2, \delta>1/2$ and every $M$ large enough
\begin{equation} \label{cont1}
\Pi (D_{n,M}^c | X_1, \dots, X_n) \to^{\mathbb P^\mathbb N_{\nu_0}} 0
\end{equation}
as $n \to \infty$. In particular we can choose $M$ in (\ref{dn}) large enough so that the last convergence to zero occurs also for $D_{n,M/2}$ replacing $D_{n,M}$. Moreover, on the set $D_{n,M}$ we also have the same contraction rates for $\nu - \nu_0$ in place of $v-v_0$ with a possibly larger constant $M$.
\end{proposition}
\begin{proof}
This is proved in Section \ref{prelimc} below.
\end{proof}

As a consequence of the previous proposition, if $\Pi^{D_{n,M}}:=\Pi^{D_{n,M}}(\cdot|X_1, \dots, X_n)$ equals the posterior measure arising from the prior $\Pi(\cdot \cap D_{n,M})/\Pi(D_{n,M})$ instead of from $\Pi$, we can deduce the basic inequality
\begin{align} 
&\sup_{B \in \mathcal S_V} |\Pi(B|X_1, \dots, X_n) -\Pi^{D_{n,M}}(B|X_1, \dots, X_n) | \nonumber\\
&\qquad\le 2 \Pi(D_{n,M}^c|X_1, \dots, X_n) \to^{\mathbb P^\mathbb N_{\nu_0}} 0\label{tvlim}
\end{align}
as $n \to \infty$. We now study certain Laplace-transform functionals of the localised posterior measure $\Pi^{D_{n,M}}$. We use the shorthand notation $V_J$ for the $L^2$-closed linear space spanned by the wavelets up to level $J$ and $g_J = P_{V_J}(g)$ for the wavelet projection of $g \in L^2(I)$ onto $V_J$. For a fixed function $\eta:I \to \mathbb R$, consider a perturbation of $\nu$ given by
\begin{align}
&\nu_t = \nu_t^{\eta} := e^{v_t},\label{pert}\\
&v_t = v + \delta_n \Big(\frac{t}{\delta_n \sqrt n} \eta + v_{0,J} -v  \Big) = (1-\delta_n) v + \delta_n \Big(\frac{t}{\delta_n \sqrt n} \eta+ v_{0,J}\Big),\nonumber
\end{align}
where $0<t <\infty$ and $\delta_n \to 0$ such that $\delta_n \sqrt n \to \infty$ is a sequence to be chosen. That the perturbation $\nu_t$ equals a convex combination of points will be useful to deal with the fact that our parameter space has a boundary (see also \cite{N07, N17}).

We have the following key proposition, giving general conditions under which a (sub-) Gaussian approximation for the Laplace transform of general functionals $F(\nu)$ of the posterior distribution holds. Its proof is given in Section~\ref{pertu}.
\begin{prop}\label{prop:functbvm}
Under the hypotheses of Proposition \ref{initio}, suppose $\delta_n$ is chosen such that~\eqref{deltacon} is satisfied and let $\mathcal H_n \subset L^\infty(I)$ be such that \eqref{etaL2}, \eqref{larges} hold uniformly for all $\eta\in\mathcal H_n$. If $T>0$ and if $F: \mathcal V \to \mathbb R$ is any fixed measurable function then 
\begin{align*}
&E^{\Pi^{D_{n,M}}}\left[e^{t \sqrt n F(\nu)} \Big|X_1, \dots, X_n \right] \\
&\qquad= \exp \Big\{\frac{t^2}{2}\|A_{\nu_0}(\eta)\|_{L^2(\mathbb P_{\nu_0})}^2 - \frac{t}{\sqrt n} \sum_{i=1}^n A_{\nu_0}(\eta)(X_i) + r_n \Big\} \times Z_n
\end{align*}
where $r_n=O_{\mathbb P^\mathbb N_{\nu_0}}(a_n)$ as $n \to \infty$ with a nonstochastic null sequence $a_n \to 0$ that is uniform in $|t|\le T$, $\eta\in\mathcal H_n$; and where $$Z_n = \frac{\int_{D_{n,M}} e^{S_n(\nu) + \ell_n(\nu_t)}d\Pi(\nu)}{\int_{D_{n,M}}e^{\ell_n(\nu)}d\Pi(\nu)},\quad \nu_t \text{ as in } (\ref{pert}),$$
$$S_n(\nu) = t \sqrt n \Big(F(\nu) +\int A_{\nu_0}(v-v_0) A_{\nu_0}(\eta)d\mathbb P_{\nu_0}\Big),~~ v=\log \nu,~ v_0 = \log \nu_0, $$ and $A_\nu: L^2(\nu) \to L^2_0(\mathbb P_\nu)$ was defined in Proposition \ref{lanprop}.
\end{prop}

Given a functional $F$ of interest, Proposition \ref{prop:functbvm} can be used to prove Bernstein--von Mises theorems by selecting appropriate $\eta$ so that $S(\nu)$ vanishes (or converges to zero). When this is the case it remains to deal with $Z_n$ by a change of measure argument for $\nu \mapsto \nu_t$.

\subsection{Change of measure in the posterior}\label{subsec}

We now study the ratio $Z_n$ for $\eta, \delta_n$ satisfying certain conditions, and under the assumption that $\sup_{\nu \in D_{n,M}} |S_n(\nu)|$ is either $O(1)$ or $o(1)$.  Note that by Assumption~\ref{overall}, $v_0=\log \nu_0$ is an `interior' point of the support $$V_{B,J} =\prod_{l=-1}^{J-1}(-Ba_l, Ba_l)^{2^l\vee1} \subset \mathbb R^{2^J},~~a_l = 2^{-l}(l^2+1)^{-1},$$ of the prior $\Pi$. We shall require that $(t/\delta_n \sqrt n) \eta+ v_{0,J}$ is also contained in $V_{B,J} $, implied by 
\begin{align} \label{intc}
t |\langle \eta , \psi_{lk} \rangle| \le \gamma 2^{-l}(l^2+1)^{-1} \sqrt n \delta_n \quad  \forall l <J-1,k, \qquad  \langle \eta , \psi_{lk} \rangle =0  \quad \forall l>J.
\end{align}
Note that under (\ref{intc}) the function $v_t$ from (\ref{pert}) is a convex combination of elements $v, (t/\delta_n \sqrt n) \eta+ v_{0,J}$ of $V_{J,B}$ and hence itself contained in the support $V_{J,B}$ of $\Pi$. We can thus write
$$\frac{\int_{D_{n,M}} e^{\ell_n(\nu_t)}d\Pi(\nu)}{\int_{D_{n,M}}e^{\ell_n(\nu)}d\Pi(\nu)} =\frac{\int_{D_{n,M}^t} e^{\ell_n(\nu)}\frac{d\Pi^t(\nu)}{d\Pi(\nu)} d\Pi(\nu)}{\int_{D_{n,M}}e^{\ell_n(\nu)}d\Pi(\nu)},$$
where $\Pi^t$ is the law of $\nu_t$, absolutely continuous with respect to $\Pi$, and where  $$D_{n,M}^t=\{\nu_t: \nu \in D_{n,M}\}.$$ The measure $\Pi^t$ corresponds to transforming each coordinate $v_{lk}$ of the $2^J$-dimensional product integral defining the prior $\Pi$  into the convex combination $v_{t,lk}=(1-\delta_n) v_{lk} + \delta_n i_{t,lk}$ where $i_{t,lk} = \langle \frac{t}{\delta_n \sqrt n} \eta+ v_{0,J}, \psi_{lk} \rangle $ is a deterministic (under $\Pi$) point in $(-B a_l, Ba_l)=I_{l,B}$ for every $k, l \le J$. The density of the law of $v_{t,lk}$ with respect to $v_{lk}$ is constant on a subinterval of $I_{l,B}$ of length $2B(1-\delta_n)$ and thus has constant density $(1-\delta_n)^{-1}$. The density of the product integrals is then also constant in $v$ and equal to 
\begin{equation} \label{deltac}
\left(\frac{1}{1-\delta_n}\right)^{2^J} = 1+o(1)\text{ whenever } 2^J \delta_n = o(1),
\end{equation}
independently of $\nu$. We conclude that if (\ref{intc}), (\ref{deltac}) hold then
\begin{align} 
\frac{\int_{D_{n,M}} e^{\ell_n(\nu_t)}d\Pi(\nu)}{\int_{D_{n,M}}e^{\ell_n(\nu)}d\Pi(\nu)}&= (1+o(1)) \times \frac{\int_{D_{n,M}^t} e^{\ell_n(\nu)} d\Pi(\nu)}{\int_{D_{n,M}}e^{\ell_n(\nu)}d\Pi(\nu)}\label{aratio}\\
&= (1+o(1)) \times \frac{\Pi(D_{n,M}^t|X_1, \dots, X_n)}{\Pi(D_{n,M}|X_1, \dots, X_n)},\nonumber
\end{align}
where the last identity follows from renormalising both numerator and denominator by $\int_\mathcal V e^{\ell_n(\nu)} d\Pi(\nu)$. The numerator in the last expression is always less than or equal to one and by Proposition \ref{initio} the denominator converges to one in probability, so that we have

\begin{lemma} \label{postupbd}
Suppose $\sup_{\nu \in D_{n,M}} |S_n(\nu)|=O(1)$ holds as $n \to \infty$ and assume $\eta, \delta_n, t$ are such that (\ref{intc}), (\ref{deltac}) hold. Then the random variable $Z_n$ in Proposition~\ref{prop:functbvm} is $O_{\mathbb P^\mathbb N_{\nu_0}}(1)$, uniformly in $\eta$, as $n \to \infty$.
\end{lemma}

To prove the exact asymptotics in the Bernstein--von Mises theorem we need:

\begin{lemma}\label{postconv}
Suppose $\eta, \delta_n$ are such that (\ref{intc}), (\ref{deltac}) hold and assume in addition that $\|\eta\|_\infty \le d$ for some fixed constant $d$. 

A) Let $D_{n,M}$ be as in (\ref{dn}) and define the set $D_{n,M}^t = \{\nu_t : \nu \in D_{n,M}\}$. Then for all $n \ge n_0(t)$ and $M$ large enough we have $D_{n,M/2} \subset D_{n,M}^t$ and thus by Proposition \ref{initio} also $\Pi(D_{n,M}^t|X_1, \dots, X_n) \to 1$ in $\mathbb P_{\nu_0}^\mathbb N$-probability.

B) Assume also that $\sup_{\nu \in D_{n,M}} |S_n(\nu)|=o(1)$ then $Z_n$ from Proposition \ref{prop:functbvm} satisfies $Z_n  = 1+o_{\mathbb P_{\nu_0}^\mathbb N}(1)$ as $n \to \infty$.
\end{lemma}
\begin{proof}
A) Let $\nu \in D_{n,M/2}$ be arbitrary. We need to show that there exists $\zeta=\zeta(\nu) \in D_{n,M}$ such that $\zeta_t=\nu$. For $v=\log \nu$ notice that by definition of $D_{n,M/2}$ we have $\|v-v_{0,J}\|_{L^2} \le \|v-v_0\|_{L^2} \le (M/2) \eps_{n}^{L^2}$ and similarly $\|v- v_{0,J}\|_{\infty} \le (M/2) \eps_{n}^{L^\infty}$. Now define $\zeta = e^z$ where $$z= z(\nu) := v_{0,J} + \frac{(v- v_{0,J}) - \frac{t}{\sqrt n} \eta}{1-\delta_n}, ~\nu \in D_{n,M/2}.$$ Then by definition 
\begin{align*}
z_t  &= (1-\delta_n) z + \frac{t}{\sqrt n} \eta + \delta_n v_{0,J} \\
& = (1-\delta_n) v_{0,J} + (v-v_{0,J}) - \frac{t}{\sqrt n} \eta + \frac{t}{\sqrt n} \eta + \delta_n v_{0,J}  = v
\end{align*}
so $\zeta_t(\nu)=\nu$ follows. It remains to verify that also $\zeta(\nu) \in D_{n,M}$ for every $\nu \in D_{n,M/2}$. To see this we let $n$ large enough such that in particular $\delta_n<1/4$ and then
\begin{align} \label{zweirat}
\|z(\nu)-v_0\|_{L^2} \le \|v_0-v_{0,J}\|_{L^2} + \frac{4}{3} \|v-v_{0,J}\|_{L^2} + \frac{4t}{3 \sqrt n} \|\eta\|_{L^2} \le M\eps_{n}^{L^2} 
\end{align}
using $\|v_0-v_{0,J}\|_{L^2} \lesssim 2^{-Js}=o(\eps_n^{L^2})$ from (\ref{wavprop}) and also $1/\sqrt n =o(\eps_n^{L^2})$. The same arguments imply $$\|z(\nu)-v_0\|_\infty   \le M\eps_{n}^{L^\infty}.$$ 
Finally we need to check that $z(\nu) \in V_{J,B}$ holds true. We notice that for all $l \le J$ $$|\langle z(\nu)-v_0, \psi_{lk}\rangle | \le \|z(\nu)-v_0\|_{L^2} \le \gamma 2^{-l}(l^2+1)^{-1} = \gamma a_l$$ is implied by 
$$\eps_{n}^{L^2} \approx n^{-\frac{s-1/2}{2s+1}}(\log n)^{1/2+\delta}=o(2^{-J}(J^2+1)^{-1}),~~~s > 5/2,$$
for $n$ large enough, so that from Assumption \ref{overall} and (\ref{zweirat}) we deduce
\begin{align*}
|\langle z(\nu), \psi_{lk} \rangle| &\le |\langle v_0, \psi_{lk}\rangle| + |\langle z(\nu)-v_0, \psi_{lk} \rangle| 
 \le (B-\gamma)a_l +\gamma a_l,~~l \le J-1,
\end{align*}
for $n$ large enough, hence $\zeta \in V_{J,B}$. The last claim in Part A) now follows directly from Proposition \ref{initio}, and Part B) also follows, from (\ref{aratio}).
\end{proof}

\subsection{Proof of Theorem \ref{suprat}}

Given the results from Sections \ref{asysec}, \ref{subsec}, the proof follows ideas in \cite{C14}. By (\ref{tvlim}) it suffices to prove the theorem with the posterior $\Pi(\cdot|X_1, \dots, X_n)$ replaced by $\Pi^{D_{n,M}}(\cdot|X_1, \dots, X_n)$. Using that $\nu=e^v$ are uniformly bounded and that $v_J=P_{V_J}v=v$ for $v \sim \Pi^{D_{n,M}}(\cdot|X_1, \dots, X_n)$, we can write
$$\|\nu - \nu_0\|_\infty \lesssim \|v-v_0\|_\infty \leq \|v_J - v_{0,J}\|_\infty + \|v_{0,J}-v_0\|_\infty.$$ The second term is of deterministic order $2^{-J_ns}=O(n^{-s/(2s+1)})$ by (\ref{wavprop}) and since $v_0 = \log \nu_0 \in C^s$, so it remains to deal with the first. We can write, using (\ref{wavprop}) again,
\begin{align} \label{wavred}
\|v_J - v_{0,J}\|_\infty &= \sup_x \bigg| \sum_{\ell < J, m} \langle v-v_0, \psi_{\ell m} \rangle \psi_{\ell m}(x)\bigg| \notag \\
& \lesssim \sum_{\ell < J} \frac{2^{\ell/2}}{\sqrt n} (\log n)^{1/2+\delta}\max_{m=0,\dots, 2^\ell-1} \frac{\sqrt n}{(\log n)^{1/2+\delta}}\left|\langle v-v_0, \psi_{\ell m} \rangle\right| \notag \\
& \lesssim \frac{2^{J/2} (J+1)}{\sqrt n} (\log n)^{1/2+\delta} \max_{\ell <J, m=0,\dots, 2^\ell-1} \sqrt n \left|\langle v-v_0, c_{\ell J}\psi_{\ell m} \rangle\right|,
\end{align}
where we have set $c_{\ell J} =\frac{2^{\ell/2}}{2^{J/2}}(\log n)^{-1/2-\delta}$, bounded by $1$ since $\ell \le J$.

Fix $\ell<J, m$ for the moment and let $\tilde \psi \equiv (\tilde \psi)_{\ell m}$ be the absolutely continuous part (\ref{psioh}) of $\tilde \psi_d$ from (\ref{tpsi}) where we choose $\psi = c_{\ell J} \psi_{\ell m}1_{I\setminus\{0\}}.$  We will apply Proposition~\ref{prop:functbvm} to the functional $F(\nu) = \langle v-v_0, c_{\ell J} \psi_{\ell m} \rangle$ and for the choices 
\begin{equation}\label{etadelta}\eta = \tilde \psi_J \quad\text{and}\quad \delta_n= \frac{K2^J(J^2+1)}{\sqrt{n}},\end{equation}
where $K>0$ is a constant. To bound the term $S_n(\nu)$ in Proposition~\ref{prop:functbvm} we need the following approximation lemma.

\begin{lemma} \label{approxi}
For any $\psi = c_{\ell J} \psi_{\ell m}1_{I\setminus\{0\}}$ with fixed $\ell < J, m,$ let $\tilde \psi_d $ be the corresponding finite measure defined in (\ref{tpsi}), let $\tilde \psi$  be its absolutely continuous part from (\ref{psioh}), and let $\tilde \psi_J=P_{V_J}(\tilde \psi)$ be its wavelet projection onto $V_J$. Then we have, for some constant $c_0$ independent of $\ell, m, J$, that
$$\left| c_{\ell J}\int_I (v-v_0) \psi_{\ell m}+\int_I A_{\nu_0}(v-v_0) A_{\nu_0}(\tilde \psi_J)d\mathbb P_{\nu_0} \right| \leq c_0 \frac{\|\nu-\nu_0\|_{L^2}}{ 2^{J}(\log n)^{1/2+\delta}}.$$
\end{lemma}
\begin{proof}
We notice that Lemma~\ref{zero} implies $$c_{\ell J} \int_I (v-v_0) \psi_{\ell m} =c_{\ell J}  \int_I (v-v_0) \psi_{\ell m}1_{I\setminus\{0\}} = -\int_I A_{\nu_0}(v-v_0) A_{\nu_0}(\tilde \psi)d\mathbb P_{\nu_0},$$ so that by linearity of the operator $A_{\nu_0}$ and Lemma \ref{adji} it suffices to bound
\begin{align*}
\int_I A_{\nu_0}(v-v_0) A_{\nu_0}(\tilde \psi_J-\tilde \psi)d\mathbb P_{\nu_0} &= \int_I \nu_0 A^*_{\nu_0}[A_{\nu_0}(v-v_0)] (\tilde \psi_J - \tilde \psi) \\
& = \sum_{l>J} \sum_k \langle h(\nu, \nu_0), \psi_{lk} \rangle \langle \tilde \psi, \psi_{lk} \rangle,
\end{align*}
where we have used  Parseval's identity, and the shorthand notation $h(\nu, \nu_0):=\nu_0 A^*_{\nu_0}[A_{\nu_0}(v-v_0)]$. 
Now $\tilde\psi$ is the absolutely continuous part of $\tilde \psi_d$ which according to \eqref{psid} (with $\Delta=1$ without loss of generality) is given by
\begin{align*}
\tilde \psi_d &= -\frac{1}{\nu_0}\, \pi_{\nu_0} \ast \bigg(\Big(\pi_{\nu_0}(-\cdot) \ast \frac{\psi}{{\nu_0}}\Big) \mathbb P_{\nu_0}\bigg) \\
&= -\frac{e^{2 \nu_0(I)}}{\nu_0}\bigg(\sum_{\iota=0}^{\infty}\sum_{\kappa=0}^\infty \frac{(-1)^{\iota+\kappa}}{\iota!\kappa!}\Big(\nu_0^{\ast \iota}\ast\nu_0(-\cdot)^{\ast \kappa}\ast\frac{\psi}{\nu_0}\Big)\PP_{\nu_0}\bigg).
\end{align*}
By standard properties of convolutions, using (\ref{convsum}) and since $\psi/\nu_0$ is absolutely continuous, removing the discrete part of $\tilde \psi_d$ means removing Dirac measure from the series expansion of $\mathbb P_{\nu_0}$ -- denote the resulting absolutely continuous measure by $P_{\nu_0}$. First we consider the part $\bar \psi$ of $\tilde \psi$ corresponding to the terms in the last series where either $\iota>0$ or $\kappa>0$, so that not all of the convolution factors in $$\nu_0^{\ast \iota}\ast\nu_0(-\cdot)^{\ast \kappa}\ast\frac{\psi}{\nu_0}$$ are Dirac measures $\delta_0$. Since $C^s(I), s>5/2,$ is imbedded into the standard periodic Sobolev space $H^\alpha(I), \alpha \le 2,$ we can use the basic convolution inequality $\|f \ast g \|_{C^\alpha(I)} \le \|f\|_{H^\alpha(I)}\|g\|_{L^2}, \alpha =0,2,$ (proved, e.g., just as Lemma 4.3.18 in \cite{GineNickl2016}), the fact that $\psi/\nu_0 = c_{\ell J}\psi_{\ell m}/\nu_0$ is bounded in $L^2=H^0$, and the multiplier property $\|fg\|_{H^2}\lesssim\|f\|_{C^{2}}\|g\|_{H^2}$ combined with the fact that the density of $P_{\nu_0}$ is contained in $C^s(I)\subset C^2(I),$ to deduce that $\bar \psi$ is contained in $C^2(I)$ and thus, by (\ref{wavprop})
\begin{align*}
\bigg|\sum_{l>J} \sum_k \langle h(\nu, \nu_0), \psi_{lk} \rangle \langle \bar \psi, \psi_{lk} \rangle \bigg| &
\le \sum_{l>J} \| \langle h(\nu, \nu_0), \psi_{l\cdot} \rangle\|_{L^2} \|\langle \bar \psi, \psi_{l\cdot} \rangle\|_{L^2}\\
&\lesssim \sum_{l>J} \|\nu-\nu_0\|_{L^2}2^{-2l}\lesssim \|\nu-\nu_0\|_{L^2}2^{-2J},
\end{align*}
which is of the desired order.

Setting $\iota=\kappa=0$ in the preceding representation of $\tilde \psi$ and using the convolution series representation of $P_{\nu_0}$ (without discrete part) yields the `critical' term which is given by $-\psi g$ where
\[g=c\frac{1}{\nu_0^2}\sum_{j=1}^\infty \frac{\nu_0^{\ast j}}{j!},\] for a suitable constant $c>0$. By arguments similar to above the function $g$ is at least in $C^2$ and for $x_{lk}$ the mid-point of the support set $S_{lk}$ of $\psi_{lk}$ (an interval of width $O(2^{-l})$ at most) we can write
\begin{align*}
\langle \psi_{\ell m} g, \psi_{lk} \rangle&= \int_I \psi_{\ell m} (g-g(x_{lk})+g(x_{lk})) \psi_{lk} \\
&= \int_I \psi_{\ell m} \psi_{lk} (g-g(x_{lk})) + g(x_{lk}) \int_I \psi_{\ell m} \psi_{lk}.
\end{align*}
The last term vanishes by orthogonality ($\ell \le J < l$), and using the mean value theorem the absolute value of the first is bounded by $$\|g'\|_\infty \int_{S_{lk}} |x-x_{lk}| |\psi_{\ell m}(x)| |\psi_{lk}(x)|dx \lesssim 2^{-l} \int_I |\psi_{\ell m}(x)| |\psi_{lk}(x)|dx.$$ Then, using (\ref{wavprop}) and the standard convolution inequalities for $L^2$-norms,
\begin{align*}
&\sum_{l>J} 2^{-l} \sum_k |\langle h(\nu, \nu_0), \psi_{lk} \rangle| \int_I |\psi_{\ell m}| |\psi_{lk}|  \\
&\qquad\le \sum_{l>J} 2^{-l} \|h(\nu,\nu_0)\|_{L^2} \int_I |\psi_{\ell m}(x)| \sum_k |\psi_{lk}(x)| dx \\
 &\qquad\lesssim \sum_{l>J} 2^{-l/2} \|h(\nu, \nu_0)\|_{L^2} \|\psi_{\ell m}\|_{L^1} \lesssim 2^{-J/2} 2^{-\ell/2}\|\nu-\nu_0\|_{L^2}
\end{align*}
Scaling the last estimate by a multiple of $c_{\ell J}=2^{\ell/2-J/2}(\log n)^{-1/2-\delta}$ leads to the result.
\end{proof}

Conclude from Proposition \ref{initio} and our choice of $J$ that $$\sup_{\nu \in D_{n,M}} |S_n(\nu)| \lesssim \frac{\sqrt n \|\nu-\nu_0\|_{L^2}}{2^{J} (\log n)^{1/2+\delta}} \lesssim \sqrt n n^{-(s+1/2)/(2s+1)} =O(1).$$ Simple calculations (using that (\ref{psid}) implies that $\tilde \psi_J, 2^{-J/2}\tilde \psi_J$ are uniformly bounded in $L^2, L^\infty$, respectively, proved by arguments similar to those used in Lemma \ref{approxi}) show that for $s>5/2$ the three conditions~\eqref{deltacon}, \eqref{etaL2}, \eqref{larges} and the two conditions (\ref{intc}), (\ref{deltac}) are all satisfied for such $\eta, \delta_n$ chosen as in~\eqref{etadelta} and $K$ large enough. We thus deduce from Proposition~\ref{prop:functbvm} and Lemma~\ref{postupbd} that for some sequence $C_n=O_{\mathbb P^{\mathbb N}_{\nu_0}}(1)$ and $|t|\le T$, 
\begin{align*}
&E^{\Pi^{D_{n,M}}}\left[e^{t \sqrt n \int(v-v_0) c_{\ell J}\psi_{\ell m}} |X_1, \dots, X_n \right] \\
&\qquad\qquad\le  C_n \exp \Big\{\frac{t^2}{2}\|{\tilde \psi_J}\|^2_{LAN} - \frac{t}{\sqrt n} \sum_{k=1}^n A_{\nu_0}(\tilde \psi_J)(X_k)  \Big\}.
\end{align*}
If we define $\tilde \nu_{\ell m} = -\frac{1}{n} \sum_{k=1}^n A_{\nu_0}(\tilde \psi_J)(X_k) + c_{\ell J}\int v_0 \psi_{\ell m}$ then for $|t|\le T$ this becomes the sub-Gaussian estimate
\begin{equation} \label{subgcon}
E^{\Pi^{D_{n,M}}}\left[e^{t \sqrt n \left(c_{\ell J}\int v \psi_{\ell m} - \tilde \nu_{\ell m}\right)} |X_1, \dots, X_n \right] \le C_n \exp \Big\{\frac{t^2}{2}\|\tilde \psi_J\|_{LAN}^2 \Big\}
\end{equation}
for the stochastic process $Z_{\ell,m}=(c_{\ell J}\int v \psi_{\ell m} - \tilde \nu_{\ell m}) |X_1, \dots, X_n$ conditional on $X_1, \dots, X_n$, with constants $\eta,t$ uniform.
We can then decompose $$\sqrt n c_{\ell J} \left|\langle v-v_0, \psi_{\ell m} \rangle\right| \le \sqrt n |Z_{\ell,m}| + \bigg| \frac{1}{\sqrt n} \sum_{k=1}^n A_{\nu_0}((\tilde \psi_{\ell m})_J)(X_k)\bigg|,$$ and the maximum over $2^J$ many variables in (\ref{wavred}) can now be estimated by the sum of the maxima of each of the preceding processes. For the first process we observe that the sub-Gaussian constants are uniformly bounded through 
\begin{equation} \label{subba}
\|\tilde \psi_J\|_{LAN}^2= \|A_{\nu_0}(\tilde \psi_J)\|_{L^2(\mathbb P_{\nu_0})}^2 \lesssim \|\tilde \psi_J\|_{L^2(I)}^2 \le \|\tilde \psi\|_{L^2(I)} \lesssim \|\psi_{\ell m}\|_{L^2(I)}^2 \lesssim 1,
\end{equation} 
using Lemma~\ref{sup}, that $\nu_0 \in L^\infty$ is bounded away from zero, that $P_{V_J}$ is a $L^2$-projector, combined with standard convolution inequalities. Using the sub-Gaussian estimate for $|t| \le T$, the display in the proof of Lemma 2.3.4 in \cite{GineNickl2016} yields that this maximum has expectation of order at most $O(J)$ with $\mathbb P_{\nu_0}^\mathbb N$-probability as close to one as desired. To the maximum of the second (empirical) process we apply Lemma 3.5.12 in \cite{GineNickl2016} (and again Lemma~\ref{sup} combined with the inequality in the previous display and also that $\|g\|_{\infty} \lesssim 2^{J/2}\|g\|_{L^2}$ for any $g \in V_J$) to see that its $\mathbb P^\mathbb N_{\nu_0}$-expectation is of order $O(\sqrt J  + J2^{J/2}/\sqrt n)=O(\sqrt J)$ uniformly in $\ell \le J, m$. Feeding these bounds into (\ref{wavred}) we see that on an event of $\mathbb P_{\nu_0}^\mathbb N$-probability as close to one as desired, 
\begin{equation}\label{ebdp}
E^{\Pi^{D_{n,M}}}[\|\nu - \nu_0\|_\infty|X_1, \dots, X_n] \lesssim \frac{2^{J/2} J}{\sqrt n} (\log n)^{1/2+\delta}  J \lesssim \frac{2^{J/2}}{\sqrt n}(\log n)^{5/2+\delta},
\end{equation}
Since $\delta>1/2$ was arbitrary an application of Markov's inequality completes the proof.

\subsection{Proof of Theorem \ref{bvm1}}\label{sec:bvm1}

Given results from Sections \ref{asysec}, \ref{subsec}, the proof follows ideas in \cite{CN14}. Let $\hat \nu (J)$ be the random element of $\mathcal M_0(w)$ from (\ref{cent}) with $J$ chosen as in \eqref{jprior}. For $D_{n,M}$ as in (\ref{dn}) let $\Pi^{D_{n,M}}(\cdot|X_n, \dots, X_n)$ be as before (\ref{tvlim}), and suppose $\nu \sim\Pi^{D_{n,M}}(\cdot|X_1,\dots,X_n)$. In view of (\ref{tvlim}), and since the total variation distance dominates the metric $\beta_{\mathcal M_0(\omega)}$, it suffices to prove the result for $\Pi^{D_{n,M}}(\cdot|X_1, \dots, X_n)$ replacing $\Pi(\cdot|X_1, \dots, X_n)$.  Let~$\tilde\Pi_n$ denote the laws of $\sqrt n (\nu-\hat\nu(J))$ conditionally on $X_1,\dots,X_n$ and let $\mathcal N_{\nu_0}$ be the Gaussian probability measure on $\mathcal M_0(w)$ defined (cylindrically) before Theorem \ref{bvm1}, arising from the law of $\mathbb X =(\mathbb X_{l,k})$.  The following norm estimate is the main step to establish tightness of the process $Z$ in $\mathcal M_0(\omega)$.
\begin{lemma}
For any monotone increasing sequence $\bar w=(\bar w_l)$, $\bar
w_l/l^4 \ge 1$, if $Z$ equals either $\mathbb X$ or the process $\sqrt n (\nu-\hat \nu(J))|X_1, \dots, X_n$, then for some fixed constant $C>0$ we have
\begin{eqnarray}\label{inexp}
E \bigl[\|Z\|_{\mathcal M_0(\bar w)}\bigr] &=& E \Bigl[\sup_l \bar w_l^{-1} \max_{k}|Z_{l,k}| \Bigr] \leq C,
\end{eqnarray}
where in case $Z=\sqrt n (\nu-\hat \nu(J))|X_1, \dots, X_n$ the operator $E$ denotes conditional expectation $E^{D_{n,M}}[\cdot|X_1, \dots, X_n]$ and the inequality holds with $\mathbb P^\mathbb N_{\nu_0}$-probability as close to one as desired. 
\end{lemma}
\begin{proof}
We first consider the more difficult case where $Z$ is the centred and scaled posterior process. We decompose, with $\nu_J=P_{V_J}(\nu)$, 
$$\sqrt n (\nu- \hat \nu (J)) = \sqrt n (\nu_J - \hat \nu (J)) + \sqrt n (\nu_0-\nu_{0,J}) + \sqrt n [(\nu-\nu_0) - (\nu-\nu_0)_J] .$$
The second term on the right hand side  has multi-scale norm $\|\nu_0-\nu_{0,J}\|_{\mathcal M(w)}$ bounded by $2^{-J(s+1/2)} w_J^{-1} = o(1/\sqrt n)$ in view of (\ref{wavprop}), $\|\psi_{lk}\|_{L^1}\lesssim 2^{-l/2}$. Similarly the expectation of the multi-scale norm of the third term is bounded by
\begin{align*}
& \int\|\nu-\nu_0-(\nu-\nu_0)_J\|_{\mathcal M(w)} d\Pi^{D_{n,M}}(\nu|X_1, \dots, X_n) \\
&= \int \sup_{l>J} w_l^{-1} \max_k |\langle \nu-\nu_0, \psi_{lk} \rangle| d\Pi^{D_{n,M}}(\nu|X_1, \dots, X_n) \\
& \le w_{J}^{-1} \sup_{l>J}  \max_k \|\psi_{lk}\|_{L^1} \int \|\nu-\nu_0\|_\infty d\Pi^{D_{n,M}}(\nu|X_1, \dots, X_n)  \\
&\lesssim \frac{2^{-J/2} 2^{J/2} }{J^4 \sqrt n}\log ^{5/2+\delta} n = o_{\mathbb P^\mathbb N_{\nu_0}}(1/\sqrt n),
\end{align*}
using (\ref{ebdp}). We turn to bounding the multi-scale norm of the first term, corresponding to
$$\sqrt n\|\nu_J - \hat \nu (J)\|_{\mathcal M(w)} = \sqrt n \sup_{l< J}w_{l}^{-1} \max_k \left|\int_I \nu \psi_{lk}  - \hat \nu(J)_{lk}\right|.$$
The first term in the decomposition 
\begin{equation}\label{auxdec}
\int_I  \nu \psi_{lk}  - \hat \nu(J)_{lk} = \int_I (\nu-\nu_0) \psi_{lk} - \Big(\hat \nu(J)_{lk} -\int_I \nu_0 \psi_{lk}\Big) \equiv  \int_I (\nu-\nu_0) \psi_{lk} - W_{lk} 
\end{equation}
equals
\begin{equation} \label{auxrem}
\int_I (\nu-\nu_0)\psi_{lk} = \int_I(e^v-e^{v_0})\psi_{lk} = \int_I (v-v_0) \nu_0 \psi_{lk} + O(\|\nu-\nu_0\|_{\infty}^2),
\end{equation}
and the quadratic remainder is of order $o(1/\sqrt n)$ uniformly in $k,l$ by definition of $D_{n,M}$ and since $s>5/2$. 

\begin{lemma} \label{approxii}
Let $\psi = \nu_0\psi_{lk}1_{I\setminus\{0\}}$ for some $l< J, k$ with corresponding $\tilde \psi=(\tilde \psi)_{lk}$ from (\ref{tpsi}), (\ref{psioh}) and wavelet approximation $\tilde \psi_J \in V_J$. We have $$\left|\int_I A_{\nu_0}(v-v_0) A_{\nu_0}(\tilde \psi_J)d\mathbb P_{\nu_0} + \int_I (v-v_0) \nu_0 \psi_{l k}\right| \lesssim  \|\nu-\nu_0\|_{\infty} 2^{-J}.$$
\end{lemma}
\begin{proof}
The proof requires only notational adaptation of the proof of Lemma~\ref{approxi}, except for the last display, where now we use Lemma~\ref{sup} (and its variant for $A_\nu^*$) in the estimate $|\langle h(\nu, \nu_0), \psi_{lk}\rangle| \leq \|h(\nu, \nu_0)\|_\infty \|\psi_{l k}\|_{L^1} \lesssim 2^{-l/2}\|\nu-\nu_0\|_\infty$ so that scaling by $c_{\ell J}$ is not necessary.
\end{proof}
The upper bound in the display of Lemma~\ref{approxii} has $E^{D_{n,M}}[\cdot|X_1, \dots, X_n]$-expectation of order $o(1/\sqrt n)$ in view of (\ref{ebdp}). We now apply Proposition \ref{prop:functbvm} to the functional 
\begin{equation}\label{eff}
F(\nu) \equiv F_{l k}(\nu)=-\int_I A_{\nu_0}(v-v_0) A_{\nu_0}(\tilde \psi_J)d\mathbb P_{\nu_0},
\end{equation}
with choices $\delta_n= K2^J(J^2+1)/\sqrt{n}$ for $K>0$ a large enough constant and $\eta=\tilde \psi_J$. Simple calculations (using that $\tilde \psi_J, 2^{-J/2}\tilde \psi_J$ are uniformly bounded in $L^2, L^\infty$, respectively) show that for $s>5/2$ the three conditions~\eqref{deltacon}, \eqref{etaL2}, \eqref{larges} and the two conditions~(\ref{intc}), (\ref{deltac}) are all satisfied. Conclude from Proposition~\ref{prop:functbvm} and Lemma \ref{postupbd}  that
$$E^{\Pi^{D_{n,M}}}\left[e^{t \sqrt n F(\nu)} |X_1, \dots, X_n \right] \le  C_n  \exp \bigg\{\frac{t^2}{2}\|\tilde \psi_J\|_{LAN}^2 - \frac{t}{\sqrt n} \sum_{k=1}^n A_{\nu_0}(\tilde \psi_J)(X_k)   \bigg\}$$
for $|t|\le T$, or equivalently, if $V_{lk} =\frac{1}{n} \sum_{k=1}^n A_{\nu_0}(\tilde \psi_J)(X_k)$, then for some $C_n'=O_{\mathbb P^{\mathbb N}_{\nu_0}}(1)$,
\begin{equation}\label{subbi}
E^{\Pi^{D_{n,M}}}\left[e^{t \sqrt n F(\nu) + t\sqrt n V_{lk}} |X_1, \dots, X_n \right] \le   C_n'  \exp \left\{\frac{t^2}{2}\|\tilde \psi_J\|_{LAN}^2\right\}.
\end{equation}
Arguing just as in (\ref{subba}) the sub-Gaussian constants $\|\tilde \psi_J\|_{LAN}^2$ are bounded by a fixed constant. We then have, for $M$ a fixed constant and using $w_l \ge  l$,
\begin{align*}&E^{\Pi^{D_{n,M}}}\left[\left. \sup_{l<J}w_{l}^{-1} \max_k \left|\sqrt n F_{lk}(\nu) + \sqrt n V_{lk}\right|\right|X_1, \dots, X_n\right] \\
&\le M+\int_M^\infty \Pi^{D_{n,M}} \left(\left. \sup_{l<J}l^{-1} \max_k \left|\sqrt n F_{lk}(\nu) + \sqrt n V_{lk}\right|>u\right|X_1, \dots, X_n\right)du\end{align*}
We bound the tail integrals using \eqref{subbi} as follows:
\begin{align*}
&\sum_{l<J,k}\int_M^\infty \Pi^{D_{n,M}} \left(  \left|\sqrt n F_{lk}(\nu) + \sqrt n V_{lk}\right|>lu|X_1, \dots, X_n\right)du\\
&\le\sum_{l<J,k}\int_M^\infty \Pi^{D_{n,M}} \left(  e^{T|\sqrt n F_{lk}(\nu) + \sqrt n V_{lk}|}>e^{Tlu}|X_1, \dots, X_n\right)du\\
&\le\sum_{l<J,k}\int_M^\infty E^{\Pi^{D_{n,M}}} \left[  e^{T|\sqrt n F_{lk}(\nu) + \sqrt n V_{lk}|}|X_1, \dots, X_n\right]e^{-Tlu}du\\
&\lesssim C_n'\sum_{l<J} 2^l\int_M^\infty e^{-Tlu}du\lesssim C_n'\sum_{l<J}2^l e^{-TMl}
= O_{\mathbb P_{\nu_0}^\mathbb N}(1)\end{align*}
for $M$ large enough. Moreover, one proves $E_{\nu_0} \sup_{l<J} w_l^{-1} \max_k |V_{lk}| \lesssim 1/\sqrt n$ and also $E_{\nu_0} \sup_{l<J} w_l^{-1} \max_k |W_{lk}|\lesssim 1/\sqrt n$ just as in the proof of Theorem 1 in \cite{CN14} (or Theorem 5.2.16 in \cite{GineNickl2016}), using Bernstein's inequality combined with the previous bound on the sub-Gaussian constants and a uniform bound of order $2^{J/2}$ (proved just as after (\ref{subba})) on the envelopes $\|A_{\nu_0}(\tilde \psi_J)\|_\infty$, $\|(A^*_{\nu_0})^{-1}(\psi_{lk}1_{\{0\}^c})\|_\infty$, $l \le J,$ of the empirical processes involved. Combining what precedes with Lemma \ref{approxii} (and the remark after it), (\ref{auxdec}), (\ref{auxrem}) proves (\ref{inexp}) for the `posterior' process. The Gaussian process  $\mathbb X$ admits by definition the same (sub-) Gaussian bound as in (\ref{subbi}) so that the result follows from the same arguments just given.
\end{proof}
The inequality (\ref{inexp}) implies in particular that for any weighting sequence $\omega$ as in Theorem~\ref{bvm1}, the processes~$Z$ concentrate in the separable subspace $\mathcal M_0(\omega)$ of $\mathcal M(\omega)$, and their laws define tight (in the case of $\mathcal N_{\nu_0}$, Gaussian) Borel probability measures in it (by Ulam's theorem, see p.225 in \cite{D02}). Then, using the estimate (\ref{inexp}) and arguing as in the proof of Proposition 6 in \cite{CN14} (or in Theorem 7.3.20 in \cite{GineNickl2016}), Theorem \ref{bvm1} will follow if we can establish convergence of the finite-dimensional distributions $\tilde \Pi_n \circ P_{V_L}^{-1}$ towards those of $\mathcal N_{\nu_0} \circ P_{V_L}^{-1}$, $L \in \mathbb N$ fixed, as $n \to \infty$, where  $P_{V_L}$ is
the projection operator onto the finite-dimensional subspace $V_L$ of $\mathcal M_0(w)$ corresponding to the first $2^L$ coordinates $(x_{lk}:l \le L , k)$. For this we proceed as in the previous lemma, combining  (\ref{auxdec}), (\ref{auxrem}) with Lemma \ref{approxii} and the definition of $W_{lk}$, to reduce the problem to showing for $\nu \sim \Pi^{D_{n,M}}(\cdot|X_1, \dots, X_n)$ weak convergence in probability of the conditional laws of $$Y_n \equiv -\sqrt n \int_I A_{\nu_0}(v-v_0) A_{\nu_0}(\tilde \psi_J)d\mathbb P_{\nu_0} - \frac{1}{\sqrt n} \sum_{i=1}^n (A^*_{\nu_0})^{-1}(\psi_{lk}1_{\{0\}^c})(X_i),$$ to the law of $\mathcal N_{\nu_0}$ for every \textit{fixed} $k, l \le L \in \mathbb N$. Applying Proposition \ref{prop:functbvm} as after (\ref{eff}) combined with Lemma \ref{postconv} (for $k,l$ fixed the corresponding $\tilde \psi_J$'s are bounded in $L^\infty$) gives convergence of $Z_n$ in Proposition \ref{prop:functbvm} to one and hence one has, as $n \to \infty$ and for all $t$,
$$E^{\Pi^{D_{n,M}}}\left[e^{t Y_n} |X_1, \dots, X_n \right] = (1+o_{\mathbb P^{\mathbb N}_{\nu_0}}(1))  \exp \left\{\frac{t^2}{2}\|A_{\nu_0}(\tilde \psi_J)\|_{L^2(\mathbb P_{\nu_0})}^2\right\} \exp (t\rho_n)$$
where $$\rho_n =  -\frac{1}{\sqrt n} \sum_{i=1}^n (A^*_{\nu_0})^{-1}(\psi_{lk}1_{\{0\}^c})(X_i) -  \frac{1}{\sqrt n} \sum_{i=1}^n A_{\nu_0}(\tilde \psi_J)(X_i).$$
Using Lemma \ref{scoinv}, (\ref{tpsi}), $A_{\nu}(\tilde \psi_d - \tilde \psi)=0$ by (\ref{akernel}) and (\ref{psioh}), and then also Lemma~\ref{sup} combined with 
$\tilde \psi \in L^2$ one has 
\begin{align*}
\|A_{\nu_0}(\tilde \psi_J) + (A^*_{\nu_0})^{-1}(\psi_{lk}1_{\{0\}^c})\|_{L^2(\mathbb P_{\nu_0})}&=\|A_{\nu_0}(\tilde \psi_J) - A_{\nu_0}(\tilde \psi)\|_{L^2(\mathbb P_{\nu_0})} \\
&\lesssim \|\tilde \psi_J- \tilde \psi\|_{L^2(I)} \to 0
\end{align*}
as $J \to \infty$, in particular by Chebyshev's inequality $\rho_n=o_{\mathbb P^{\mathbb N}_{\nu_0}}(1)$ for every fixed $l \le L, k$. Thus the Laplace-transforms of each such coordinate projection  converge to the Laplace transform of the correct normal limit distribution, for all $t$,
$$E^{\Pi^{D_{n,M}}}\left[e^{t Y_n} |X_1, \dots, X_n \right] =  (1+o_{\mathbb P^{\mathbb N}_{\nu_0}}(1)) \times \exp \left\{\frac{t^2}{2}\|(A^*_{\nu_0})^{-1}(\psi_{lk}1_{\{0\}^c})\|_{L^2(\mathbb P_{\nu_0})}^2\right\},$$ and convergence in distribution now follows from standard arguments (see, e.g., Proposition 29 in \cite{N17}). This argument extends directly to all linear combinations $\sum_{l \le L, k} a_{l,k}\psi_{l k}$, so that we can apply the Cramer--Wold device to obtain joint convergence in $V_L$ for any $L \in \mathbb N$. The proof is complete.

\section{Proof of Proposition \ref{initio}} \label{prelimc}

We first derive a general contraction theorem from which we will deduce Proposition \ref{initio} (after Proposition \ref{propball}). We follow the usual `testing and small ball probability approach' (as in Theorem 7.3.1 in \cite{GineNickl2016}, see also \cite{GvdV17}), which in our setting gives the following starting point to prove contraction rates, where $K(\mathbb P_\nu, \mathbb P_{\nu'})$ denotes the usual Kullback--Leibler (KL-) divergence between two probability measures $\mathbb P_\nu, \mathbb P_{\nu'}$.

\begin{proposition} \label{cont0}
Consider a prior $\Pi$ on a $\sigma$-field $\mathcal S_V$ of some set $\mathcal V$ of L\'evy measures for which the map $(\nu,x) \mapsto p_\nu(x),$ defined before (\ref{likelihood}) is jointly measurable. Let $d$ be some metric on $\mathcal V$ such that $\nu \mapsto d(\nu,\nu')$ is measurable for all $\nu' \in \mathcal V$. Suppose for some sequence $\eps_n \to 0$ such that $\sqrt n \eps_n \to \infty$, constant $C>0$ and $n$ large enough we have
$$\Pi\left(\nu \in \mathcal V: K(\mathbb P_{\nu_0}, \mathbb P_\nu) \le \eps^2_n, \Var_{\mathbb P_{\nu_0}}\big(\log \frac{d\mathbb P_{\nu}}{d\mathbb P_{\nu_0}}\big) \le \eps^2_n\right) \ge e^{-Cn\eps_n^2}$$ and that for $\mathcal V_n \subset \mathcal V$ such that
$\Pi(\mathcal V \setminus \mathcal V_n) \le L e^{-(C+4)n \eps_n^2}$ we can find tests $\Psi_n=\Psi(X_1, \dots, X_n)$ and $\delta_n>0, M_0>0,$ such that 
$$\mathbb E_{\nu_0} \Psi_n  \to 0,~~~ \sup_{\nu \in \mathcal V_n,\, d(\nu, \nu_0) \ge M_0 \delta_n} \mathbb E_\nu(1-\Psi_n) \le L e^{-(C+4)n \eps_n^2}.$$ Then if $\Pi(\cdot|X_1, \dots, X_n)$ is the posterior distribution from (\ref{post}) we have, for every $M \ge M_0$, $$\Pi(\nu: d(\nu, \nu_0) \ge M \delta_n|X_1, \dots, X_n) \to 0$$ as $n \to \infty$ in $\mathbb P_{\nu_0}^\mathbb N$-probability.
\end{proposition}

As in previously studied `inverse problems' settings \cite{R13, NS17, N17}, to apply this proposition with a metric~$d$ different from the Hellinger distance $h(\mathbb P_\nu, \mathbb P_{\nu_0})$ requires new approaches to the construction of frequentist tests, and as in these references we use tools from `concentration of measure' theory put forward in \cite{GN11}, where we initially choose for~$d$ the weak (or `robust') metric induced by the norm $\|\cdot\|_{\mathbb H(\delta)}$ of
\begin{equation} \label{hdelta}
\mathbb H(\delta) = \bigg\{f: \|f\|^2_{\mathbb H(\delta)}=\sum_{l,k} 2^{-l} l^{-2\delta} \langle f, \psi_{lk}\rangle^2<\infty\bigg\}, \quad \delta>1/2,
\end{equation}
a negative order Sobolev space. Contraction rates in stronger norms will then be deduced from interpolation arguments.  Before doing so, however, we need to calculate KL-divergences for the observation scheme relevant in our context, and show that they can be bounded in terms of the distance of their L\'evy measures.

\begin{lemma}\label{L2bound}
Let $D>0$ such that $e^{-D}\le { d \nu}/{ d\Lambda}\le e^D$ and $e^{-D}\le { d\nu_0}/{ d \Lambda}\le e^D$ on~$I$. Then there exists $K_D>0$ such that
\begin{align*}
&K(\mathbb P_{\nu_0}, \mathbb P_\nu) = \int_I \log \frac{d\mathbb P_{\nu_0}}{d\mathbb P_\nu} d\mathbb P_{\nu_0} \le K_D \|\nu-\nu_0\|_{L^2}^2,\\
&\Var_{\mathbb P_{\nu_0}}\Big(\log \frac{d\mathbb P_{\nu}}{d\mathbb P_{\nu_0}}\Big) \le \int_I \Big(\log \frac{d\mathbb P_{\nu}}{d\mathbb P_{\nu_0}}\Big)^2 d\mathbb P_{\nu_0} \le K_D \|\nu-\nu_0\|_{L^2}^2.
\end{align*}
\end{lemma}

\begin{proof}
We define the path $s\mapsto\exp(s(v-v_0)+v_0)=\nu^{(s)}$, $s\in[0,1]$, from $\nu_0$ to $\nu$ and consider the function $f(s)= \int \log (d\mathbb P_{\nu^{(s)}}/d\mathbb P_{\nu_0}) d\mathbb P_{\nu_0}.$ Observing $f(0)=0$ a Taylor expansion at $s=0$ yields some $s\in[0,1]$ such that $f(1)=f'(0)+\tfrac{1}{2}f''(s)$. By the upper and lower bounds on the L\'evy densities the differentiation may be performed under the integral and we obtain
\begin{align*}
&\int \log \frac{d\mathbb P_{\nu_0}}{d\mathbb P_\nu} d\mathbb P_{\nu_0}
=-\int \frac{ d \frac{ d}{ d s}\PP_{\nu^{(s)}}}{ d \PP_{\nu^{(s)}}}\bigg|_{s=0}
+\frac{1}{2}\frac{ d \frac{ d^2}{ d s^2}\PP_{\nu^{(s)}}}{ d \PP_{\nu^{(s)}}}-\frac{1}{2}\Big(\frac{ d \frac{ d}{ d s}\PP_{\nu^{(s)}}}{ d \PP_{\nu^{(s)}}}\Big)^2 d\mathbb P_{\nu_0}\displaybreak[0]\\
&\qquad=- \int A_{\nu_0}(v-v_0) d\mathbb P_{\nu_0}\\
&\qquad\quad-\frac{1}{2} \int A_{\nu^{(s)}}((v-v_0)^2)+A_{\nu^{(s)}}(v-v_0,v-v_0)-(A_{\nu^{(s)}}(v-v_0))^2 d\mathbb P_{\nu_0}\displaybreak[0]\\
&\qquad=-\frac{1}{2} \int A_{\nu^{(s)}}((v-v_0)^2)+A_{\nu^{(s)}}(v-v_0,v-v_0)-(A_{\nu^{(s)}}(v-v_0))^2 d\mathbb P_{\nu_0}\displaybreak[0]\\
&\qquad\lesssim \| A_{\nu^{(s)}}((v-v_0)^2)\|_{L^1(\PP_{\nu^{(s)}})}+\|A_{\nu^{(s)}}(v-v_0,v-v_0)\|_{L^1(\PP_{\nu^{(s)}})}\\
&\qquad\quad+\|A_{\nu^{(s)}}(v-v_0)\|^2_{L^2(\PP_{\nu^{(s)}})},
\end{align*}
where the last step contains a change of measure from $\PP_{\nu_0}$ to $\PP_{\nu^{(s)}}$ such that we may now apply Lemma~\ref{sup}
\begin{align*}
 \int \log \frac{d\mathbb P_{\nu_0}}{d\mathbb P_\nu} d\mathbb P_{\nu_0}
&\lesssim \| (v-v_0)^2\|_{L^1({\nu^{(s)}})}+\|v-v_0\|^2_{L^1({\nu^{(s)}})}+\|v-v_0\|^2_{L^2({\nu^{(s)}})}\\
&\lesssim \|v-v_0\|^2_{L^2({\nu^{(s)}})}\lesssim \|v-v_0\|^2_{L^2}\lesssim \|\nu-\nu_0\|^2_{L^2}.
\end{align*}
For the second inequality we consider the folllowing function $g$ and its derivatives
\begin{align*}
g(s)&=\int \Big(\log \frac{d\mathbb P_{\nu^{(s)}}}{d\mathbb P_{\nu_0}}\Big)^2  d\mathbb P_{\nu_0},\\
g'(s)&=\int 2\Big(\log \frac{d\mathbb P_{\nu^{(s)}}}{d\mathbb P_{\nu_0}}\Big) \frac{d\tfrac{d}{ds}\mathbb P_{\nu^{(s)}}}{d\mathbb P_{\nu^{(s)}}}d\mathbb P_{\nu_0},\\
g''(s)&=\int 2\Big(\log \frac{d\mathbb P_{\nu^{(s)}}}{d\mathbb P_{\nu_0}}\Big) 
\bigg(\Big(\frac{d\tfrac{d}{ds}\mathbb P_{\nu^{(s)}}}{d\mathbb P_{\nu^{(s)}}}\Big)^2
+\frac{d\tfrac{d^2}{ds^2}\mathbb P_{\nu^{(s)}}}{d\mathbb P_{\nu^{(s)}}}
-\Big(\frac{d\tfrac{d}{ds}\mathbb P_{\nu^{(s)}}}{d\mathbb P_{\nu^{(s)}}}\Big)^2\bigg)
d\mathbb P_{\nu_0}\\
&=\int 2\Big(\log \frac{d\mathbb P_{\nu^{(s)}}}{d\mathbb P_{\nu_0}}\Big) 
\frac{d\tfrac{d^2}{ds^2}\mathbb P_{\nu^{(s)}}}{d\mathbb P_{\nu^{(s)}}}
d\mathbb P_{\nu_0}.
\end{align*}
Observing $g(0)=g'(0)=0$ we obtain by a Taylor expansion $g(1)=g''(s)$ for some $s\in[0,1]$ and thus
\begin{align*}
&\int \Big(\log \frac{d\mathbb P_{\nu}}{d\mathbb P_{\nu_0}}\Big)^2  d\mathbb P_{\nu_0}
=\int 2\Big(\log \frac{d\mathbb P_{\nu^{(s)}}}{d\mathbb P_{\nu_0}}\Big) \frac{d\tfrac{d^2}{ds^2}\mathbb P_{\nu^{(s)}}}{d\mathbb P_{\nu^{(s)}}}d\mathbb P_{\nu_0}
\lesssim \int \Big|\frac{d\tfrac{d^2}{ds^2}\mathbb P_{\nu^{(s)}}}{d\mathbb P_{\nu^{(s)}}}\Big|d\mathbb P_{\nu^{(s)}}\\
&\qquad\qquad\lesssim \| A_{\nu^{(s)}}((v-v_0)^2)\|_{L^1(\PP_{\nu^{(s)}})}+\|A_{\nu^{(s)}}(v-v_0,v-v_0)\|_{L^1(\PP_{\nu^{(s)}})}\\
&\qquad\qquad\lesssim \| (v-v_0)^2\|_{L^1({\nu^{(s)}})}+\|v-v_0\|^2_{L^1({\nu^{(s)}})}\\
&\qquad\qquad\lesssim \| v-v_0\|^2_{L^2}+\| v-v_0\|^2_{L^1} \lesssim \| v-v_0\|^2_{L^2} \lesssim \| \nu-\nu_0\|^2_{L^2}.
\end{align*}
\end{proof}

\begin{assumption}\label{id}
The intensity $\lambda$ of $\nu$ satisfies $\lambda<\pi/\Delta$.
\end{assumption}

For L\'evy processes on $\mathbb R$ the L\'evy measure can be identified by taking the complex logarithm of the  characteristic function of $\mathbb P_\nu$ in such a way that the resulting function is continuous. (This is known as the distinguished logarithm.) For L\'evy processes on a circle the characteristic function is defined only on the integer lattice and a continuous version of the logarithm cannot be defined. However, this problem can be resolved by assuming $\lambda<\pi/\Delta$ since then the exponent in the L\'evy-Khintchine representation always coincides with the principle branch of the logarithm of the characteristic function, ensuring identifiability. This condition is sharp as the following examples show.

\begin{examples}
By the L\'evy--Khintchine representation \eqref{LK} we see that $\PP_{\nu_1}$ and $\PP_{\nu_2}$ coincide if $\F \nu_1(k)$ equals $\F \nu_2(k)$ modulo multiples of $2\pi i/\Delta$ for all $k\in\Z$.
\begin{enumerate}
\item For $\nu_1=(\pi/\Delta)\delta_{1/4}$ and $\nu_2=(\pi/\Delta)\delta_{-1/4}$ we have $\F \nu_1(k)=\F\nu_2(k)$ for all even $k$ and $\F \nu_1(k)=\F\nu_2(k)+(2\pi/\Delta) i$ or $\F \nu_1(k)=\F\nu_2(k)-(2\pi/\Delta) i$ for all odd $k$. This shows that the intensity bound in Assumption~\ref{id} is sharp.
\item For $\nu_1(x)=(4\pi/\Delta)(\sin(2\pi x))_+$ and $\nu_2(x)=(4\pi/\Delta)(\sin(2\pi x))_-$ we have $\F \nu_1(1)=\F\nu_2(1)+(2\pi/\Delta) i$ and $\F \nu_1(-1)=\F\nu_2(-1)-(2\pi/\Delta) i$. For all other $k$ it can be shown that $\F \nu_1(k)=\F\nu_2(k)$. This demonstrates that there exist nonidentifiable L\'evy measures which are absolutely continuous with respect to Lebesgue measure.
\end{enumerate}
\end{examples}

\begin{lemma}\label{concentration}
For any $c, x, D>0, \delta>1/2$, and integer $K\ge2$, there exist constants $R_1(c,D,\Delta)>0$, $R_2(c,D,\Delta)>0$ and an estimator $\hat\nu=\hat \nu(X_1, \dots, X_n)$ such that
\begin{align}
\sup_{\nu: \|\nu\|_{L^1} < \pi/\Delta, \|\nu\|_{L^2} \le c} \mathbb P^\mathbb N_\nu \left(\|\hat \nu - \nu\|_{\mathbb H(\delta)} > R_1\bigg( \frac{\sqrt {\log K} + x}{\sqrt n} + \frac{1}{\sqrt K} \bigg) \right) \nonumber\\
\leq  e^{-Dx^2}+\frac{e^{-\frac{n R_2}{\log K}}}{R_2} .\label{freqrat}
\end{align}
\end{lemma}

\begin{proof}
We first show the above concentration inequality with $\|\hat \nu - \nu\|_{\mathbb H(\delta)}$ replaced by $|\hat\lambda-\lambda|$, where $\lambda= \int_I \nu = (\mathcal F \nu)(0)$ is the intensity and $\hat\lambda$ is an estimator defined as follows: Let $\phi_n(k) = (1/n) \sum_{j=1}^n \exp\{2\pi i k X_j\}$ be the empirical characteristic function, set $\Phi_n(k) = \Delta^{-1} \log \phi_n(k)$ for $\phi_n(k)\neq0$ and $\Phi_n(k)=0$ otherwise, where we take the principal branch of the complex logarithm. For $K \ge 2$ consider the estimator $\hat \lambda = -(1/K) \sum_{k=1}^K \Re \Phi_n(k).$ The L\'evy--Khintchine representation (\ref{LK}) yields $\Phi_\nu (k) :=\Delta^{-1} \log \phi_{\nu}(k) = \mathcal F \nu (k) - \lambda,$ where thanks to the restriction $\|\nu\|_{L^1} < \pi/\Delta$ the imaginary part on the r.h.s.~lies in $(-\pi/\Delta, \pi/\Delta)$ and hence $\log$ is the logarithm in the principle branch. We obtain 
\begin{align}
\hat \lambda - \lambda &= -\frac{1}{K} \sum_{k=1}^K \Re  (\Phi_n(k)- \Phi_\nu(k)) - \frac{1}{K} \sum_{k=1}^K (\Re  \Phi_\nu(k) + \lambda ) \nonumber\\
&= -\frac{1}{K} \sum_{k=1}^K \Re  (\Phi_n(k)- \Phi_\nu(k)) - \frac{1}{K} \sum_{k=1}^K \Re  \mathcal F\nu(k) \label{errorLambda}
\end{align} 
In order to linearise the first term in previous equation we define the event
\[
A_n=\left\{\left\|\frac{\phi_n-\phi_\nu}{\phi_\nu}\right\|_{K}\le \frac{1}{2}\right\}
\qquad \text{ with } \|f\|_{K}=\sup_{|k|\le K}|f(k)|.
\]
It holds $|\log(1+z)-z|\le 2|z|^2$ for $|z|\le 1/2$. Thus we have on the event $A_n$ for $|k|\le K$
\begin{align*}
\Phi_n(k)-\Phi_\nu(k)&=\frac{1}{\Delta}\log\left(\frac{\phi_n(k)-\phi_\nu(k)}{\phi_\nu(k)}+1\right)\displaybreak[0]\\
&=\frac1{\Delta}\left\{\frac{\phi_n(k)-\phi_\nu(k)}{\phi_\nu(k)}+O\bigg(\Big|\frac{\phi_n(k)-\phi_\nu(k)}{\phi_\nu(k)}\Big|^2\bigg)\right\}.
\end{align*}

The first term in \eqref{errorLambda}, up to linearisation, is purely stochastic and bounded by a term of the form $$\frac{1}{\Delta K} \sum_{k=1}^K \frac{|\phi_n(k)-\phi_{\nu}(k)|}{|\phi_{\nu}(k)|}.$$ Since $\|\nu\|_1 <\pi/\Delta$ we know that $\sup_k |1/\phi_{\nu}(k)| \le c'$ for some constant $c'=c'(\Delta)$. 
For the numerator we consider the $4K+4$ random variables 
\begin{align*}
&\pm\Re(\phi_n(-K)-\phi_\nu(-K)),\dots,\pm\Re(\phi_n(K)-\phi_\nu(K)),\\
&\pm\Im(\phi_n(-K)-\phi_\nu(-K)),\dots,\pm\Im(\phi_n(K)-\phi_\nu(K))
\end{align*}
and denote them by $Z_j$ with $j=1,\dots,4K+4$. These have bounded differences with constant $c^2=4/n$ which follows from using example b) before  Theorem~3.3.14 in \cite{GineNickl2016} and observing that $e^{2\pi i k (\cdot)}$ are uniformly bounded by 1. Applying this theorem we have $ E e^{\lambda Z_j}\le e^{\lambda^2 c^2 /8}=e^{\lambda^2/(2n)}$. By Lemma~2.3.4 in \cite{GineNickl2016} we further obtain that 
\[
 E\left[\max_{j=1,\dots,4K+4}Z_j\right]\le \sqrt{\frac2{n}\log(4K+4)}
\]
and denoting $Z=\max_{|k| \le K} |\phi_n(k)-\phi_{\nu}(k)|$ we have
\begin{align*}
 E[Z]&\le 2 E\left[\max_{|k|\le K}\left(\Re(\phi_n(k)-\phi_\nu(k)),\Im(\phi_n(k)-\phi_\nu(k))\right)\right] \\
&\le\sqrt{\frac{8}{n}\log(4K+4)}
\lesssim\sqrt{\frac{\log K}{n}}.
\end{align*}
For the concentration around the mean we observe that $Z$ itself also has bounded differences with $c^2=4/n$ and applying Theorem 3.3.14 in \cite{GineNickl2016} yields 
\begin{align*}
\PP(Z\ge  E Z +t)&\le e^{-2 t^2/c^2}=e^{-nt^2/2},~~\PP(Z\le  E Z -t)\le e^{-nt^2/2}.
\end{align*}
This shows that the linearisation of the first term in~\eqref{errorLambda} is bounded by a multiple of $(\sqrt{\log K}+x)/\sqrt{n}$.
On $A_n$ we can bound the remainder in the linearisation by a multiple of the same quantity. For $n/\log K$ large enough $ E Z$ is smaller than $1/(4c')$ and we can bound $\PP(A_n^c)$ by $\exp(-R_2 n)\le \exp(-R_2 n/\log K)$ using the concentration of $Z$. The bound $\PP(A_n^c)\le (1/R_2)\exp(-R_2 n/\log K)$ for all $n$ and $K$ is obtained by choosing a possibly smaller constant~$R_2$.

For the bias we bound, using the Cauchy--Schwarz inequality,
$$\left|\frac{1}{K} \sum_{k=1}^K \Re  \mathcal F\nu(k)\right| \le K^{-1/2} \sqrt {\sum_{k=1}^K |\mathcal F \nu(k)|^2} \leq \frac{\|\nu\|_{L^2}}{\sqrt K},$$ which explains the second regime in the inequality in Lemma \ref{concentration}. 

Now to estimate $\nu$ we first estimate $\mathcal F \nu(k), k \neq 0,$ by
$\mathcal F \hat \nu(k) = (\Phi_n(k) + \hat \lambda) 1_{[-K,K]}(k)$, where~$K$ is a spectral cut-off parameter. By standard theory of Sobolev spaces on the unit circle, an equivalent norm on $\mathbb H(\delta)$ is given by $$\|f\|_{\mathbb H(\delta)}'=\sum_k |\mathcal Ff(k)|^2 k^{-1} (\log(e+k))^{-2\delta}.$$ Using that $\sum_{k} k^{-1} (\log (e+k))^{-2\delta} $ converges for $\delta>1/2$ we obtain
\begin{align*}
\|\hat \nu - \nu\|_{\mathbb H(\delta)}^2 
&= \sum_{k} k^{-1} (\log (e+k))^{-2\delta} |\mathcal F \hat \nu(k) -  \mathcal F \nu(k)|^2 \displaybreak[0]\\
&= \sum_{|k|\le K} k^{-1} (\log (e+k))^{-2\delta} |\Phi_n(k)   - \Phi_\nu(k) + \hat \lambda - \lambda|^2 \\
&\qquad\qquad\qquad+\sum_{|k|> K} k^{-1} (\log k)^{-2\delta} |\mathcal F \nu(k)|^2\displaybreak[0]\\
&\lesssim (\hat \lambda - \lambda)^2 + \sum_{|k|\le K} k^{-1} (\log (e+ k))^{-2\delta} |\Phi_n(k) - \Phi_\nu(k)|^2 \\
&\qquad\qquad\qquad+ \sum_{|k|>K} k^{-1} (\log (e+k))^{-2\delta} |\mathcal F \nu(k)|^2  \\
&\lesssim (\hat \lambda - \lambda)^2 + \max_{|k|\le K} |\Phi_n(k) - \Phi_\nu(k)|^2 +\|\nu\|^2_{L^2}/K,
\end{align*}
which, repeating the above, gives the same bounds as those obtained for error of the intensity $\hat \lambda-\lambda$.
\end{proof}

The proof of the following proposition is contained in Section~\ref{sec:tests}.

\begin{proposition}\label{prop:tests}
Denote $\overline{\mathcal V}=\{\nu\in\mathcal V:\|\nu\|_{L^1}<\pi/\Delta \text{ and }\|\nu\|_{L^2}\le c\}$ for some $c,\Delta>0$.
Let~$\eps_n$ be such that $\sqrt{(\log n)/n}\lesssim \eps_n$ and $\eps_n=o( 1/\sqrt{\log n})$. Then for $\nu_0\in\overline{\mathcal V}$ there exists a sequence of tests (indicator functions) $\Psi_n\equiv \Psi(X_1, \dots, X_n)$ such that for every $C>0$, there exist $M=M(C,c,\Delta)>0$ such that for all $n$ large enough
\[
 E_{\nu_0}[\Psi_n]\to_{n\to\infty}0, \qquad 
\sup_{\nu\in\overline{\mathcal V}:\|\nu-\nu_0\|_{\mathbb H(\delta)}\ge M\eps_n} E_{\nu}[1-\Psi_n]\le 2 e^{-(C+4)n\eps_n^2}.
\]
\end{proposition}

\begin{proposition} \label{better}
Suppose we have for some constants $c,C,D>0$, for a sequence $\eps_n$ such that $\sqrt{(\log n)/n}\lesssim \eps_n$ and $\eps_n=o( 1/\sqrt{\log n})$, for $\nu_0$ such that $e^{-D}\le  d\nu_0/ d\Lambda \le e^D$, for some prior $\Pi$ on a set $\{\nu\in\mathcal V: e^{-D}\le  d\nu/ d\Lambda \le e^D\}$ of L\'evy measures bounded from above and away from zero, for $n$ large enough and with $K_D$ from Lemma~\ref{L2bound} that
\begin{equation} \label{sball}
\Pi\left(\nu \in \mathcal V: \|\nu-\nu_0\|_{L^2} \le \eps_n/\sqrt{K_D}\right) \ge e^{-Cn\eps_n^2}
\end{equation}
and that 
\begin{equation} \label{except}
\Pi(\nu\in\mathcal V: \|\nu\|_{L^1} \ge \pi/\Delta\text{ or } \|\nu\|_{L^2} > c) \le L e^{-(C+4)n \eps_n^2}.
\end{equation}
If $\Pi(\cdot|X_1, \dots, X_n)$ is the posterior distribution from (\ref{post}), then there exists $M_0$ such that for every $M \ge M_0$, as $n \to \infty$ and in $\mathbb P_{\nu_0}^\mathbb N$-probability, $$\Pi(\nu: \|\nu -\nu_0\|_{\mathbb H(\delta)} \ge M\eps_n|X_1, \dots, X_n) \to 0.$$ 
\end{proposition}


\begin{proof}
Starting with Proposition~\ref{cont0} we replace the condition on the Kullback--Leibler neighbourhood by a condition on a $L^2$ neighbourhood using Lemma~\ref{L2bound}. Further we choose $\mathcal V_n=\{\nu\in\mathcal V:\|\nu\|_{L^1} < \pi/\Delta, \|\nu\|_{L^2} \le c\}$, $d(\nu,\nu_0)=\|\nu -\nu_0\|_{\mathbb H(\delta)}$ and $\delta_n=\eps_n$. The existence of tests follows by Proposition~\ref{prop:tests}.
\end{proof}

\begin{proposition}\label{propball}
Grant Assumption~\ref{overall} for some $s>5/2$, $B>0$, and set
\begin{align}\label{epsn}
\eps_n=n^{-s/(2s+1)}(\log n)^{1/2}.
\end{align}
For the choice $J=J_n$ with $2^{J_n} \approx n^{1/(2s+1)}$ the prior \eqref{prior} satisfies for $n$ large enough the small ball probability condition~\eqref{sball}.
\end{proposition}

The above proposition is proved in Section~\ref{sec:propball}. We now turn to the proof of Proposition~\ref{initio}. When modelling an $s$-regular function $\nu$, and when $\nu_0 \in C^s$ as well, Proposition~\ref{propball} shows (\ref{sball}) for the choice $\eps_n \approx n^{-s/(2s+1)} (\log n)^{1/2},$ and so we obtain the lower bound on the small ball probabilities. By Assumption~\ref{overall} we have $\|\nu\|_{L^1}<\pi/\Delta$ and we also see that the prior concentrates almost surely on a fixed $L^\infty$- (and then also $L^2$-) ball since $\|v\|_{\infty}^2 \lesssim \sum_l 2^{-l/2}$, thus~(\ref{except}) holds for $\Pi$ too. As a consequence we obtain 
\begin{equation}
\Pi(\nu: \|\nu-\nu_0\|_{\mathbb H(\delta)} \le M \eps_n |X_1, \dots, X_n) \to^{\mathbb P_{\nu_0}^\mathbb N} 1.
\end{equation}
Restricting to this event we can further bound $L^2$-distances: by $v_0 = \log \nu_0 \in C^s$ and (\ref{wavprop}) and using Lemma \ref{tedious} below (and the remark before it) we have on an event with posterior probability tending to one
\begin{align*}
&\|\nu-\nu_0\|_{L^2}^2 \lesssim \|v-v_0\|^2_{L^2} = \sum_{l < J} \sum_k \langle v-v_0, \psi_{lk} \rangle^2 + \sum_{l\ge J,k} \langle v_0, \psi_{lk} \rangle^2 \\
& \le 2^J J^{2\delta} \|v-v_0\|^2_{\mathbb H(\delta)} + O(2^{-2Js})  \lesssim 2^J J^{2\delta} \|\nu-\nu_0\|^2_{\mathbb H(\delta)} + O(2^{-2Js})  \lesssim {2^{J} J^{2\delta}}\eps_n^2
\end{align*}
so that, as $n \to \infty$,
$$\Pi(\nu: \|\nu-\nu_0\|_{L^2} \ge C 2^{J/2} J^\delta \eps_n |X_1, \dots, X_n) \to^{\mathbb P_{\nu_0}^\mathbb N} 0$$ 
and further using that with posterior probability tending to one
\begin{align*}
\|\nu-\nu_0\|_{\infty} &\lesssim \|v-v_0\|_{\infty} = \sum_{l < J} 2^{l/2} \max_k |\langle v-v_0, \psi_{lk} \rangle| + \sum_{l\ge J}2^{l/2}\max_{k} |\langle v_0, \psi_{lk} \rangle| \\
& \lesssim 2^{J/2}  \|v-v_0\|_{L^2} + O(2^{-Js})  \lesssim {2^{J} J^{\delta}}\eps_n
\end{align*}
which also implies that 
$$\Pi(\nu: \|\nu-\nu_0\|_{\infty} \ge C 2^{J}J^\delta \eps_n |X_1, \dots, X_n) \to^{\mathbb P_{\nu_0}^\mathbb N} 0.$$
For $\delta>1/2$ we have posterior contraction with rates $\eps_n^{L^2}$ and $\eps_n^{L^\infty}$ in $L^2$ and $L^\infty$, respectively, where 
\begin{align*}
\eps_n^{L^2}=n^{-\frac{s-1/2}{2s+1}}(\log n)^{1/2+\delta}\qquad\text{ and }\qquad
\eps_n^{L^\infty}=n^{-\frac{s-1}{2s+1}}(\log n)^{1/2+\delta}.
\end{align*}
Estimating $\|v-v_0\|_{L^p} \lesssim \|\nu-\nu_0\|_p$ for $p=2,\infty$ implies Proposition \ref{initio}. Moreover, using $(\eps_n^{L^p})^p\le (\eps_n^{L^\infty})^{p-2} (\eps_n^{L^2})^2$ we obtain for contraction in $L^p$ the rate
\begin{align}\label{contractionLp}
\eps_n^{L^p}=n^{-\frac{s+1/p-1}{2s+1}}(\log n)^{1/2+\delta}.
\end{align}

It remains to prove Lemma \ref{tedious}. Let us introduce the spaces $$\mathbb B(\delta) = \bigg\{f: \|f\|^2_{\mathbb B(\delta)}=\sum_{l,k} 2^{l} l^{2\delta} \langle f, \psi_{lk}\rangle^2<\infty\bigg\}, \quad \delta>1/2,$$ which are equal to the (logarithmically refined) Sobolev spaces $H^{1/2,\delta}(I)$. As in Proposition~4.3.12 in \cite{GineNickl2016} one shows that $\mathbb H(\delta)$ is the topological dual space of $\mathbb B(\delta)$. We further see directly from the definition of the prior that $v=\log \nu$ satisfies $$\|v\|^2_{\mathbb B(\delta')} = \sum_{l,k} 2^{l} l^{2\delta'}a_l^2 u^2_{lk} \le \sum_{l \le J} l^{2\delta'-4} \le c,~~~\text{any } \delta'<3/2,$$  and one further shows that also $\|\nu\|_{\mathbb B(\delta')}=\|e^v\|_{\mathbb B(\delta')}$ is bounded by a fixed constant $\Pi$-almost surely (e.g., using the modulus of continuity characterisation of the $\mathbb B(\delta)$-norm, proved as in Section 4.3.5 in \cite{GineNickl2016}). This justifies the application of the following lemma with $1/2<\delta < \delta'  <3/2$ in the above estimate. The lemma is proved in Section~\ref{sec:tedious}.
\begin{lemma} \label{tedious}
a) For any $\nu, \nu_0 \in  \mathbb B(\delta), \delta>1/2,$ such that $\nu, \nu_0$ are bounded away from zero on~$I$ and such that $\|\nu - \nu_0\|_{\mathbb B(\delta)} \to 0,$ we have $\|\log \nu - \log \nu_0 \|_{\mathbb H(\delta)} \lesssim \|\nu-\nu_0\|_{\mathbb H(\delta)}.$

b)If $\|\nu-\nu_0\|_{\mathbb H(\delta)} \to 0$ and $\nu, \nu_0$ are uniformly bounded in $\mathbb B(\delta')$, then for any $\delta< \delta'$ we have $\|\nu-\nu_0\|_{\mathbb B(\delta)} \to 0$.
\end{lemma}

\section{Proof of Proposition \ref{prop:functbvm}}\label{pertu}

Using the definition of $S_n(\nu)$ and the formula for the posterior distribution we obtain
\begin{align}
&E^{\Pi^{D_{n,M}}}\left[e^{t \sqrt n F(\nu)} \Big|X_1, \dots, X_n \right]\nonumber\\
&\qquad\qquad = E^{\Pi^{D_{n,M}}}\left[e^{S_n(\nu)-t \sqrt n \int A_{\nu_0}(v-v_0) A_{\nu_0}(\eta)d\mathbb P_{\nu_0}} \Big|X_1, \dots, X_n \right]\label{lapfunc}\\
&\qquad\qquad= \frac{\int_{D_{n,M}} e^{S_n(\nu)-t \sqrt n \int A_{\nu_0}(v-v_0) A_{\nu_0}(\eta)d\mathbb P_{\nu_0} + \ell_n(\nu)}d\Pi(\nu)}{\int_{D_{n,M}}e^{\ell_n(\nu)}d\Pi(\nu)}.\nonumber
\end{align}
By Assumption~\ref{overall} we have $s>5/2$ so that by Remark~\ref{rem} condition~\eqref{larges} implies condition~\eqref{smalls} and we conclude that the entire Assumption~\ref{simpler} is satisfied. By Lemma~\ref{lemsim}, the choice of~$J$ as in \eqref{jprior}, Assumption~\ref{simpler} and the $L^p$-contraction rates~\eqref{contractionLp} derived from Proposition~\ref{initio} we have that Assumption~\ref{collection} is satisfied. In Section~\ref{sec:likexp} we prove that under Assumption~\ref{collection}
\begin{align} 
& -t \sqrt n \int A_{\nu_0}(v-v_0) A_{\nu_0}(\eta)d\mathbb P_{\nu_0} + \ell_n(\nu)\label{crucexp}\\
&\qquad\qquad=\frac{t^2}{2}\|A_{\nu_0}(\eta)\|_{L^2(\mathbb P_{\nu_0})}^2  - \frac{t}{\sqrt n} \sum_{k=1}^n A_{\nu_0}(\eta)(X_k)+ \ell_n(\nu_t) + r_n'(\nu),\nonumber
\end{align}
where $\sup_{\nu\in D_{n,M}}|r_n'(\nu)|=o_{\mathbb P_{\nu_0}^{\mathbb N}}(1)$ with the nonstochastic null sequence implicit in the $o_{\mathbb P^{\mathbb N}_{\nu_0}}$ notation uniform in $\eta\in\mathcal H_n$. Since the first two terms on the right hand side do not depend on $\nu$ they can be taken outside the posterior integral in \eqref{lapfunc} so that
\begin{align*}
&E^{\Pi^{D_{n,M}}}\left[e^{t \sqrt n F(\nu)} \Big|X_1, \dots, X_n \right]\\
&=\exp \bigg\{\frac{t^2}{2}\|A_{\nu_0}(\eta)\|_{L^2(\mathbb P_{\nu_0})}^2  - \frac{t}{\sqrt n} \sum_{k=1}^n A_{\nu_0}(\eta)(X_k)  \bigg\}\\
&\qquad\qquad\qquad\times \frac{\int_{D_{n,M}} e^{S_n(\nu) + \ell_n(\nu_t)+ r_n'(\nu)}d\Pi(\nu)}{\int_{D_{n,M}}e^{\ell_n(\nu)}d\Pi(\nu)}.
\end{align*}
By the mean value theorem for integrals $r_n'(\nu)$ can be replaced by $r_n$ not depending on~$\nu$ with $|r_n|\le\sup_{\nu\in D_{n,M}}|r_n'(\nu)|=o_{\mathbb P_{\nu_0}^{\mathbb N}}(1)$ in the above display finishing the proof of the proposition.

In order to prove the crucial perturbation approximation (\ref{crucexp}), we first need to obtain formulas for the directional derivatives of the likelihood function, which is done in the next section.

\subsection{Directional derivatives of the likelihood function}

We fix a positive and absolutely continuous L\'evy measure $\nu_0 = \lambda_0 \mu_0$ with corresponding infinitely divisible distribution~$\mathbb P_{\nu_0}$. We set $v_0 = \log \nu_0$ so that $\nu_0 = \exp v_0$ and parametrise a path away from $\nu_0$ as $$\nu^{(s)} = \exp (s(v-v_0)+v_0), ~s \in [0,1].$$ The resulting compound Poisson measure can be identified in the Fourier domain as
\begin{align*}
&\mathcal F \mathbb P_{\nu^{(s+h)}}(k) = \exp \left(\Delta\int (e^{2\pi i k x} -1 ) d\nu^{(s+h)}(x) \right) \displaybreak[0]\\
&= \exp \left(\Delta\int (e^{2\pi i k x} -1 ) \nu^{(s)}(x) e^{h(v-v_0)(x)}dx \right) \displaybreak[0]\\
&= \exp \left(\Delta\int (e^{2\pi i k x} -1 ) \nu^{(s)}(x) \left(e^{h(v-v_0)(x)}-1\right)dx + \Delta\int (e^{2\pi i k x} -1 ) \nu^{(s)}(x) dx \right) \displaybreak[0]\\
& = \mathcal F \mathbb P_{\nu^{(s)}}(k) \times \exp \left(\Delta\int (e^{2\pi i k x} -1 ) \nu^{(s),h}(x)dx \right),
\end{align*}
where $\nu^{(s),h}(x):=\nu^{(s)}(x) \left(e^{h(v-v_0)(x)}-1\right)$ is a finite signed measure on~$I$. One checks by the usual properties of convolution and definition of $e^z$ that the second factor in the last product is the Fourier transform of the finite signed measure $$e^{-\Delta\nu^{(s),h} (I)} \sum_{k=0}^\infty \frac{\Delta^k(\nu^{(s),h})^{\ast k}}{k!}$$ and so we conclude by injectivity of $\mathcal F$ that 
\begin{equation}\label{signed}
\mathbb P_{\nu^{(s+h)}} = e^{-\Delta\nu^{(s),h}(I)} \sum_{k=0}^\infty \frac{\Delta^k(\nu^{(s),h})^{\ast k}}{k!} \ast \mathbb P_{\nu^{(s)}}.
\end{equation}
Let $\Lambda$ denote the Lebesgue (probability) measure on~$I$. We observe that the resulting compound Poisson measure is of the form $\PP_\Lambda=e^{-\Delta}\delta_0+(1-e^{-\Delta})\Lambda$. Both $\mathbb P_{\nu^{(s)}}$ and $\mathbb P_{\nu^{(s+h)}}$ are absolutely continuous with respect to $\PP_\Lambda$. We will now determine the first five derivatives of $ d\PP_{\nu^{(s)}}/ d\PP_{\Lambda}$. To this end we expand \eqref{signed} in terms of $h$. We start with the factor in front of the sum and expand
\begin{align*}
&e^{-\Delta\nu^{(s),h}(I)}
=\exp\left(-\Delta\int(e^{h(v-v_0)(x)}-1) d \nu^{(s)}\right)\displaybreak[0]\\
&=\exp\left(-\Delta\int h(v-v_0)(x)+\frac{h^2}{2}(v-v_0)^2(x)+\frac{h^3}{6}(v-v_0)^3(x)+O(h^4) d \nu^{(s)}\right)\displaybreak[0]\\
&=1-\Delta\int h(v-v_0)(x)+\frac{h^2}{2}(v-v_0)^2(x)+\frac{h^3}{6}(v-v_0)^3(x) d \nu^{(s)}\\
&\quad+\frac{\Delta^2}{2}\left(\int h(v-v_0)(x)+\frac{h^2}{2}(v-v_0)^2(x)+\frac{h^3}{6}(v-v_0)^3(x) d \nu^{(s)}\right)^2\\
&\quad-\frac{\Delta^3}{6}\left(\int h(v-v_0)(x)+\frac{h^2}{2}(v-v_0)^2(x)+\frac{h^3}{6}(v-v_0)^3(x) d \nu^{(s)}\right)^3
+O(h^4)\displaybreak[0]\\
&=1-\Delta h\int v-v_0 d\nu^{(s)}-\Delta\frac{h^2}{2}\int(v-v_0)^2 d\nu^{(s)}-\Delta\frac{h^3}{6}\int(v-v_0)^3 d\nu^{(s)}\\
&\quad+\frac{\Delta^2}{2}h^2\left(\int v-v_0 d\nu^{(s)}\right)^2+\frac{\Delta^2}{2}h^3\int v-v_0  d\nu^{(s)} \int (v-v_0)^2 d\nu^{(s)}\\
&\quad-\frac{\Delta^3}{6}h^3\left(\int v-v_0  d\nu^{(s)}\right)^3+O(h^4).
\end{align*} 
From the definition of $\nu^{(s),h}$ we observe that $(\nu^{(s),h})^{\ast k}=O(h^k)$.
Using \eqref{signed} we obtain
\begin{align*}
&\frac{ d\PP_{\nu^{(s+h)}}}{ d\PP_\Lambda}-\frac{ d\PP_{\nu^{(s)}}}{ d\PP_\Lambda}
=\frac{ d}{ d\PP_\Lambda}\bigg\{ e^{-\Delta\nu^{(s),h}(I)} \sum_{k=0}^\infty \frac{\Delta^k(\nu^{(s),h})^{\ast k}}{k!} \ast \mathbb P_{\nu^{(s)}}-\mathbb P_{\nu^{(s)}}\bigg\}\displaybreak[0]\\
&=\frac{ d}{ d\PP_\Lambda}\bigg\{\bigg(
1-\Delta h\int v-v_0 d\nu^{(s)}-\Delta\frac{h^2}{2}\int(v-v_0)^2 d\nu^{(s)}-\Delta\frac{h^3}{6}\int(v-v_0)^3 d\nu^{(s)}\\
&\quad+\frac{\Delta^2}{2}h^2\left(\int v-v_0 d\nu^{(s)}\right)^2+\frac{\Delta^2}{2}h^3\int v-v_0  d\nu^{(s)} \int (v-v_0)^2 d\nu^{(s)}\\
&\quad-\frac{\Delta^3}{6}h^3\left(\int v-v_0  d\nu^{(s)}\right)^3+O(h^4)
\bigg)\\
&\bigg(
\delta_0+\Delta \nu^{(s)}(e^{h(v-v_0)(x)}-1)
+\frac{\Delta^2}{2} (\nu^{(s)}(e^{h(v-v_0)(x)}-1))^{\ast2}\\
&
+\frac{\Delta^3}{6} (\nu^{(s)}(e^{h(v-v_0)(x)}-1))^{\ast3}+O(h^4)
\bigg)
\ast \PP_{\nu^{(s)}}
-\PP_{\nu^{(s)}}
\bigg\}.
\end{align*}
To find the first derivative we gather all terms that are linear in $h$ and obtain
\begin{align*}
\frac{ d}{ d\PP_\Lambda}\bigg\{
\bigg(\Delta\nu^{(s)} h (v-v_0)-\Delta h \int v-v_0 d\nu^{(s)}\delta_0\bigg)\ast\PP_{\nu^{(s)}}
\bigg\}\\
=h\Delta\frac{ d((\nu^{(s)}(v-v_0))\ast\PP_{\nu^{(s)}}-\int v-v_0 d \nu^{(s)} \PP_{\nu^{(s)}} )}{ d\PP_\Lambda}.
\end{align*}
This gives the first derivative
\begin{align*}
&\frac{ d}{ d s}\frac{ d\PP_{\nu^{(s)}}}{ d\PP_\Lambda}
=\Delta\frac{ d((\nu^{(s)}(v-v_0))\ast\PP_{\nu^{(s)}}-\int v-v_0 d \nu^{(s)} \PP_{\nu^{(s)}} )}{ d\PP_\Lambda}.
\end{align*}
Gathering all terms quadratic in $h$ we find
\begin{align*}
&\frac{ d}{ d\PP_\Lambda}\bigg\{
\bigg(
\Delta\nu^{(s)}\frac{h^2}{2}(v-v_0)^2+\frac{\Delta^2}{2}(\nu^{(s)}h(v-v_0))^{\ast2}-\Delta^2h\int v-v_0  d\nu^{(s)} \nu^{(s)} h(v-v_0)\\
&\qquad\qquad-\frac{\Delta h^2}{2}\int(v-v_0)^2 d\nu^{(s)}\delta_0+\frac{\Delta^2h^2}{2}\bigg(\int v-v_0 d\nu^{(s)}\bigg)^2\delta_0
\bigg)\ast\PP_{\nu^{(s)}}\bigg\}\displaybreak[0]\\
&=\frac{h^2}{2}\frac{ d}{ d\PP_\Lambda}\bigg\{
\bigg(\Delta\nu^{(s)}(v-v_0)^2-\Delta\int(v-v_0)^2 d\nu^{(s)}\delta_0+\Delta^2(\nu^{(s)}(v-v_0))^{\ast2}\\
&\qquad\qquad\qquad-2\Delta^2\int v-v_0 d\nu^{(s)} (v-v_0)\nu^{(s)}+\Delta^2\bigg(\int v-v_0 d\nu^{(s)}\bigg)^2\delta_0
\bigg)\ast\PP_{\nu^{(s)}}\bigg\}.
\end{align*}
And this gives the second derivative 
\begin{align*}
&\frac{ d^2}{ d s^2}\frac{ d\PP_{\nu^{(s)}}}{ d\PP_\Lambda}
=\frac{ d}{ d\PP_\Lambda}\bigg\{\Delta((v-v_0)^2\nu^{(s)})*\PP_{\nu^{(s)}}-\Delta\int(v-v_0)^2 d \nu^{(s)}\PP_{\nu^{(s)}}\\
&\qquad +\Delta^2\left(((v-v_0)\nu^{(s)})-\delta_0\int(v-v_0) d \nu^{(s)}\right)^{*2}*\PP_{\nu^{(s)}}\bigg\}.
\end{align*}
Finally we gather all terms which are cubic in $h$. This yields
\begin{align*}
&\frac{ d}{ d\PP_\Lambda}\bigg\{
\bigg(\Delta\nu^{(s)}\frac{h^3}{6}(v-v_0)^3+\Delta^2 \bigg((\nu^{(s)}h(v-v_0))\ast\Big(\nu^{(s)}\frac{h^2}{2}(v-v_0)^2\Big)\bigg)\\
&+\frac{\Delta^3}{6}(\nu^{(s)}h(v-v_0))^{\ast3}\\
&-\Delta^2 h \int v-v_0 d \nu^{(s)}\nu^{(s)}\frac{h^2}{2}(v-v_0)^2
-\Delta^3 h \int v-v_0 d\nu^{(s)}\frac{1}{2}(\nu^{(s)}h (v-v_0))^{\ast2}\\
&+\frac{h^2}{2}\bigg(\Delta^3\Big(\int v-v_0 d \nu^{(s)}\Big)^2-\Delta^2\int(v-v_0)^2 d \nu^{(s)}\bigg)\nu^{(s)} h (v-v_0) \\
& -\frac{h^3\Delta^3}{6}\bigg(\int v-v_0 d\nu^{(s)}\bigg)^3\delta_0\\
&-\frac{h^3\Delta}{6}\int(v-v_0)^3 d\nu^{(s)}\delta_0+\frac{h^3\Delta^2}{2}\int v-v_0 d\nu^{(s)}\int(v-v_0)^2 d\nu^{(s)}\delta_0\bigg)\ast\PP_{\nu^{(s)}}\bigg\}.
\end{align*}
In this way we obtain the third derivative
\begin{align*}
&\frac{ d^3}{ d s^3}\frac{ d\PP_{\nu^{(s)}}}{ d\PP_\Lambda}
=\frac{ d}{ d\PP_\Lambda}\bigg\{\Delta((v-v_0)^3\nu^{(s)})*\PP_{\nu^{(s)}}-\Delta\int(v-v_0)^3 d \nu^{(s)}\PP_{\nu^{(s)}}\\
&+3\Delta^2\left(((v-v_0)\nu^{(s)})-\delta_0\int(v-v_0) d \nu^{(s)}\right)*\left(((v-v_0)^2\nu^{(s)})-\delta_0\int(v-v_0)^2 d \nu^{(s)}\right)*\PP_{\nu^{(s)}}\\
&+\Delta^3\left(((v-v_0)\nu^{(s)})-\delta_0\int(v-v_0) d \nu^{(s)}\right)^{*3}*\PP_{\nu^{(s)}}\bigg\}.
\end{align*}
In a similar way we obtain for the fourth and fifth derivative
\begin{align*}
&\frac{ d^4}{ d s^4}\frac{ d\PP_{\nu^{(s)}}}{ d\PP_\Lambda}
=\frac{ d}{ d\PP_\Lambda}\bigg\{\Delta((v-v_0)^4\nu^{(s)})*\PP_{\nu^{(s)}}-\Delta\int(v-v_0)^4 d \nu^{(s)}\PP_{\nu^{(s)}}\\
&+3\Delta^2\left(((v-v_0)^2\nu^{(s)})-\delta_0\int(v-v_0)^2 d \nu^{(s)}\right)^{*2}*\PP_{\nu^{(s)}}\\
&+4\Delta^2\left(((v-v_0)\nu^{(s)})-\delta_0\int(v-v_0) d \nu^{(s)}\right)*\left(((v-v_0)^3\nu^{(s)})-\delta_0\int(v-v_0)^3 d \nu^{(s)}\right)*\PP_{\nu^{(s)}}\\
&+6\Delta^3\left(((v-v_0)\nu^{(s)})-\delta_0\int(v-v_0) d \nu^{(s)}\right)^{*2}*\left(((v-v_0)^2\nu^{(s)})-\delta_0\int(v-v_0)^2 d \nu^{(s)}\right)*\PP_{\nu^{(s)}}\\
&+\Delta^4\left(((v-v_0)\nu^{(s)})-\delta_0\int(v-v_0) d \nu^{(s)}\right)^{*4}*\PP_{\nu^{(s)}}\bigg\},\allowdisplaybreaks[1]\\
&\frac{ d^5}{ d s^5}\frac{ d\PP_{\nu^{(s)}}}{ d\PP_\Lambda}
=\frac{ d}{ d\PP_\Lambda}\bigg\{\Delta((v-v_0)^5\nu^{(s)})*\PP_{\nu^{(s)}}-\Delta\int(v-v_0)^5 d \nu^{(s)}\PP_{\nu^{(s)}}\\
&+10\Delta^2\left(((v-v_0)^2\nu^{(s)})-\delta_0\int(v-v_0)^2 d \nu^{(s)}\right)*\left(((v-v_0)^3\nu^{(s)})-\delta_0\int(v-v_0)^3 d \nu^{(s)}\right)*\PP_{\nu^{(s)}}\\
&+5\Delta^2\left(((v-v_0)\nu^{(s)})-\delta_0\int(v-v_0) d \nu^{(s)}\right)*\left(((v-v_0)^4\nu^{(s)})-\delta_0\int(v-v_0)^4 d \nu^{(s)}\right)*\PP_{\nu^{(s)}}\\
&+10\Delta^3\left(((v-v_0)\nu^{(s)})-\delta_0\int(v-v_0) d \nu^{(s)}\right)^{*2}*\left(((v-v_0)^3\nu^{(s)})-\delta_0\int(v-v_0)^3 d \nu^{(s)}\right)*\PP_{\nu^{(s)}}\\
&+15\Delta^3\left(((v-v_0)^2\nu^{(s)})-\delta_0\int(v-v_0)^2 d \nu^{(s)}\right)^{*2}*\left(((v-v_0)\nu^{(s)})-\delta_0\int(v-v_0) d \nu^{(s)}\right)*\PP_{\nu^{(s)}}\\
&+10\Delta^4\left(((v-v_0)\nu^{(s)})-\delta_0\int(v-v_0) d \nu^{(s)}\right)^{*3}*\left(((v-v_0)^2\nu^{(s)})-\delta_0\int(v-v_0)^2 d \nu^{(s)}\right)*\PP_{\nu^{(s)}}\\
&+\Delta^5\left(((v-v_0)\nu^{(s)})-\delta_0\int(v-v_0) d \nu^{(s)}\right)^{*5}*\PP_{\nu^{(s)}}\bigg\}.
\end{align*}
Let $L_0^2(\PP_\nu):=\{g\in L^2(\PP_\nu):\int g d \PP_\nu=0\}$. Motivated by the structure of the derivatives we define the multilinear form
\begin{align}
&A_\nu|_{L^2(\nu)^{\otimes k}}:L^2(\nu)^{\otimes k}\to L_0^2(\PP_\nu),\label{multform}\\
&(w_1,\dots,w_k)\mapsto\Delta^k\frac{ d((w_1\nu-\delta_0\int w_1 d\nu)*\dots*(w_k\nu-\delta_0\int w_k d\nu)*\PP_\nu)}{ d\PP_\nu}.\nonumber
\end{align}

In view of the derivatives of the log-likelihood we divide the derivatives by $ d \PP_{\nu^{(s)}}/ d\PP_\Lambda$. Then the dominating measure $\PP_\Lambda$ cancels and we suppress it in the notation. We obtain the following expressions
\begin{align*}
\frac{ d \frac{ d}{ d s}\PP_{\nu^{(s)}}}{ d \PP_{\nu^{(s)}}}
&=A_{\nu^{(s)}}(v-v_0),\allowdisplaybreaks[1]\\
\frac{ d \frac{ d^2}{ d s^2}\PP_{\nu^{(s)}}}{ d \PP_{\nu^{(s)}}}
&=A_{\nu^{(s)}}(v-v_0)^2+ A_{\nu^{(s)}}(v-v_0,v-v_0),\allowdisplaybreaks[1]\\
\frac{ d \frac{ d^3}{ d s^3}\PP_{\nu^{(s)}}}{ d \PP_{\nu^{(s)}}}
&=A_{\nu^{(s)}}(v-v_0)^3+3 A_{\nu^{(s)}}(v-v_0,(v-v_0)^2)+A_{\nu^{(s)}}(v-v_0,v-v_0,v-v_0),\allowdisplaybreaks[1]\\
\frac{ d \frac{ d^4}{ d s^4}\PP_{\nu^{(s)}}}{ d \PP_{\nu^{(s)}}}
&=A_{\nu^{(s)}}(v-v_0)^4+4 A_{\nu^{(s)}}(v-v_0,(v-v_0)^3)+3A_{\nu^{(s)}}((v-v_0)^2,(v-v_0)^2)\\
&\quad+6 A_{\nu^{(s)}}(v-v_0,v-v_0,(v-v_0)^2)+A_{\nu^{(s)}}(v-v_0,v-v_0,v-v_0,v-v_0)\\
\frac{ d \frac{ d^5}{ d s^5}\PP_{\nu^{(s)}}}{ d \PP_{\nu^{(s)}}}
&=A_{\nu^{(s)}}(v-v_0)^5+5 A_{\nu^{(s)}}(v-v_0,(v-v_0)^4)+10A_{\nu^{(s)}}((v-v_0)^2,(v-v_0)^3)\\
&\quad+10 A_{\nu^{(s)}}(v-v_0,v-v_0,(v-v_0)^3)+15A_{\nu^{(s)}}((v-v_0)^2,(v-v_0)^2,v-v_0)\\
&\quad+10 A_{\nu^{(s)}}(v-v_0,v-v_0,v-v_0,(v-v_0)^2)\\
&\quad+A_{\nu^{(s)}}(v-v_0,v-v_0,v-v_0,v-v_0,v-v_0).
\end{align*}
With the densities at hand we can determine the derivatives of the empirical log-likelihood
\begin{align*}
&D\ell_n(\nu_0)[v-v_0]=\sum_{j=1}^{n}\frac{ d \frac{ d}{ d s}\PP_{\nu^{(s)}}}{ d \PP_{\nu^{(s)}}}\bigg|_{s=0}(X_j)
,\\
&D^2\ell_n(\nu_0)[v-v_0,v-v_0]=\sum_{j=1}^{n}\frac{ d \frac{ d^2}{ d s^2}\PP_{\nu^{(s)}}}{ d \PP_{\nu^{(s)}}}\bigg|_{s=0}(X_j)-\sum_{j=1}^{n}\bigg(\frac{ d \frac{ d}{ d s}\PP_{\nu^{(s)}}}{ d \PP_{\nu^{(s)}}}\bigg)^2\bigg|_{s=0}(X_j)\\
&D^3\ell_n(\nu^{(s)})[v-v_0,v-v_0,v-v_0]\displaybreak[0]\\
&=\sum_{j=1}^{n}\frac{ d \frac{ d^3}{ d s^3}\PP_{\nu^{(s)}}}{ d \PP_{\nu^{(s)}}}(X_j)-3\sum_{j=1}^{n}\frac{ d \frac{ d^2}{ d s^2}\PP_{\nu^{(s)}}}{ d \PP_{\nu^{(s)}}}(X_j)\frac{ d \frac{ d}{ d s}\PP_{\nu^{(s)}}}{ d \PP_{\nu^{(s)}}}(X_j)+2\sum_{j=1}^{n}\bigg(\frac{ d \frac{ d}{ d s}\PP_{\nu^{(s)}}}{ d \PP_{\nu^{(s)}}}(X_j)\bigg)^3\displaybreak[0]\\
&D^4\ell_n(\nu^{(s)})[v-v_0,v-v_0,v-v_0,v-v_0]\\
&=\sum_{j=1}^{n}\frac{ d \frac{ d^4}{ d s^4}\PP_{\nu^{(s)}}}{ d \PP_{\nu^{(s)}}}(X_j)-4\sum_{j=1}^{n}\frac{ d \frac{ d^3}{ d s^3}\PP_{\nu^{(s)}}}{ d \PP_{\nu^{(s)}}}(X_j)\frac{ d \frac{ d}{ d s}\PP_{\nu^{(s)}}}{ d \PP_{\nu^{(s)}}}(X_j)-3\sum_{j=1}^{n}\bigg(\frac{ d \frac{ d^2}{ d s^2}\PP_{\nu^{(s)}}}{ d \PP_{\nu^{(s)}}}(X_j)\bigg)^2\\
&\quad+12\sum_{j=1}^{n}\frac{ d \frac{ d^2}{ d s^2}\PP_{\nu^{(s)}}}{ d \PP_{\nu^{(s)}}}(X_j)\bigg(\frac{ d \frac{ d}{ d s}\PP_{\nu^{(s)}}}{ d \PP_{\nu^{(s)}}}(X_j)\bigg)^2-6\sum_{j=1}^{n}\bigg(\frac{ d \frac{ d}{ d s}\PP_{\nu^{(s)}}}{ d \PP_{\nu^{(s)}}}(X_j)\bigg)^4\displaybreak[0]\\
&D^5\ell_n(\nu^{(s)})[v-v_0,v-v_0,v-v_0,v-v_0,v-v_0]\\
&=\sum_{j=1}^{n}\frac{ d \frac{ d^5}{ d s^5}\PP_{\nu^{(s)}}}{ d \PP_{\nu^{(s)}}}(X_j)-5\sum_{j=1}^{n}\frac{ d \frac{ d^4}{ d s^4}\PP_{\nu^{(s)}}}{ d \PP_{\nu^{(s)}}}(X_j)\frac{ d \frac{ d}{ d s}\PP_{\nu^{(s)}}}{ d \PP_{\nu^{(s)}}}(X_j)
\\
&\quad+20\sum_{j=1}^{n}\frac{ d \frac{ d^3}{ d s^3}\PP_{\nu^{(s)}}}{ d \PP_{\nu^{(s)}}}(X_j)\bigg(\frac{ d \frac{ d}{ d s}\PP_{\nu^{(s)}}}{ d \PP_{\nu^{(s)}}}(X_j)\bigg)^2-10\sum_{j=1}^{n}\frac{ d \frac{ d^3}{ d s^3}\PP_{\nu^{(s)}}}{ d \PP_{\nu^{(s)}}}(X_j)\frac{ d \frac{ d^2}{ d s^2}\PP_{\nu^{(s)}}}{ d \PP_{\nu^{(s)}}}(X_j)\\
&\quad-60\sum_{j=1}^{n}\frac{ d \frac{ d^2}{ d s^2}\PP_{\nu^{(s)}}}{ d \PP_{\nu^{(s)}}}(X_j)\bigg(\frac{ d \frac{ d}{ d s}\PP_{\nu^{(s)}}}{ d \PP_{\nu^{(s)}}}(X_j)\bigg)^3+30\sum_{j=1}^{n}\bigg(\frac{ d \frac{ d^2}{ d s^2}\PP_{\nu^{(s)}}}{ d \PP_{\nu^{(s)}}}(X_j)\bigg)^2\frac{ d \frac{ d}{ d s}\PP_{\nu^{(s)}}}{ d \PP_{\nu^{(s)}}}(X_j)\\
&\quad+24\sum_{j=1}^{n}\bigg(\frac{ d \frac{ d}{ d s}\PP_{\nu^{(s)}}}{ d \PP_{\nu^{(s)}}}(X_j)\bigg)^5.
\end{align*}

The previous quantities simply denote one-dimensional derivatives of the empirical log-likelihood along the curve $\nu^{(s)}$. These derivatives can be viewed as values on the diagonal of symmetric multilinear forms and by means of polarization we extend the derivatives to symmetric multilinear forms.

\subsection{Likelihood expansion}\label{sec:likexp}

In this section we will use a likelihood expansion to show the statement used in Section~\ref{pertu} that
\begin{align*}
&-t \sqrt n \int A_{\nu_0}(v-v_0) A_{\nu_0}(\eta)d\mathbb P_{\nu_0} + \ell_n(\nu)\allowdisplaybreaks[1]\\
&\qquad =\frac{t^2}{2}\|A_{\nu_0}(\eta)\|_{L^2(\mathbb P_{\nu_0})}^2  - \frac{t}{\sqrt n} \sum_{k=1}^n A_{\nu_0}(\eta)(X_k)+ \ell_n(\nu_t) + r_n'(\nu),
\end{align*}
where $\sup_{\nu\in D_{n,M}}|r_n'(\nu)|=o_{\mathbb P_{\nu_0}^{\mathbb N}}(1)$. Let $\eps_{n}^{L^p}$ with $2<p<\infty$ be rates such that for
\[ D_{n,p}=D_{n,p,M}:= \left\{\nu: v \in V_{B,J}, \|v - v_0\|_{L^p} \le M \eps_{n}^{L^p} \right\}\]
we have 
\[\Pi (D_{n,p}^c | X_1, \dots, X_n) \to^{\mathbb P^\mathbb N_{\nu_0}} 0.\]
For example we can take $(\eps_{n}^{L^p})^p=(\eps_{n}^{L^\infty})^{p-2}(\eps_{n}^{L^2})^2$.
Setting $\omega_n^{L^p}={t}{n}^{-1/2}\|\eta\|_{L^p}+\delta_n\eps_n^{L^p}$ we work under the following conditions.
\begin{assumption}\label{collection}
Let $\mathcal H_n \subset L^\infty(I)$.
Assume $J$, $\delta_n$, $\eps_n^{L^p}$ and $\omega_n^{L^p}$ satisfy uniformly over $\eta\in\mathcal H_n$
\begin{align*}
& 2^{-Js} =o(\eps_n^{L^2}),\quad 2^{-Js} =o(\eps_n^{L^\infty}),\quad\text{(bias conditions)}\\
&\sqrt{n}\delta_n \eps_n^{L^2} 2^{J/2}\sqrt{\log \frac{c}{\eps_n^{L^2}}} =o(1),\quad\text{(for term $II$)}\\
&\frac{2^{J/2}}{\sqrt{n}}\sqrt{\log \frac{c}{\eps_n^{L^2} }}\lesssim \eps_n^{L^2},\quad\text{(first term dominates in $II$)}\\
&n\delta_n\big(\eps_n^{L^2}\big)^2=o(1),\quad\text{(for centring of $III(ii)$)}\\
&t\|\eta\|_\infty\eps_n^{L^2}2^{J/2}\sqrt{\log\frac{c}{\eps_n^{L^2}}}=o(1),\quad\text{(for term $III(i)$)}\\
&\frac{t^2}{\sqrt{n}}\|\eta\|^2_{L^4}=o(1),\quad\text{(for deviation from mean of $IV(i)$)}\\
&t\delta_n\sqrt{n}\|\eta\|_{L^2}\eps_n^{L^2}=o(1),\quad\text{(for centring of $IV(iii)$)}\\
&n\,\omega_n^{L^3}\left(\eps_n^{L^3}\right)^2=o(1),\qquad
n\left(\omega_n^{L^3}\right)^3=o(1),\quad\text{(for centring of third derivative)}\\
&n\,\omega_n^{L^4}\left(\eps_n^{L^4}\right)^3=o(1),\qquad
n\left(\omega_n^{L^4}\right)^4=o(1),\quad\text{(for centring of fourth derivative)}\\
&n\left(\eps_n^{L^5}\right)^5=o(1),\qquad
n\left(\omega_n^{L^5}\right)^5=o(1),\quad\text{(for centring of fifth derivative)}\\
&\sqrt{n}\Big(\eps_n^{L^\infty}+\omega_n^{L^\infty}\Big)^2  \Big(\eps_n^{L^2}+\omega_n^{L^2}\Big) 2^{J/2}\bigg(\log \frac{c}{\eps_n^{L^2}+\omega_n^{L^2}}\bigg)^{1/2}=o(1),\quad\text{(for $R_n$)}\\
&\frac{1}{\sqrt{n}} 2^{J/2}\bigg(\log \frac{c}{\eps_n^{L^2}+\omega_n^{L^2}}\bigg)^{1/2}\lesssim \eps_n^{L^2}+\omega_n^{L^2}.\quad\text{ (first term dominates in $R_n$)}
\end{align*}
\end{assumption}

We consider the following path from $\nu_0$ to $\nu$, $s\mapsto\exp(s(v-v_0)+v_0)=\nu^{(s)}$. A Taylor expansion of the log-likelihood $\ell_n$ along this path gives
\begin{align*}
\ell_n(\nu)-\ell_n(\nu_0)=D\ell_n(\nu_0)[v-v_0]&+\tfrac1{2}D^2\ell_n(\nu_0)[v-v_0,v-v_0]\\
&+\tfrac1{6}D^3\ell_n(\nu^{(s)})[v-v_0,v-v_0,v-v_0],
\end{align*}
where the first two terms denote first and second derivative at zero and the last term denotes the third derivative at some intermediate point $s\in[0,1]$. We will see later that the derivatives depend linearly on the directions. Thus it is possible to extend them to symmetric multilinear forms.
The corresponding path from $\nu_0$ to $\nu_t=\exp(v_t)$ is $u\mapsto\exp(u(v_t-v_0)+v_0)=\nu_t^{(u)}$. 

We recall the perturbation~\eqref{pert} and define $\tilde \delta_n(v)$ by
\[
v_t = v + \delta_n \Big(\frac{t}{\delta_n \sqrt n} \eta + v_{0,J} -v  \Big) = v +\tilde \delta_n(v).
\]
With this definition we calculate
\begin{align*}
&\ell_n(\nu)-\ell_n(\nu_0)-(\ell_n(\nu_t)-\ell_n(\nu_0))\\
&\qquad\qquad=D\ell_n(\nu_0)[v-v_0]-D\ell_n(\nu_0)[v_t-v_0]
+\tfrac1{2}D^2\ell_n(\nu_0)[v-v_0,v-v_0]\\
&\qquad\qquad\quad-\tfrac1{2}D^2\ell_n(\nu_0)[v_t-v_0,v_t-v_0]+R_n\\
&\qquad\qquad=D\ell_n(\nu_0)[v-v_t]
+\tfrac1{2}D^2\ell_n(\nu_0)[v-v_0,v-v_0]\\
&\qquad\qquad\quad-\tfrac1{2}D^2\ell_n(\nu_0)[v-v_0+\tilde \delta_n(v),v-v_0+\tilde \delta_n(v)]+R_n\\
&\qquad\qquad=-D\ell_n(\nu_0)[(t/\sqrt{n})\eta]-\delta_n D\ell_n(\nu_0)[v_{0,J}-v]
-D^2\ell_n(\nu_0)[v-v_0,\tilde \delta_n(v)]\\
&\qquad\qquad\quad-\tfrac1{2}D^2\ell_n(\nu_0)[\tilde \delta_n(v),\tilde \delta_n(v)]+R_n\\
&\qquad\qquad=I+II+III+IV+R_n,
\end{align*}
where
\[R_n=\tfrac1{6}D^3\ell_n(\nu^{(s)})[v-v_0,v-v_0,v-v_0]-\tfrac1{6}D^3\ell_n(\nu_t^{(u)})[v_t-v_0,v_t-v_0,v_t-v_0]\]
with intermediate points $s,u\in[0,1]$.

We need to show that
\begin{align}
I+II+III+IV+R_n&=t \sqrt n \int A_{\nu_0}(v-v_0) A_{\nu_0}(\eta)d\mathbb P_{\nu_0}\nonumber\\
&\quad +\frac{t^2}{2}\|A_{\nu_0}(\eta)\|_{L^2(\mathbb P_{\nu_0})}^2  - \frac{t}{\sqrt n} \sum_{k=1}^n A_{\nu_0}(\eta)(X_k)+  r_n'(\nu).\label{123}
\end{align}
The first term is given by $I=-\frac{t}{\sqrt{n}}D\ell_n(\nu_0)[\eta]=-\frac{t}{\sqrt{n}}\sum_{k=1}^n A_{\nu_0}[\eta](X_k)$. 
For the second term we have
\begin{align*}
II=-\delta_n D\ell_n(\nu_0)[v_{0,J}-v]
=\sqrt{n}\delta_n\frac{1}{\sqrt{n}}\sum_{k=1}^n
A_{\nu_0}(v-v_{0,J})(X_k)
=\sqrt{n}\delta_n\mathbb G_n f_v, 
\end{align*}
where $\mathbb G_n=\sqrt{n}(\PP_{\nu_0,n}-\PP_{\nu_0})$ is the empirical process and 
\(f_v=A_{\nu_0}(v-v_{0,J}).\)

On $D_{n,M}$ we have $\|v-v_0\|_{L^2} \le M\eps_n^{L^2}$ and $\|v-v_0\|_{\infty} \le M\eps_n^{L^\infty}$. Using the usual bias bounds $\|v_{0,J}-v_0\|_{L^2}\lesssim 2^{-Js}$, $\|v_{0,J}-v_0\|_{\infty}\lesssim 2^{-Js}$ and the bias condition in Assumption~\ref{collection} we obtain $\|v-v_{0,J}\|_{L^2} \le M\eps_n^{L^2}$ and $\|v-v_{0,J}\|_{\infty} \le M\eps_n^{L^\infty}$ with a possibly larger constant $M$. 
We recall \(f_v=A_{\nu_0}(v-v_{0,J})\) and consider the finite dimensional class of functions
\begin{equation}\label{funcclass}
\F:=\left\{f_v: v\in V_{B,J}, \|v-v_{0,J}\|_{L^2} \le M\eps_n^{L^2}, \|v-v_{0,J}\|_{\infty} \le M\eps_n^{L^\infty} \right\}. 
\end{equation}
We observe that there is $D>0$ such that $\|v_0\|_\infty\le D$ and $\|v\|_\infty\le D$ for all $v\in V_{B,J}$. We will bound the norms of functions in $\F$ using the following lemma.
\begin{lemma}\label{sup}
Let $\|v\|_\infty\le D$ and $\nu=\exp(v)$.
Then for $A_\nu$ defined in \eqref{multform} and for $1\le p\le\infty$
\begin{align*}
\|A_{\nu}(w_1,\dots,w_k)\|_{L^p(\PP_\nu)}\lesssim \|w_1\|_{L^p(\nu)}\dots\|w_k\|_{L^p(\nu)}.
\end{align*}
The constants only depends on $k$, $D$ and $\Delta$.
\end{lemma}
\begin{proof}
We write $\nu$ for both the L\'evy measure and its density.
The measure $\PP_{\nu}$ can be written as a convolution exponential $\PP_{\nu}=e^{-\Delta\lambda}\sum_{k=0}^\infty \frac{\Delta^k}{k!}\nu^{*k}$ with intensity $\lambda=\nu((-1/2,1/2])$. The function $v$ is bounded such that the corresponding L\'evy density $\nu=\exp(v)$ is bounded from above and bounded away from zero. Likewise the intensity~$\lambda$ is bounded from above and bounded away from zero. We denote by $\Lambda$ the Lebesgue measure on $[-1/2,1/2]$. Then $\frac{ d \Lambda}{ d\PP_{\nu}}$ is in $L^\infty(\PP_\nu)$ with norm bounded by a constant depending on $D$ and $\Delta$ only. Defining by $\PP_{\nu}^a=e^{-\Delta\lambda}\sum_{k=1}^\infty \frac{\Delta^k}{k!}\nu^{*k}$ the absolutely continuous part with respect to the Lebesgue measure $\Lambda$ we see likewise that the density $\frac{ d\PP_{\nu}^a}{ d \Lambda}$ is bounded in $L^\infty(\Lambda)$ from above depending on $D$ and $\Delta$ only.
By definition we have
\begin{align*}
&\|A_\nu(w_1,\dots,w_k)\|_{L^p(\PP_\nu)}\\
&\le\Delta^k\left\|\frac{ d((w_1\nu-\delta_0\int w_1 d\nu)*\dots*(w_k\nu-\delta_0\int w_k d\nu)*\PP_\nu)}{ d\PP_\nu}\right\|_{L^p(\PP_\nu)}.
\end{align*}
The nominator consists of $2^k$ terms and a typical term is of the from
\begin{align*}
\int w_1 d\nu\dots\int w_j d \nu \cdot (w_{j+1}\nu)*\dots*(w_k\nu)*\PP_\nu
\end{align*}
and up to permutation and choice of $j$ between 0 and $k$ all terms are of this form. So it suffices to bound
\begin{align*}
&\left\|\frac{ d(\int w_1 d\nu\dots\int w_j d \nu \cdot (w_{j+1}\nu)*\dots*(w_k\nu)*\PP_\nu)}{ d\PP_\nu}\right\|_{L^p(\PP_\nu)}\\
&\lesssim \|w_1\|_{L^1(\nu)}\dots\| w_j\|_{L^1(\nu)}\left\|\frac{ d( (w_{j+1}\nu)*\dots*(w_k\nu)*\PP_\nu)}{ d\PP_\nu}\right\|_{L^p(\PP_\nu)}\\
&\lesssim \|w_1\|_{L^p(\nu)}\dots\| w_j\|_{L^p(\nu)}\left\|\frac{ d( (w_{j+1}\nu)*\dots*(w_k\nu)*\PP_\nu)}{ d\PP_\nu}\right\|_{L^p(\PP_\nu)}.
\end{align*}
For $j=k$ this gives the desired bound and for $j<k$ the previous line can be bounded by
\begin{align*}
& \|w_1\|_{L^p(\nu)}\dots\| w_j\|_{L^p(\nu)}\left\|\frac{ d( (w_{j+1}\nu)*\dots*(w_k\nu)*\PP_\nu)}{ d\Lambda}\right\|_{L^p(\PP_\nu)}\left\|\frac{ d\Lambda}{ d\PP_\nu}\right\|_{L^\infty(\PP_\nu)}\allowdisplaybreaks[1]\\
&\lesssim \|w_1\|_{L^p(\nu)}\dots\| w_j\|_{L^p(\nu)}\left\|\frac{ d( (w_{j+1}\nu)*\dots*(w_k\nu)*\PP_\nu)}{ d\Lambda}\right\|_{L^p(\Lambda)},
\end{align*}
where we have used boundedness of $\frac{ d \Lambda}{ d\PP_{\nu}}$ and $\frac{ d\PP_{\nu}^a}{ d \Lambda}$. Young's inequality for convolutions yields the bound
\begin{align*}
&\|w_1\|_{L^p(\nu)}\dots\| w_j\|_{L^p(\nu)}\left\|w_{j+1}\nu\right\|_{L^1(\Lambda)}\dots\left\|w_{k-1}\nu\right\|_{L^1(\Lambda)}\left\|w_{k}\nu\right\|_{L^p(\Lambda)}\\
&\lesssim
\|w_1\|_{L^p(\nu)}\dots\left\|w_{k}\right\|_{L^p(\nu)}
\end{align*}
and the lemma follows by treating all $2^k$ terms in this way.
\end{proof}

We define $v(u)=\sum_{l\le J-1}\sum_{k}a_l u_{lk}\psi_{lk} $ with $a_l=2^{-l}(l^2+1)^{-1}$.
For $u,u'\in \R^{2^J}$ we denote $v=v(u), v'=v(u')$. 
Applying Lemma~\ref{sup} with $w_1=v-v'$ yields 
$\|f_v-f_{v'}\|_\infty\lesssim \|v-v'\|_\infty\lesssim \|u-u'\|_\infty$, where the constant only depends on $D$ and $\Delta$. It follows that $\sup_{\Q}\|f_v-f_{v'}\|_{L^2(\Q)}\lesssim\|u-u'\|_{\infty}$, where the supremum is over all Borel probability measures $\Q$. Consequently we have $\sup_{\Q} N(\F,L^2(\Q),\eps \|F\|_{L^2(\Q)})\le (A/\eps)^{2^J}$, for some $A\ge2$ and for $0<\eps< A$ and where the envelope can be taken as a constant function $F$ with constant only depending on $D$ and $\Delta$.

Let $\sigma^2=\sup_{f\in\F}\PP_{\nu_0} f^2$. Lemma~\ref{sup} yields 
\begin{align*}
\sigma\le \sup_{\|v-v_{0,J}\|\le M \eps_{n}^{L^2}}\|A_{\nu_0}(v-v_{0,J})\|_{L^2(\PP_{\nu_0})}
\lesssim\sup_{\|v-v_{0,J}\|\le M \eps_{n}^{L^2}}\|v-v_{0,J}\|_{L^2({\nu_0})}\lesssim \eps_n^{L^2}.
\end{align*}
Then we have by Corollary~3.5.8 in \cite{GineNickl2016} for some $c>0$
\begin{align*}
E\|\mathbb G_n\|_{\F}\lesssim \eps_n^{L^2} 2^{J/2}\sqrt{\log \frac{c}{\eps_n^{L^2}}}
+ \frac{1}{\sqrt n} 2^{J}\log \frac{c}{\eps_n^{L^2}}.
\end{align*}
We obtain $II=o_{\PP}(1)$ using the conditions
\begin{align*}
\sqrt{n}\delta_n \eps_n^{L^2} 2^{J/2}\sqrt{\log \frac{c}{\eps_n^{L^2}}} =o(1)
\quad
\text{ and }
\quad
\frac{2^{J/2}}{\sqrt{n}}\sqrt{\log \frac{c}{\eps_n^{L^2} }}\lesssim \eps_n^{L^2}.
\end{align*}

Next we consider the term $III$. It equals
\begin{align*}
&\quad-D^2\ell_n(\nu_0)[v-v_0,\tilde \delta_n(v)]
=\underbrace{-{n}^{-1/2}tD^2\ell_n(\nu_0)[v-v_0,\eta]}_{(i)}\underbrace{+\delta_n D^2\ell_n(\nu_0)[v-v_0,v-v_{0,J}]}_{(ii)}\allowdisplaybreaks[0]\\
&=\underbrace{-\frac{t}{\sqrt{n}}\sum_{j=1}^{n}\frac{ d \frac{ d^2}{ d s^2}\PP_{\nu^{(s)}}}{ d \PP_{\nu^{(s)}}}\bigg|_{s=0}[v-v_0,\eta](X_j)}_{(i)(a)}\underbrace{+\frac{t}{\sqrt{n}}\sum_{j=1}^{n}\left(\frac{ d \frac{ d}{ d s}\PP_{\nu^{(s)}}}{ d \PP_{\nu^{(s)}}}\right)^2\bigg|_{s=0}[v-v_0,\eta](X_j)}_{(i)(b)}\\
&\underbrace{+\delta_n \sum_{j=1}^{n}\frac{ d \frac{ d^2}{ d s^2}\PP_{\nu^{(s)}}}{ d \PP_{\nu^{(s)}}}\bigg|_{s=0}[v-v_0,v-v_{0,J}](X_j)}_{(ii)(a)}\underbrace{-\delta_n \sum_{j=1}^{n}\left(\frac{ d \frac{ d}{ d s}\PP_{\nu^{(s)}}}{ d \PP_{\nu^{(s)}}}\right)^2\bigg|_{s=0}[v-v_0,v-v_{0,J}](X_j)}_{(ii)(b)},
\end{align*}
where we understand the bilinear forms through polarization and by abuse of notation $\nu^{(s)}$ denotes a generic path.

The terms $(i)(a)$ and $(ii)(a)$ are both centred. The term $(i)(b)$ is centred after subtracting
\begin{align*}
&\sqrt{n}\,t \int A_{\nu_0}(v-v_0)A_{\nu_0}(\eta) d\PP_{\nu_0}
\end{align*}
yielding the corresponding term in \eqref{123}.
The centring of the term $(ii)(b)$ is of order
\begin{align*}
&\delta_n n \bigg|\int A_{\nu_0}(v-v_0)A_{\nu_0}(v-v_{0,J}) d\PP_{\nu_0}\bigg|\\
&\lesssim \delta_n n \left(E_{\nu_0}\left[(A_{\nu_0}(v-v_0))^2\right]\right)^{1/2}\left(E_{\nu_0}\left[(A_{\nu_0}(v-v_{0,J}))^2\right]\right)^{1/2}\\
&\lesssim  \delta_n n \|v-v_0\|_{L^2(\nu_0)}\|v-v_{0,J}\|_{L^2(\nu_0)}
\lesssim \delta_n n (\eps_n^{L^2})^2=o(1).
\end{align*}

We start with the term $(i)(a)$. We define functions 
\begin{align*}
f_v&=A_{\nu_0}((v-v_0)\eta)+ A_{\nu_0}(v-v_0,\eta)
\end{align*}
and consider the corresponding class of functions as in (\ref{funcclass}).
For $u,u'\in \R^{2^J}$ we denote again $v=v(u), v'=v(u')$ and apply Lemma~\ref{sup} to the function $f_v-f_{v'}$.
This yields
\begin{align*}
\|f_{v}-f_{v'}\|_\infty 
\lesssim \|\eta\|_\infty \|v-v'\|_\infty \lesssim  \|\eta\|_\infty \|u-u'\|_\infty,
\end{align*}
where the constant only depends on $D$ and $\Delta$. We choose the envelope $F$ of the class $\F$ as a constant function $C\|\eta\|_\infty$, where the constant $C$ depends only on $D$ and $\Delta$. Then the bound \(\|f_{v}-f_{v'}\|_\infty \lesssim  \|\eta\|_\infty \|u-u'\|_\infty\) shows that we have $\sup_{\Q}N(\F,L^2(\Q),\eps\|F\|_{L^2(\Q)})\le(A/\eps)^{2^J}$ for some $A\ge2$ and for all $0<\eps<A$.

The next step is to bound $\sigma^2=\sup_{f\in\F}\PP_{\nu_0}f^2$. By Lemma~\ref{sup} we have
\begin{align*}
\sigma= \sup_{f\in\F}\|f\|_{L^2(\PP_{\nu_0})}
\lesssim \|\eta\|_\infty \eps_n^{L^2}.
\end{align*}

Corollary 3.5.8 in \cite{GineNickl2016} allows to bound the empirical process appearing in term $(i)(a)$.  For some $c>0$ we obtain
\begin{align*}
E\|\mathbb G_n\|_{\F}\lesssim \|\eta\|_\infty \: \eps_n^{L^2} 2^{J/2}\sqrt{\log \frac{c}{\eps_n^{L^2}}}
+ \frac{1}{\sqrt n} \|\eta\|_\infty 2^{J}\log \frac{c}{\eps_n^{L^2}}.
\end{align*}
The conditions for the first term dominating the second term is the same as for the term $II$. To bound the term $(i)(a)$ we use
\begin{align*}
t \|\eta\|_\infty \: \eps_n^{L^2} 2^{J/2}\sqrt{\log \frac{c}{\eps_n^{L^2}}}=o(1).
\end{align*}

Next we treat term $(i)(b)$, which is given by
\begin{align*}
\frac{t}{\sqrt{n}}\sum_{j=1}^{n} A_{\nu_0}(v-v_0)(X_j) A_{\nu_0}(\eta)(X_j).
\end{align*}
We define $g_v=A_{\nu_0}(v-v_0) A_{\nu_0}(\eta)$ and $f_v=g_v-E_{\nu_0}[g_v]$. 
So after centring the term is given by $t\mathbb G_n f_v$.
We have by Lemma~\ref{sup}
\begin{align*}
 \|g_v-g_{v'}\|_\infty = \|A_{\nu_0}(v-v')A_{\nu_0}(\eta)\|_\infty
&\le \|A_{\nu_0}(v-v')\|_\infty\|A_{\nu_0}(\eta)\|_\infty\\
&\lesssim \|v-v'\|_\infty \|\eta\|_\infty
\end{align*}
and thus also $ \|f_v-f_{v'}\|_\infty \lesssim \|v-v'\|_\infty \|\eta\|_\infty$. We consider the class of functions $\F$ as in \eqref{funcclass} corresponding to the functions of the form $f_v$ here and bound
\begin{align*}
\sigma&=\sup_{f\in\F} \|f\|_{L^2(\PP_{\nu_0})}
\le \sup_{\|v-v_0\|_{L^2}\le 2M \eps_n^{L^2}} \|g_v\|_{L^2(\PP_{\nu_0})}\\
&\le \sup_{\|v-v_0\|_{L^2}\le 2M \eps_n^{L^2}} \| A_{\nu_0}(\eta)\|_\infty \|A_{\nu_0}(v-v_0)\|_{L^2(\PP_{\nu_0})}
\lesssim  \|\eta\|_\infty \eps_n^{L^2}.
\end{align*}
Just as for term $(i)(a)$ we apply now Corollary 3.5.8 in \cite{GineNickl2016} with envelop proportional to $\|\eta\|_\infty$. So the conditions for term $(ii)(b)$ are the same as for the term $(i)(a)$.

We move on to the term $(ii)(a)$. We define
\begin{align*}
f_{v v'}&=A_{\nu_0}((v-v_0)(v'-v_{0,J}))+ A_{\nu_0}(v-v_0,v'-v_{0,J})
\end{align*}
and $f_v=f_{v v}$. We now consider the class of functions $\F$ with this definition of $f_v$.
Then we have
\begin{align*}
\|f_v-f_{v'}\|_\infty&\lesssim  \|f_{vv}-f_{vv'}\|_\infty + \|f_{vv'}-f_{v'v'}\|_\infty
\lesssim \eps_n^{L^\infty} \|v-v'\|_{\infty}.
\end{align*}
Choosing the envelope as a constant function proportional to $\eps_n^{L^\infty}$ we obtain for the covering numbers
$\sup_{\Q}N(\F,L^2(\Q),\eps\|F\|_{L^2(\Q)})\le(A/\eps)^{2^J}$. Turning to $\sigma$ we see
\begin{align*}
\sigma =\sup_{f\in\F}\|f\|_{L^2(\PP_{\nu_0})}\lesssim \eps_n^{L^\infty} \eps_n^{L^2}.
\end{align*}
Again we apply Corollary 3.5.8 in \cite{GineNickl2016}, which gives the following bound for term $(ii)(a)$
\begin{align*}
\delta_n\sqrt{n}E\|\mathbb G_n\|\lesssim \delta_n\sqrt{n} \eps_n^{L^\infty} \eps_n^{L^2} 2^{J/2}\sqrt{\log \frac{c}{\eps_n^{L^2}}}
+ \delta_n \eps_n^{L^\infty} 2^{J}\log \frac{c}{\eps_n^{L^2}}.
\end{align*}
This tends to zero by the assumption for the term $II$.

The only remaining term of $III$ is $(ii)(b)$. This term takes the from
\begin{align*}
-\delta_n \sum_{j=1}^{n}A_{\nu_0}(v-v_0)(X_j)A_{\nu_0}(v-v_{0,J})(X_j).
\end{align*}
With the definitions $g_{vv'}=A_{\nu_0}(v-v_0)A_{\nu_0}(v'-v_{0,J})$ and $f_{v}=g_{vv}-E_{\nu_0}[g_{vv}]$ the term $(ii)(b)$ can be written after centring as $-\delta_n \sqrt{n}\mathbb G_n f_{v}$ and we bound
\begin{align*}
\|g_{vv}-g_{v'v'}\|_\infty 
&\le \|g_{vv}-g_{vv'}\|_\infty + \|g_{vv'}-g_{v'v'}\|_\infty\\
&\lesssim \|v-v_0\|_\infty \|v-v'\|_\infty + \|v-v'\|_\infty \|v'-v_{0,J}\|_\infty
\lesssim \eps_n^{L^\infty} \|v-v'\|_\infty.
\end{align*}
Consequently we also have $\|f_{v}-f_{v'}\|_\infty \lesssim \eps_n^{L^\infty} \|v-v'\|_\infty$. We denote by $\F$ the class of functions corresponding to $f_v$ as in \eqref{funcclass} and further bound
\begin{align*}
\sigma&=\sup_{f\in\F} \|f\|_{L^2(\PP_{\nu_0})}
\le  \sup\{\|g_{vv}\|_{L^2(\PP_{\nu_0})}:\|v-v_0\|_{L^2}\le M \eps_n^{L^2},\|v-v_0\|_{L^\infty}\le M \eps_n^{L^\infty}\}\\
&\le  \sup\{\| A_{\nu_0}(v-v_0)\|_\infty \|A_{\nu_0}(v-v_{0,J})\|_{L^2(\PP_{\nu_0})}:\|v-v_0\|_{L^2}\le M \eps_n^{L^2},\|v-v_0\|_{L^\infty}\le M \eps_n^{L^\infty}\}\\
&\lesssim  \eps_n^{L^\infty} \eps_n^{L^2}.
\end{align*}
We see that $(ii)(b)$ leads to the same condition as the term $(ii)(a)$.

The term $IV$ equals
\begin{align*}
& \underbrace{-\frac{t^2}{2n}D^2\ell_n(\nu_0)[\eta,\eta]}_{(i)}
\underbrace{-\frac{\delta_n^2}{2}D^2\ell_n(\nu_0)[v-v_{0,J},v-v_{0,J}]}_{(ii)}
\underbrace{+\frac{t\delta_n}{\sqrt{n}}D^2\ell_n(\nu_0)[\eta,v-v_{0,J}]}_{(iii)}\\
&=\underbrace{-\frac{t^2}{2n}\sum_{j=1}^{n}\frac{ d \frac{ d^2}{ d s^2}\PP_{\nu^{(s)}}}{ d \PP_{\nu^{(s)}}}\bigg|_{s=0}[\eta,\eta](X_j)}_{(i)(a)}\underbrace{+\frac{t^2}{2n}\sum_{j=1}^{n}\bigg(\frac{ d \frac{ d}{ d s}\PP_{\nu^{(s)}}}{ d \PP_{\nu^{(s)}}}\bigg|_{s=0}[\eta](X_j)\bigg)^2}_{(i)(b)}\\
&\underbrace{-\frac{\delta_n^2}{2} \sum_{j=1}^{n}\frac{ d \frac{ d^2}{ d s^2}\PP_{\nu^{(s)}}}{ d \PP_{\nu^{(s)}}}\bigg|_{s=0}[v-v_{0,J},v-v_{0,J}](X_j)}_{(ii)(a)}\underbrace{+\frac{\delta_n^2}{2} \sum_{j=1}^{n}\bigg(\frac{ d \frac{ d}{ d s}\PP_{\nu^{(s)}}}{ d \PP_{\nu^{(s)}}}\bigg|_{s=0}[v-v_{0,J}](X_j)\bigg)^2}_{(ii)(b)}\\
&\underbrace{+\frac{t\delta_n}{\sqrt{n}}\sum_{j=1}^{n}\frac{ d \frac{ d^2}{ d s^2}\PP_{\nu^{(s)}}}{ d \PP_{\nu^{(s)}}}\bigg|_{s=0}[\eta,v-v_{0,J}](X_j)}_{(iii)(a)}\underbrace{-\frac{t\delta_n}{\sqrt{n}}\sum_{j=1}^{n}\bigg(\frac{ d \frac{ d}{ d s}\PP_{\nu^{(s)}}}{ d \PP_{\nu^{(s)}}}\bigg)^2\bigg|_{s=0}[\eta,v-v_{0,J}](X_j)}_{(iii)(b)}.
\end{align*}
The terms $(i)(a)$, $(ii)(a)$ and $(iii)(a)$ are centred. The term $(i)(b)$ can be centred by subtracting $$\frac{t^2}{2}\|A_{\nu_0}(\eta)\|_{L^2(\mathbb P_{\nu_0})}^2 $$ and gives the corresponding expression in \eqref{123}. For the centring of term $(ii)(b)$ we subtract
\begin{align*}
\frac{\delta_n^2 n}{2}\|A_{\nu_0}(v-v_{0,J})\|^2_{L^2(\mathbb P_{\nu_0})}
\lesssim \frac{\delta_n^2 n}{2}\|v-v_{0,J}\|^2_{L^2(\nu_0)}
\lesssim \delta_n^2 n (\eps_n^{L^2})^2=o(1).
\end{align*}
To centre the term $(iii)(b)$ we add $t\delta_n\sqrt{n}E_{\nu_0}[A_{\nu_0}(\eta)A_{\nu_0}(v-v_{0,J})]$ and this is bounded in absolute value by
\begin{align*}
&|t\delta_n\sqrt{n}E_{\nu_0}[A_{\nu_0}(\eta)A_{\nu_0}(v-v_{0,J})]|
\lesssim t \delta_n \sqrt{n} \|A_{\nu_0}(\eta)\|_{L^2(\PP_{\nu_0})} \|A_{\nu_0}(v-v_{0,J})\|_{L^2(\PP_{\nu_0})}\\
&\lesssim t \delta_n \sqrt{n} \|\eta\|_{L^2(\nu_0)} \|v-v_{0,J}\|_{L^2(\nu_0)}
\lesssim t \delta_n \sqrt{n} \|\eta\|_{L^2}\eps_n^{L^2}=o(1).
\end{align*}

For term $(i)(a)$ we bound using Lemma~\ref{sup}
\begin{align*}
E_{\nu_0}\left[ \left(\frac{ d \frac{ d^2}{ d s^2}\PP_{\nu^{(s)}}}{ d \PP_{\nu^{(s)}}}\bigg|_{s=0}[\eta,\eta]\right)^2\right]
&\lesssim \|A_{\nu_0}\eta^2\|^2_{L^2(\PP_{\nu_0})}+\left\|A_{\nu_0}(\eta,\eta)\right\|^2_{L^2(\PP_{\nu_0})}\\
&\lesssim \|\eta^2\|^2_{L^2(\nu_0)}+\|\eta\|^4_{L^2(\nu_0)}\lesssim \|\eta\|^4_{L^4}
\end{align*}
and for term $(i)(b)$ we bound using Lemma~\ref{sup} 
\begin{align*}
E_{\nu_0}\left[(A_{\nu_0}(\eta))^4\right]= \left\|A_{\nu_0}(\eta)\right\|^4_{L^4(\PP_{\nu_0})}\lesssim  \left\|\eta\right\|^4_{L^4}.
\end{align*}
So after centring term $(i)$ is of order $O_{\PP}(t^2 n^{-1/2}\|\eta\|^2_{L^4})$ and we use $t^2 n^{-1/2}\|\eta\|^2_{L^4}=o(1)$.

The terms $IV(ii)$ and $IV(iii)$ are treated in the same way as the terms $III(ii)$ and $III(i)$, respectively. Since the terms $IV(ii)$ and $IV(iii)$ both have an additional factor $\delta_n$, no extra condition is needed.

The remainder term can be expressed as
\begin{align*}
R_n&=\tfrac1{3!}D^3\ell_n(\nu_0)[v-v_0,v-v_0,v-v_0]-\tfrac1{3!}D^3\ell_n(\nu_0)[v_t-v_0,v_t-v_0,v_t-v_0]\\
&\quad+\tfrac1{4!}D^4\ell_n(\nu_0)[v-v_0,v-v_0,v-v_0,v-v_0]\\
&\quad-\tfrac1{4!}D^4\ell_n(\nu_0)[v_t-v_0,v_t-v_0,v_t-v_0,v_t-v_0]\\
&\quad+\tfrac1{5!}D^5\ell_n(\nu^{(s)})[v-v_0,v-v_0,v-v_0,v-v_0,v-v_0]\\
&\quad-\tfrac1{5!}D^5\ell_n(\nu_t^{(u)})[v_t-v_0,v_t-v_0,v_t-v_0,v_t-v_0,v_t-v_0]\displaybreak[0]\\
&=-\tfrac3{3!}D^3\ell_n(\nu_0)[\tilde \delta_n(v),v-v_0,v-v_0]-\tfrac3{3!}D^3\ell_n(\nu_0)[\tilde \delta_n(v),\tilde \delta_n(v),v-v_0]\\
&\quad-\tfrac1{3!}D^3\ell_n(\nu_0)[\tilde \delta_n(v),\tilde \delta_n(v),\tilde \delta_n(v)]\\
&\quad-\tfrac4{4!}D^4\ell_n(\nu_0)[\tilde \delta_n(v),v-v_0,v-v_0,v-v_0]\\
&\quad-\tfrac6{4!}D^4\ell_n(\nu_0)[\tilde \delta_n(v),\tilde \delta_n(v),v-v_0,v-v_0]\\
&\quad-\tfrac4{4!}D^4\ell_n(\nu_0)[\tilde \delta_n(v),\tilde \delta_n(v),\tilde \delta_n(v),v-v_0]\\
&\quad-\tfrac1{4!}D^4\ell_n(\nu_0)[\tilde \delta_n(v),\tilde \delta_n(v),\tilde \delta_n(v),\tilde \delta_n(v)]\\
&\quad+\tfrac1{5!}D^5\ell_n(\nu^{(s)})[v-v_0,v-v_0,v-v_0,v-v_0,v-v_0]\\
&\quad-\tfrac1{5!}D^5\ell_n(\nu_t^{(u)})[v_t-v_0,v_t-v_0,v_t-v_0,v_t-v_0,v_t-v_0].
\end{align*}

We start with the centring of the third derivatives. So the aim is to bound $E_{\nu_0}[|D^3\ell_n(\nu_0)[w_1,w_2,w_3]|]$.
\begin{align*}
& D^3\ell_n(\nu_0)[w,w,w]\\
&=\underbrace{\sum_{j=1}^{n}\frac{ d \frac{ d^3}{ d r^3}\PP_{\nu^{(r)}}}{ d \PP_{\nu^{(r)}}}\bigg|_{r=0}(X_j)}_{(a)}\underbrace{-3\sum_{j=1}^{n}\frac{ d \frac{ d^2}{ d r^2}\PP_{\nu^{(r)}}}{ d \PP_{\nu^{(r)}}}\bigg|_{r=0}(X_j)\frac{ d \frac{ d}{ d r}\PP_{\nu^{(r)}}}{ d \PP_{\nu^{(r)}}}\bigg|_{r=0}(X_j)}_{(b)}\\
&\quad \underbrace{+2\sum_{j=1}^{n}\bigg(\frac{ d \frac{ d}{ d r}\PP_{\nu^{(r)}}}{ d \PP_{\nu^{(r)}}}\bigg|_{r=0}(X_j)\bigg)^3}_{(c)}.
\end{align*}
The term $(a)$ is centred. For term $(b)$ we calculate using H\"older's inequality
\begin{align*}
&E_{\nu_0}[|(A_{\nu_0}(w_1w_2)+A_{\nu_0}(w_1,w_2))A_{\nu_0}(w_3)|]\\
&\le \|A_{\nu_0}(w_1w_2)+A_{\nu_0}(w_1,w_2)\|_{L^{3/2}(\PP_{\nu_0})}  \|A_{\nu_0}(w_3)\|_{L^{3}(\PP_{\nu_0})}\\
&\lesssim (\|w_1w_2\|_{L^{3/2}({\nu_0})}+\|w_1\|_{L^{3/2}({\nu_0})}\|w_2\|_{L^{3/2}({\nu_0})})  \|w_3\|_{L^{3}({\nu_0})}\\
&\lesssim \|w_1\|_{L^{3}}\|w_2\|_{L^{3}}\|w_3\|_{L^{3}}
\end{align*}
and for term $(c)$ we likewise obtain
\begin{align*}
E_{\nu_0}[|A_{\nu_0}(w_1)A_{\nu_0}(w_2)A_{\nu_0}(w_3)|]
\lesssim \|w_1\|_{L^{3}}\|w_2\|_{L^{3}}\|w_3\|_{L^{3}}.
\end{align*}
We conclude 
\[
E_{\nu_0}\left[|D^3\ell_n(\nu_0)[w_1,w_2,w_3]|\right]\lesssim \|w_1\|_{L^{3}}\|w_2\|_{L^{3}}\|w_3\|_{L^{3}}.
\]
Using Lemma~\ref{sup} and the generalization of H\"older's inequality $\|\prod_{j=1}^k f_j\|_{L^1(\mu)}\le\prod_{j=1}^k\|f_j\|_{L^{p_j}(\mu)}$ for $\sum_{j=1}^k\frac1{p_j}=1$ and some measure $\mu$, it follows in the same way that 
\begin{align*}
E_{\nu_0}\left[|D^4\ell_n(\nu_0)[w_1,w_2,w_3,w_4]|\right]&\lesssim \|w_1\|_{L^{4}}\|w_2\|_{L^{4}}\|w_3\|_{L^{4}}\|w_4\|_{L^{4}}.
\end{align*}
For the fifth derivative we let $\tilde\nu$ be either $\nu^{(s)}$ or $\nu_t^{(u)}$ and first apply a measure change
\begin{align*}
E_{\nu_0}\left[|D^5\ell_n(\tilde\nu)[w_1,w_2,w_3,w_4]|\right]&\lesssim E_{\tilde\nu}\left[|D^5\ell_n(\tilde\nu)[w_1,w_2,w_3,w_4]|\right]\\
&\lesssim \|w_1\|_{L^{5}}\|w_2\|_{L^{5}}\|w_3\|_{L^{5}}\|w_4\|_{L^{5}}\|w_5\|_{L^{5}}.
\end{align*}
We observe that \[\omega_n^{L^p}=\frac{t}{\sqrt{n}}\|\eta\|_{L^p}+\delta_n\eps_n^{L^p}\] is the rate at which $\tilde \delta_n(v)$ converges to zero in $L^p$. For the centring of the third, fourth and fifth derivative we use the following conditions
\begin{align*}
n\,\omega_n^{L^3}\left(\eps_n^{L^3}\right)^2&=o(1),\qquad
n\left(\omega_n^{L^3}\right)^3=o(1),\\
n\,\omega_n^{L^4}\left(\eps_n^{L^4}\right)^3&=o(1),\qquad
n\left(\omega_n^{L^4}\right)^4=o(1),\\
n\left(\eps_n^{L^5}\right)^5&=o(1),\qquad
n\left(\omega_n^{L^5}\right)^5=o(1).
\end{align*}

For the empirical process part we develop the remainder term only to the third derivative so that it takes the form
\begin{align*}
R_n=\underbrace{\tfrac1{6}D^3\ell_n(\nu^{(s')})[v-v_0,v-v_0,v-v_0]}_{(i)}\underbrace{-\tfrac1{6}D^3\ell_n(\nu_t^{(u')})[v_t-v_0,v_t-v_0,v_t-v_0]}_{(ii)}.
\end{align*}
We have $\|v-v_0\|_{L^p}\lesssim \eps_n^{L^p}$ and $\|v_t-v_0\|_{L^p}\lesssim \eps_n^{L^p}+\omega_n^{L^p}$.
Both $(i)$ and $(ii)$ can be treated jointly by bounding a term of the form \(D^3\ell_n(\tilde\nu_n)[w,w,w]\) with $\tilde \nu_n=\exp(\tilde v_n)$, $\|\tilde v_n\|_\infty\le D$, and either $w=v-v_0$ or $w=v+\tilde \delta_n-v_0$. 

Let $\nu^{(r)}=\tilde \nu_n \exp(rw)$ so that 
\begin{align*}
& D^3\ell_n(\tilde \nu_n)[w,w,w]\\
&=\underbrace{\sum_{j=1}^{n}\frac{ d \frac{ d^3}{ d r^3}\PP_{\nu^{(r)}}}{ d \PP_{\nu^{(r)}}}\bigg|_{r=0}(X_j)}_{(a)}\underbrace{-3\sum_{j=1}^{n}\frac{ d \frac{ d^2}{ d r^2}\PP_{\nu^{(r)}}}{ d \PP_{\nu^{(r)}}}\bigg|_{r=0}(X_j)\frac{ d \frac{ d}{ d r}\PP_{\nu^{(r)}}}{ d \PP_{\nu^{(r)}}}\bigg|_{r=0}(X_j)}_{(b)}\\
&\quad \underbrace{+2\sum_{j=1}^{n}\bigg(\frac{ d \frac{ d}{ d r}\PP_{\nu^{(r)}}}{ d \PP_{\nu^{(r)}}}\bigg|_{r=0}(X_j)\bigg)^3}_{(c)}.
\end{align*}

For term $(a)$ we define the functions
\begin{align*}
 g_v=A_{\tilde \nu_n}w^3+3 A_{\tilde \nu_n}(w,w^2)+ A_{\tilde \nu_n}(w,w,w).
\end{align*}
We denote $f_v=g_v- E_{\nu_0}[g_v]$. After centring the term $(a)$ is given by $\sqrt{n}\mathbbm G_n f_v$ with $f_v$ varying in the class of functions corresponding to \eqref{funcclass}, where the functions $f_v$ are defined as here.
We bound using Lemma~\ref{sup}
\begin{align*}
\|g_v-g_{v'}\|_\infty &\lesssim (\eps_n^{L^\infty}+\omega_n^{L^\infty})^2 \|v-v'\|_\infty
\text{ so that }\\
\|f_v-f_{v'}\|_\infty &\lesssim (\eps_n^{L^\infty}+\omega_n^{L^\infty})^2 \|v-v'\|_\infty.
\end{align*}
With $v=v(u)$ and $v'=v(u')$ from the definition of the prior we further bound $\|v-v'\|_\infty\lesssim \|u-u'\|_\infty$.  We take the envelope $F$ to be a constant function proportional to $\Big(\eps_n^{L^\infty}+\omega_n^{L^\infty}\Big)^2$ and obtain $\sup_{\Q}N(\F,L^2(\Q),\eps\|F\|_{L^2(\Q)})\le(A/\eps)^{2^J}$ for some $A\ge2$ and for all $0<\eps<A$.

We bound $\sigma$ by
\begin{align*}
\sigma&= \sup_{f\in\F}\|f\|_{L^2(\PP_{\nu_0})}
\le \sup_{\|v-v_{0}\|_{L^2}\le 2M \eps_n^{L^2}}\left\| g_v \right\|_{L^2(\PP_{\nu_0})}
\lesssim \sup_{\|v-v_{0}\|_{L^2}\le 2M \eps_n^{L^2}}\left\| g_v \right\|_{L^2(\PP_{\tilde \nu_n})}\\
&\lesssim \|w^3\|_{L^2(\tilde \nu_n)}+\|w^2\|_{L^2(\tilde \nu_n)}\|w\|_{L^2(\tilde \nu_n)}+\|w\|^3_{L^2(\tilde \nu_n)}\lesssim\|w\|^3_{L^6(\tilde \nu_n)} \\
&\lesssim \Big(\eps_n^{L^6}+\omega_n^{L^6}\Big)^3 \lesssim \Big(\eps_n^{L^6}\Big)^3+\Big(\omega_n^{L^6}\Big)^3 
\lesssim \Big(\eps_n^{L^\infty}+\omega_n^{L^\infty}\Big)^2  \Big(\eps_n^{L^2}+\omega_n^{L^2}\Big).
\end{align*}
Using Corollary~3.5.8 in \cite{GineNickl2016} this yields some $c>0$ such that
\begin{align}
 E\|\mathbb G_n\|_{\F}&\lesssim \Big(\eps_n^{L^\infty}+\omega_n^{L^\infty}\Big)^2  \Big(\eps_n^{L^2}+\omega_n^{L^2}\Big) 2^{J/2}\bigg(\log \frac{c}{\eps_n^{L^2}+\omega_n^{L^2}}\bigg)^{1/2}\nonumber\\
&\qquad+ \frac{1}{\sqrt n}\Big(\eps_n^{L^\infty}+\omega_n^{L^\infty}\Big)^2 2^{J}\log \frac{c}{\eps_n^{L^2}+\omega_n^{L^2}}.\label{boundrem}
\end{align}
For the term $(b)$ and $(c)$ we obtain the same bounds for the uniform covering numbers and for $\sigma$ as for term $(a)$. So the bound \eqref{boundrem} applies likewise to terms $(b)$ and $(c)$.

\subsection{Simplification of Assumption~\ref{collection}}

In this section we simplify Assumption~\ref{collection} and reduce it to a condition involving $\eta$ and $\delta_n$ only. To this end we recall $\eps_n$ from\eqref{epsn} and the $L^p$-contraction rates $\eps_n^{L^p}$ from~\eqref{contractionLp} both in Section~\ref{prelimc}. We set $2^J\approx n^{1/(2s+1)}$.

\begin{assumption}\label{simpler}
Suppose $t=O(1)$, $s>11/6$ and $\mathcal H_n \subset L^\infty(I)$.
Furthermore, assume for~$\delta_n$ and uniformly for all $\eta\in\mathcal H_n$
\begin{align}
&\delta_n n^{2/(2s+1)}(\log n)^{1+2\delta}=o(1),\label{deltacon}\\
&\|\eta\|_{L^2}=O(1),\label{etaL2}\\
&\|\eta\|_\infty n^{(-s+1)/(2s+1)} (\log n)^{1+\delta}=o(1),\label{larges}\\
&\|\eta\|_\infty n^{(-3s+11/2)/(2s+1)}(\log n)^{3+6\delta}=o(1).\label{smalls}
\end{align}
\end{assumption}

\begin{remark}\label{rem}
For $s>9/4$ (and so in particular for $s>10/4=5/2$) condition \eqref{larges} implies condition~\eqref{smalls}.
\end{remark}

\begin{lemma}\label{lemsim}
Let $2^J\approx n^{1/(2s+1)} $ and grant Assumption~\ref{simpler}. Then $t$, $\delta_n$, $\mathcal H_n$ and $\eps_n^{L^p}$ from~\eqref{contractionLp} satisfy Assumption~\ref{collection}.
\end{lemma}

\begin{proof}
The bias conditions are satisfied for this choice of $2^J$. Further we have
\begin{align*}
\sqrt{n}\delta_n\eps_n^{L^2}2^{J/2}\sqrt{\log\frac{c}{\eps_n^{L^2}}}
&\lesssim \sqrt{n}\delta_n n^{-\frac{s-1/2}{2s+1}}(\log n)^{1/2+\delta} n^{\frac{1/2}{2s+1}} \sqrt{\log n}\\
&=\delta_n n^{\frac{3/2}{2s+1}}(\log n)^{1+\delta}=o(1)
\end{align*}
by \eqref{deltacon}. Next we verify
\begin{align*}
\frac{2^{J/2}}{\sqrt{n}}\sqrt{\log \frac{c}{\eps_n^{L^2}}}
\lesssim n^{-1/2} n^{\frac{1/2}{2s+1}}\sqrt{\log n}=n^{-s/(2s+1)}(\log n)^{1/2}\lesssim \eps_n^{L^2}
\end{align*}
and 
\begin{align*}
n\delta_n(\eps_n^{L^2})^2=\delta_n n^{2/(2s+1)} (\log n)^{1+2\delta}=o(1)
\end{align*}
using \eqref{deltacon}.
For term III(i) we bound
\begin{align*}
t\|\eta\|_{\infty}\eps_n^{L^2}2^{J/2}\sqrt{\log \frac{c}{\eps_n^{L^2}}}
\lesssim \|\eta\|_{\infty}n^{(-s+1)/(2s+1)}(\log n)^{1+\delta}=o(1)
\end{align*}
by \eqref{larges}.
We check that
\begin{align*}
\frac{t^2}{\sqrt{n}}\|\eta\|_{L^4}^2
\lesssim n^{-1/2}\|\eta\|_\infty=o(1)
\end{align*}
by \eqref{larges} and that
\begin{align*}
t \delta_n \sqrt{n}\|\eta\|_{L^2} \eps^{L^2}_n
\lesssim \delta_n n^{1/(2s+1)} (\log n)^{1/2+\delta}=o(1)
\end{align*}
by~\eqref{deltacon}. For the centring of the third derivatives we bound
\begin{align*}
&n\omega_n^{L^3}\left(\eps_n^{L^3}\right)^2
\lesssim n^{1/2} \|\eta\|_\infty^{1/3} \|\eta\|_{L^2}^{2/3}\left(\eps_{n}^{L^3}\right)^{2}+n\delta_n\left(\eps_{n}^{L^3}\right)^{3}\\
&\lesssim \|\eta\|_\infty^{1/3} n^{(-s+11/6)/(2s+1)} \left(\log n\right)^{1+2\delta}
+\delta_n n^{(-s+3)/(2s+1)}(\log n)^{3/2+3\delta}
=o(1),
\end{align*}
where we used~\eqref{smalls} for the first term and~\eqref{deltacon} for the second term. Further we have
\begin{align*}
n \left(\omega_n^{L^3}\right)^3
\lesssim n\frac{t^3}{n^{3/2}}\|\eta\|_{L^3}^3+n\delta_n^3\left(\eps_n^{L^3}\right)^3\lesssim n^{-1/2}\|\eta\|_{\infty}+o(1)=o(1)
\end{align*}
using \eqref{larges} for the first term and $n\delta_n(\eps_n^{L^3})^3=o(1)$ from the next to last display for the second term.
The terms for the centering of the fourth derivates are treated by
\begin{align*}
&n\omega_n^{L^4}(\eps_n^{L^4})^3
\lesssim n \frac{t}{n^{1/2}}\|\eta\|_{L^4}(\eps_n^{L^4})^3+n\delta_n(\eps_n^{L^4})^4\\
&\lesssim n^{(-2s+11/4)/(2s+1)}(\log n)^{3/2+3\delta}\|\eta\|_\infty^{1/2}
+n^{(-2s+4)/(2s+1)}(\log n)^{2+4\delta}\delta_n=o(1), 
\end{align*}
where we used \eqref{smalls} for the first term and \eqref{deltacon} for the second term, and by
\begin{align*}
n(\omega_n^{L^4})^4 & \lesssim n \frac{t^4}{n^{2}}\|\eta\|_{L^4}^4+n\delta_n^4(\eps_n^{L^4})^4\\
& \lesssim n^{-1} \|\eta\|_{\infty}^2+o(1)=o(1),
\end{align*}
where we used \eqref{larges} for the first term and the next to last display for the second term. Turning to the centring of the fifth derivatives we observe
\begin{align*}
n(\eps_n^{L^5})^5=n^{(-3s+5)/(2s+1)}(\log n)^{5/2+5\delta}=o(1)
\end{align*}
and
\begin{align*}
n(\omega_n^{L^5})^5\lesssim n \frac{t^5}{n^{5/2}}\|\eta\|_{L^5}^5+n\delta_n^5(\eps_n^{L^5})^5
\lesssim n^{-3/2}\|\eta\|_\infty^3+o(1)=o(1)
\end{align*}
using \eqref{larges} for the first term and the next to last display for the second term.
For the remainder term $R_n$ we bound
\begin{align*}
&\sqrt{n}\Big(\eps_n^{L^\infty}+\omega_n^{L^\infty}\Big)^2  \Big(\eps_n^{L^2}+\omega_n^{L^2}\Big) 2^{J/2}\bigg(\log \frac{c}{\eps_n^{L^2}+\omega_n^{L^2}}\bigg)^{1/2}\\
&\lesssim \sqrt{n} \Big(\eps_n^{L^\infty}+\frac{t}{\sqrt{n}}\|\eta\|_\infty\Big)^2 \Big(\eps_n^{L^2}+\frac{t}{\sqrt{n}}\|\eta\|_{L^2}\Big) 2^{J/2}(\log n)^{1/2}\\
&\lesssim \sqrt{n} \Big(\big(\eps_n^{L^\infty}\big)^2+\frac{\|\eta\|_\infty^2}{n}\Big) \Big(\eps_n^{L^2}+n^{-1/2}\Big) n^{(1/2)/(2s+1)}(\log n)^{1/2}\\
&\lesssim \Big(\big(\eps_n^{L^\infty}\big)^2 \eps_n^{L^2}+\big(\eps_n^{L^\infty}\big)^2n^{-1/2}+\frac{\|\eta\|_\infty^2}{n}\eps_n^{L^2}+\frac{\|\eta\|_\infty^2}{n^{3/2}}\Big)  n^{(s+1)/(2s+1)}(\log n)^{1/2}\\
&\lesssim n^{(-2s+7/2)/(2s+1)}(\log n)^{2+3\delta} + n^{(-2s+5/2)/(2s+1)}(\log n)^{3/2+2\delta}\\
&\quad + \|\eta\|_\infty^2 n^{(-2s+1/2)/(2s+1)}(\log n)^{1+\delta} + \|\eta\|_\infty^2 n^{(-2s-1/2)/(2s+1)}(\log n)^{1/2} =o(1)
\end{align*}
using that $s>11/6$ for the first and the second term and \eqref{larges} for the third and the fourth term. Finally for the condition that the first term dominates in $R_n$ we verify
\begin{align*}
\frac{1}{\sqrt{n}} 2^{J/2}\frac{1}{\eps_n^{L^2}+\omega_n^{L^2}}\sqrt{\log \frac{c}{\eps_n^{L^2}+\omega_n^{L^2}}} &\lesssim n^{(-s-1/2)/(2s+1)}n^{(1/2)/(2s+1)} \frac{1}{\eps_n^{L^2}}\sqrt{\log \frac{c}{\eps_n^{L^2}}}\\
&\lesssim n^{(-1/2)/(2s+1)}(\log n)^{-\delta}=O(1).
\end{align*}
\end{proof}

\section{Proof of Proposition \ref{lanprop}} \label{lansec}

The Radon--Nikodym density in~\eqref{oneform} is well defined in view of the convolution series representation of $\mathbb P_\nu$ in~\eqref{convsum}. That $A_\nu$ maps $L^2(\nu)$ into $L^2(\mathbb P_\nu)$ is proved in Lemma~\ref{sup}, and an application of Fubini's theorem gives $\int_I A_\nu (h) d\mathbb P_\nu = 0$ for all $h \in L^2(\nu)$. The expansion~\eqref{lanee} follows by the same arguments used for the proof in Section~\ref{sec:likexp} but is in fact easier and no empirical process tools are needed here. In the case $v\in V_J$ for some $J$ the expansion follows directly from setting $v_0=v$ and $\eta=h$ in~\eqref{123}. For the general case we consider the path $s\mapsto\exp(v+sh/\sqrt{n})=\nu^{(s)}$ and obtain by a Taylor expansion for some~$s\in[0,1]$
\begin{align*}
&\ell_n(\nu_{h,n})-\ell_n(\nu)\\
&=D\ell_n(\nu_0)\Big[\frac{h}{\sqrt{n}}\Big]+\tfrac1{2}D^2\ell_n(\nu_0)\Big[\frac{h}{\sqrt{n}},\frac{h}{\sqrt{n}}\Big]+\tfrac1{6}D^3\ell_n(\nu^{(s)})\Big[\frac{h}{\sqrt{n}},\frac{h}{\sqrt{n}},\frac{h}{\sqrt{n}}\Big]\allowdisplaybreaks[1]\\
&=\frac{1}{\sqrt n} \sum_{i=1}^n A_{\nu}(h)(X_i) - \frac{1}{2}\|A_\nu(h)\|_{L^2(\mathbb P_\nu)}^2 
+\sum_{j=1}^{n}\frac{ d \frac{ d^2}{ d s^2}\PP_{\nu^{(s)}}}{ d \PP_{\nu^{(s)}}}\bigg|_{s=0}[h,h](X_j)\\
&\quad+\bigg(-\sum_{j=1}^{n}\bigg(\frac{ d \frac{ d}{ d s}\PP_{\nu^{(s)}}}{ d \PP_{\nu^{(s)}}}[h](X_j)\bigg)^2\bigg|_{s=0}
+\frac{1}{2}\|A_\nu(h)\|_{L^2(\mathbb P_\nu)}^2\bigg)\\
&\quad+\frac1{6n^{3/2}}D^3\ell_n(\nu^{(s)})[h,h,h]\\
&=\frac{1}{\sqrt n} \sum_{i=1}^n A_{\nu}(h)(X_i) - \frac{1}{2}\|A_\nu(h)\|_{L^2(\mathbb P_\nu)}^2 +I+II+III.
\end{align*}
The terms $I$ and $II$ are both centred and are treated exactly as the term $IV(i)(a)$ and the centred version of $IV(i)(b)$ in Section~\ref{sec:likexp}. This yields $I+II=O_{\mathbb P_{\nu}^{\mathbb N}}(n^{-1/2}\|h\|_{L^4}^2)$. The centring of term $III$ is shown to be $O_{\mathbb P_{\nu}^{\mathbb N}}(n^{-3/2}\|h\|_{L^3}^{3})$, which is proved  along the same lines as the centring of the third derivatives of the term $R_n$ in Section~\ref{sec:likexp} combined with the measure change there applied to the fifth derivatives. After centring the term $III$ is shown to be of order $O_{\mathbb P_{\nu}^{\mathbb N}}(n^{-1}\|h\|_{L^6}^{3})$ with the same bounds as used for bounding~$\sigma$ when treating the empirical process part of~$R_n$ except that here $h$ is fixed and so a simple variance bound suffices instead of the empirical process inequality used for~$R_n$. We conclude $I+II+III=o_{\mathbb P_{\nu}^{\mathbb N}}(1)$.

\section{Proof of Proposition \ref{prop:tests}}\label{sec:tests}

We define, for $L'>0$ to be chosen
\begin{align*}
 \Psi_n=\left\{
\begin{array}{ll}
 0 & \text{ if }\|\hat\nu-\nu_0\|_{\mathbb H(\delta)}< L' \eps_n\\
 1 & \text{ if }\|\hat\nu-\nu_0\|_{\mathbb H(\delta)}\ge L' \eps_n.
\end{array}
\right.
\end{align*}
Applying Lemma~\ref{concentration} with $K=n$ and $x=\sqrt{n}\eps_n$ yields, for $L'$ 
large enough,
\( E_{\nu_0}\left[\Psi_n\right] \to0\) as $n\to\infty$.
For the error of second type we obtain, for $M$ large enough depending on $L', C$ 
that, again by Lemma~~\ref{concentration},
\begin{align*}
& \sup_{\nu \in \overline{\mathcal V}: \|\nu-\nu_0\|_{\mathbb H(\delta)}\ge M\eps_n} E_{\nu}\left[1-\Psi_n\right] \\
&=\sup_{\nu \in \overline{\mathcal V}: \|\nu-\nu_0\|_{\mathbb H(\delta)}\ge M\eps_n}\PP^\mathbb N_\nu\left(\|\hat\nu - 
\nu_0\|_{\mathbb H(\delta)}< L' \eps_n\right)\allowdisplaybreaks[1]\\
&\le\sup_{\nu \in \overline{\mathcal V}: \|\nu-\nu_0\|_{\mathbb H(\delta)}\ge M\eps_n}\PP^\mathbb N_\nu\left(\|\nu_0 - 
\nu\|_{\mathbb H(\delta)}-\|\nu - 
\hat\nu\|_{\mathbb H(\delta)}< L' \eps_n\right)\allowdisplaybreaks[1]\\
&\le \sup_{\nu \in \overline{\mathcal V}}\PP^\mathbb N_\nu\left(\|\nu- \hat\nu\|_{\mathbb H(\delta)} > 
(M/2)\eps_n \right) \allowdisplaybreaks[1]\\
&\le  e^{-(C+4) n \eps_n^2}+\frac{1}{R_2} e^{-n R_2/\log n}\le 2 e^{-(C+4) n \eps_n^2},
\end{align*}
where we used $\eps_n=o( 1/\sqrt{\log n})$ and $n$ large enough in the last inequality.

\section{Proof of Proposition \ref{propball}}\label{sec:propball}

Since $v,v_0$ are bounded and thus $\exp$ is Lipschitz on the range of $v,v_0$ we have
\begin{align*}
&\PP\left(\|\nu-\nu_0\|_{L^2}\le\frac{\eps_n}{\sqrt{K_D}}\right)
\ge\PP\left(\|v-v_0\|_{\infty}\le c\eps_n\right)\\
&\ge\PP\bigg( \sum_l 2^{l/2}\max_k |\beta_{lk}-2^{-l}(l^2+1)^{-1}u_{lk}|<c'\eps_n \bigg),
\end{align*}
where $u_{lk}=0$ for $l\ge J$ and $\beta_{lk}=\langle v_0,\psi_{lk}\rangle$. We define $b_{lk}=2^l(l^2+1)\beta_{lk}$ such that $|b_{lk}|\le B$, and $M(J)=\sum_{l=-1}^{J-1}\sum_{k=0}^{(2^l-1)\vee0}1=2^J$. We can bound the last probability from below by
\begin{align*}
&\PP\bigg( \sum_{l\le J-1}2^{-l/2}(l^2+1)^{-1}\max_k|b_{lk}-u_{lk}|<c'\eps_n-\bar c 2^{-J_n s}/(J_n^2+1) \bigg)\\
&\ge\PP\Big( \max_{l\le J-1}\max_{k} |b_{lk}-u_{lk}|<c''\eps_n \Big)=\prod_{l\le J-1}\prod_{k}\PP\left( |b_{lk}-u_{lk}|<c''\eps_n \right) \\
& \ge \Big(\frac{c''\eps_n}{2B}\Big)^{M(J)}\ge e^{-Cn\eps_n^2}
\end{align*}
 for $n$ large enough and for some constant $C>0$.

\section{Proof of Lemma \ref{tedious}}\label{sec:tedious}

a) Write $B$ for the unit ball of the space $\mathbb B=\mathbb B(\delta)$ which can be shown to be closed under  pointwise multiplication in the sense that $\|f g\|_\mathbb B \le c_0 \|f\|_\mathbb B \|g\|_{\mathbb B}$. Since $\nu_0^{-1} \in \mathbb B, \|\nu-\nu_0\|_{\mathbb B} \to 0$  we also have $\|(\nu-\nu_0)/\nu_0\|_\infty \lesssim \|(\nu-\nu_0)/\nu_0\|_\mathbb B \to 0$ and thus $\|[(\nu-\nu_0)/\nu_0]^k\|_\mathbb B \le c_0^k \|(\nu-\nu_0)/\nu_0\|_\mathbb B^k$. Since eventually $\|(\nu-\nu_0)/\nu_0\|_\mathbb B < 1/(2c_0)$ we deduce that the series $$g = \sum_k \frac{(-1)^k}{k} \Big(\frac{\nu-\nu_0}{\nu_0}\Big)^{k-1}$$ converges absolutely uniformly and in $\mathbb B$ and has $\|\cdot\|_\mathbb B$-norm less than a constant multiple of $\|\nu-\nu_0\|_\mathbb B$. Thus, using again the multiplication property of the norm

\begin{align*}
&\|\log \nu - \log \nu_0 \|_{\mathbb H(\delta)} = \sup_{f \in B} \left|\int f \log \Big( 1+\frac{\nu-\nu_0}{\nu_0} \Big)  \right| \allowdisplaybreaks[0]\\ 
&\qquad= \sup_{f \in B} \left|\int (\nu-\nu_0)  \sum_{k=1}^\infty \frac{(-1)^k}{k} \frac{(\nu-\nu_0)^{k-1}}{\nu_0^{k-1}} \frac{f}{\nu_0}    \right|\allowdisplaybreaks[0]\\
&\qquad= \sup_{f \in B} \left|\int (\nu-\nu_0)  g \frac{f}{\nu_0} \right|  \le \sup_{h \in c_1B} \left|\int h (\nu-\nu_0) \right| = c_1 \|\nu-\nu_0\|_{\mathbb H(\delta)}. 
\end{align*}

b) For any $j$ we have, using the Cauchy--Schwarz inequality,
\begin{align*}
\|\nu-\nu_0\|^2_{\mathbb B(\delta)} & \lesssim \sum_{l \le j} 2^l l^{2\delta} \sum_k|\langle \nu-\nu_0, \psi_{lk} \rangle|^2 + j^{2\delta-2\delta'} \sum_{l>j} 2^{l} l^{2\delta'} \sum_k |\langle \nu-\nu_0, \psi_{lk} \rangle|^2 \\
& \leq 2^{2j} j^{4\delta} \sum_{l \le j} 2^{-l} l^{-2\delta} \sum_k |\langle \nu-\nu_0, \psi_{lk} \rangle|^2 + j^{2\delta-2\delta'} \|\nu -\nu_0\|_{ B^{1/2,\delta'}_{2 2}} \\
& \lesssim 2^{2j} j^{4\delta} \|\nu-\nu_0\|_{\mathbb H(\delta)} + j^{-2(\delta'-\delta)}.
\end{align*}
Using  $\|\nu-\nu_0\|_{\mathbb H(\delta)} = o(1)$ and letting $j\to \infty$ slowly enough we deduce $\|\nu-\nu_0\|_{\mathbb B(\delta)} \to 0$.

\bibliographystyle{plain}

\bibliography{bib}

\end{document}